\documentclass[a4paper,11pt]{amsart}

\usepackage{mathrsfs}
\usepackage[all]{xy}
\usepackage{here}
\usepackage{color}
\usepackage{longtable}
\usepackage{multirow}
\usepackage[driverfallback=dvipdfm]{hyperref}
\usepackage{geometry}
\newtheorem{thm}{Theorem}[section]
\newtheorem{lem}[thm]{Lemma}

\newtheorem{prop}[thm]{Proposition}
\newtheorem{prob}[thm]{Problem}

\theoremstyle{definition}
\newtheorem{defin}[thm]{Definition}
\newtheorem{eg}[thm]{Example}

\newtheorem*{Notation}{Notation and Conventions}
\newtheorem*{Acknowledgement}{Acknowledgement}

\theoremstyle{remark}
\newtheorem{rem}[thm]{Remark}

\numberwithin{equation}{section}

\newcommand{\bA}{\mathbb{A}}
\newcommand{\bC}{\mathbb{C}}

\newcommand{\bF}{\mathbb{F}}

\newcommand{\kc}{\overline{k}}
\newcommand{\sL}{\mathscr{L}}
\newcommand{\wL}{\widetilde{L}}
\newcommand{\wsL}{\widetilde{\mathscr{L}}}

\newcommand{\sO}{\mathscr{O}}
\newcommand{\bP}{\mathbb{P}}

\newcommand{\wtp}{\widetilde{p}}
\newcommand{\bQ}{\mathbb{Q}}
\newcommand{\bR}{\mathbb{R}}

\newcommand{\wS}{\widetilde{S}}

\newcommand{\wU}{\widetilde{U}}

\newcommand{\bZ}{\mathbb{Z}}

\newcommand{\Bs}{\mathrm{Bs}}

\newcommand{\Gal}{\mathrm{Gal}(\overline{k}/k)}

\newcommand{\Pic}{\mathrm{Pic}}

\newcommand{\Proj}{\mathrm{Proj}}

\newcommand{\Spec}{\mathrm{Spec}}
\newcommand{\Supp}{\mathrm{Supp}}
\newcommand{\mn}{\lceil \frac{n}{2} \rceil}

\begin{document}

\title{Cylinders in canonical del Pezzo fibrations}


\author{}
\address{}
\curraddr{}
\email{}
\thanks{}

\author{Masatomo Sawahara}
\address{Graduate School of Science and Engineering, Saitama University, Shimo-Okubo 255, Sakura-ku Saitama-shi,  Saitama 338-8570, JAPAN}
\curraddr{}
\email{sawahara.masatomo@gmail.com}
\thanks{}



\date{}

\dedicatory{}

\begin{abstract}
Cylinders in projective varieties play an important role in connection with unipotent group actions on certain affine algebraic varieties. The previous work due to Dubouloz and Kishimoto deals with the condition for a del Pezzo fibration to contain a vertical cylinder. In the present work, as a generalization in the sense of singularities, we shall determine the condition under which a del Pezzo fibration with canonical singularities admits a vertical cylinder by means of degree and type of singularities found on the corresponding to the generic fiber. 
\end{abstract}

\maketitle

In this article, let $k$ be a field of characteristic zero (not necessarily an algebraically closed field) and let $\kc$ be the algebraic closure of $k$. 
\section{Introduction}\label{1}

Let $X$ be an algebraic variety over $k$. 
Then an open subset $U$ of $X$ is called a {\it cylinder} if $U$ is isomorphic to $\bA ^1_k \times Z$ for some variety $Z$ over $k$. 
Certainly, cylinders are geometrically very simple, however recently they begin to receive a lot of attention in connection with unipotent group actions on affine cones over polarized varieties (see, e.g., {\cite{CPPZ21, CPW16a,KPZ11,KPZ13,KPZ14a,KPZ14b}}). 
Thus, it is important to find a cylinder in projective varieties, but in general, it is not easy to decide whether a given projective variety $V$ contains a cylinder. 
Supposing that $V$ contains a cylinder, a resolution of singularities of $V$ still contains a cylinder, in particular, its canonical divisor is not pseudo-effective. 
Then by virtue of {\cite[Corollary 1.3.3]{BCHM10}}, $V$ is birational to a suitable Mori Fiber Space (MFS, for short) by means of the minimal model program (MMP, for short). 
Conversely, assuming that a normal projective variety $V$ admits a process of MMP $V \dasharrow X$, where $X$ is MFS which contains a cylinder, it follows that so does the initial $V$ by {\cite[Lemma 9]{DK19}}. Thus, in some sense, it is important and essential to try to find cylinders contained in MFS. 

Let $f: X \to Y$ be MFS. In case of $\dim (Y)=0$, namely, $X$ is a $\bQ$-Fano variety of Picard rank one, it is a delicate question to know whether $X$ possesses a cylinder or not (see {\cite{Fur93a,Fur93b,KPZ11,KPZ14a,PZ16,PZ17,PZ18}}). 
In order to deal with cylinders found in MFS with base variety $Y$ of positive dimension, it is useful to prepare the notion of {\it vertical cylinders}: 
\begin{defin}[{\cite{DK18}}]\label{vertical} 
Let $f: X \to Y$ be a dominant morphism between normal algebraic varieties defined over $k$ and let $U \simeq Z \times \bA^1_k$ be a cylinder on $X$. 
We say that $U$ is a {\it vertical cylinder with respect to} $f$ if there exists a morphism $g : Z \to Y$ such that the restriction of $f$ to $U$ coincides with $g \circ pr_Z: U \simeq \bA ^1_k \times Z \overset{pr_Z}\to Z \overset{g}\to Y$. 
\end{defin} 
We note that the existence of a vertical cylinder with respect to $f$ is equivalent to saying that the generic fiber of $f$ contains a cylinder defined over the field of function over the base variety (cf. \cite[Lemma 3]{DK18}). 
Let $f: X \to Y$ be MFS of $\dim (Y)=\dim (X)-1$, i.e., Mori conic bundle. 
Then it is not difficult to see that the existence of vertical cylinders with respect to $f$ results in that of a rational point in the generic fiber of $f$. 
In this article, we are mainly interested in the case of $\dim (Y)=\dim (X)-2$, i.e., $f$ is a {\it del Pezzo fibration}. 
Since general fibers are del Pezzo surfaces, they contain cylinders. 
Hence to some extent it would be reasonable to expect that cylinders found on general fibers of $f$ could be unified to yield a vertical cylinder in the total space $X$. 
However this is too optimistic, indeed if $X$ has only terminal singularities, then the total space $X$ of a del Pezzo fibration $f: X \to Y$ admits a vertical cylinder if and only if the generic fiber $X_{\eta}$ of $f$ admits a rational point and is of degree greater than or equal to $5$ (see {\cite[Theorem 1]{DK18}}). 
Our next target lies in del Pezzo fibrations $f: X \to Y$ whose total spaces possess canonical singularities. 
\begin{defin}\label{def} 
A dominant projective morphism $f: X \to Y$ of relative dimension two between complex algebraic varieties is called a {\it canonical del Pezzo fibration} if the generic fiber $X_{\eta}$ of $f$ is a Du Val del Pezzo surface of Picard rank one over the function field $k= \bC (Y)$ of $Y$, i.e., the base extension $X_{\eta, \kc}$ to the algebraic closure $\kc$ of $k$ has at most Du Val singularities. 
\end{defin} 
In consideration of {\cite[Lemma 3]{DK18}}, the existence of a vertical cylinder contained in a given canonical del Pezzo fibration consists in a cylinder defined over $\bC (Y)$ in the generic fiber $X_{\eta}$. 
The field $\bC (Y)$ being not algebraically closed, the essence lies in the following problem: 
\begin{prob}\label{prob}
Let $S$ be a Du Val del Pezzo surface of $\rho _k(S)=1$ over $k$. 
In which case does $S$ contain a cylinder ?
\end{prob}
See \cite{CPW16b,Bel17} for relevant results about Problem \ref{prob} in case where $k$ is algebraically closed. Our main results in this article, which are stated in what follows depending on the degree $d$ of Du Val del Pezzo surfaces, shall reply completely to Problem \ref{prob}. The meanings of the notation used in Theorems \ref{main(1-2)} and  \ref{main(1-3)} are defined later in Definitions \ref{def:duval} and \ref{Urabe}: 
\begin{thm}\label{main(1-1)}
Any Du Val del Pezzo surface $S$ over $k$ with Picard rank $\rho_k(S)=1$ and of degree greater than or equal to $5$ contains a cylinder. 
\end{thm}
\begin{thm}\label{main(1-2)}
Let $S$ be a Du Val del Pezzo surface over $k$ with $\rho_k (S)=1$ and of degree $3$ or $4$. 
Then $S$ contains a cylinder if and only if $S$ has a singular point, which is $k$-rational, such that it is not of type $A_1^{++}$. 
\end{thm}
\begin{thm}\label{main(1-3)}
Let $S$ be a Du Val del Pezzo surface over $k$ with $\rho_k (S)=1$ and of degree $d \le 2$. 
Then we have the following: 
\begin{enumerate}
\item If $d=2$ (resp. $d=1$) and $S_{\kc}$ has a singular point of type $A_6$, $A_7$, $D_4$, $D_5$, $D_6$, $E_6$ or $E_7$ (resp. type $A_8$, $D_6$, $D_7$, $D_8$, $E_7$ or $E_8$), then $S$ contains a cylinder. 
\item If $d=2$ (resp. $d=1$) and $S_{\kc}$ has a singular point of type $(A_5)''$ (resp. type $(A_7)''$)\footnote{Note that a singular point of type $(A_5)''$ (resp. $(A_7)''$) on a Du Val de Pezzo surface of degree $2$ (resp. $1$) is automatically $k$-rational. }, then $S$ contains a cylinder if and only if this singular point is not of type $A_5^{++}$ (resp. type $A_7^{++}$) on $S$. 
\item If $d=2$ (resp. $d=1$) and $S_{\kc}$ allows only singular points of type $A_1$ (resp. $A_1$, $A_2$, $A_3$, $D_4$), then $S$ does not contain a cylinder. 
\item If $S$ does not satisfy any condition on singularities of (1), (2) and (3) above, then $S$ contains a cylinder if and only if $S$ has a singular point, which is $k$-rational, of type $A^-_n$, $D^-_n$ or $E^-_n$. 
\end{enumerate}
\end{thm}
Since Problem \ref{prob} is completely settled by virtue of Theorems \ref{main(1-1)}, \ref{main(1-2)} and \ref{main(1-3)}, we can determine canonical del Pezzo fibrations $f: X \to Y$ admitting vertical cylinders depending on degree of $f$ and singularities in the generic fiber $X_{\eta}$ of $f$. 
However, the treatment of singularities on generic fibers is quite subtle, for instance, even in the case that two Du Val del Pezzo surfaces are of Picard rank one over a field $k$ whose base extensions over $\kc$ are mutually isomorphic, it may be that exactly only one of them contains a cylinder (see Example \ref{eg6}). 

The scheme of the article proceeds as follows: 
In \S \ref{2}, we shall prepare preliminaries, which are used in the later sections. 
In particular, the notation of Du Val singularities over a field $k$ of characteristic zero plays important roles, which are used in Theorems \ref{main(1-2)} and \ref{main(1-3)} in the article. 
We shall prove the main results in \S \ref{3} and \S \ref{5}. 
More precisely, Theorems \ref{main(1-1)} and \ref{main(1-2)} are proved in \S \ref{3} and Theorem \ref{main(1-3)} is proved in \S \ref{5}, respectively. 
Since the proof for Theorem \ref{main(1-3)} is somehow involved, we need to prepare a further result in \S \ref{4} in addition to those in \S \ref{2}. 
In the last section \S \ref{6}, we shall provide some examples of cylinders in Du Val del Pezzo surface of Picard rank one over a field $k$ of characteristic zero and those in canonical del Pezzo fibrations. 
\begin{Notation}
We will use the following notations: 
\begin{itemize}
\item $\bC$: The field of complex numbers. 
\item $\bR$: The field of real numbers. 
\item $\bQ$: The field of rational numbers. 
\item $k$: A field of characteristic zero. 
\item $\kc$: An algebraic closure of $k$, i.e., an algebraically closed field of characteristic zero. 
\item $\bF _n$: The Hirzebruch surface of degree $n$, i.e., $\bF _n := \bP (\sO _{\bP ^1} \oplus \sO _{\bP ^1}(n))$, where $n$ is a non-negative integer. 
\item $V_{\kc}$: The base extension over $\kc$ of an algebraic variety $V$ over $k$, i.e., $V_{\kc} := V \times _{\Spec (k)} \Spec (\kc )$. 
\item $\bA ^1 _{\ast ,k}$: The affine line over $k$ with one $k$-rational point removed, i.e., $\bA ^1_{\ast ,k} := \Spec (k[t^{\pm 1}])$. 
\item $C_{(n)}$: A $k$-form of the affine line $\bA ^1_{\kc}$ with $n$-times closed points removed. 
\item $\rho _k(V)$: The Picard rank of a projective variety $V$ over $k$. 
\item $(D \cdot D')$: The intersection number of two divisors $D$ and $D'$.  
\item $(D)^2$: The self-intersection number of a divisor $D$, i.e., $(D)^2 = (D \cdot D)$. 
\item $n$-curve: A smooth rational curve with self-intersection number $n$. 
\item $\Bs (\sL )$: A (set-theoretic sense) base locus of a linear system $\sL$. 
\item $\delta _{i,j}$: The Kronecker delta. 
\item $\lceil q \rceil$: The round up of a rational number $q$. 
\item Weak del Pezzo surface: A smooth projective surface such that its anti-canonical divisor is nef and big but not ample. 
\item $\bP ^1$-fibration (resp. $\bP ^1$-bundle) $\pi :V \to B$ (over $k$): A surjective morphism over $k$ such that for a general closed point (resp. any closed point) $b \in B$ the fiber $V_b$ over the residue field $k(b)$ is isomorphic to $\bP ^1_{k(b)}$. 
\item Mori conic bundle $\pi :V \to B$ (over $k$): A surjective morphism over $k$ such that for any closed point $b \in B$ the fiber $V_b$ over the residue field $k(b)$ is isomorphic to the plane conic, which is not necessarily irreducible. 
\end{itemize}
\end{Notation}
\begin{Acknowledgement}
The author is deeply grateful to his supervisor Professor Takashi Kishimoto for his useful advice. 
Also, he would like to thank the referees for suggesting many valuable comments that helped to improve this article. 
\end{Acknowledgement}

\section{Preliminaries}\label{2}
\subsection{Du Val singularities}\label{2-1}
In this subsection, we quickly review Du Val singularities over algebraically closed fields, and then define Du Val singularities over algebraically non-closed fields. 

The properties of Du Val singularities over an algebraically closed field are well known (see, e.g., {\cite{Dur79}}, for details). 
In particular, we recall that a Du Val singular point on a normal algebraic surface over $\bC$ is analytically equivalent to one of the following:  
\begin{align*}
A_n:\ &(x^2+y^2+z^{n+1}=0) \subseteq \bA ^3_{\bC} = \Spec (\bC [x,y,z])& &(n \ge 1);& \\
D_n:\ &(x^2+y^2z+z^{n-1}=0) \subseteq \bA ^3_{\bC} = \Spec (\bC [x,y,z])& &(n \ge 4);& \\
E_6:\ &(x^2+y^3+z^4=0) \subseteq \bA ^3_{\bC} = \Spec (\bC [x,y,z]);&&& \\
E_7:\ &(x^2+y^3+yz^3=0) \subseteq \bA ^3_{\bC} = \Spec (\bC [x,y,z]);&&& \\
E_8:\ &(x^2+y^3+z^5=0) \subseteq \bA ^3_{\bC} = \Spec (\bC [x,y,z]),&&&
\end{align*}
where we note that Du Val singularity types of $A_n$, $D_n$, $E_6$, $E_7$ and $E_8$ correspond to the Dynkin diagrams of types $A_n$, $D_n$, $E_6$, $E_7$ and $E_8$, respectively. 

From now on, we shall consider Du Val singularities over algebraically non-closed fields. 
Let $V$ be a normal algebraic surface over $k$ and let $p$ be a Du Val singular point on $V_{\kc}$, which is $k$-rational. 
Notice that the exceptional set of the minimal resolution at $p \in V_{\kc}$ is invariant under the action of the Galois group $\Gal$. 
Thus, depending on a fashion of the $\Gal$-action on the exceptional set, we shall divide the type of Du Val singularities in a more refined way as follows:
\begin{defin}\label{def:duval} 
Let $V$ be a normal algebraic surface over $k$, let $p$ be a Du Val singular point on $V_{\kc}$, which is $k$-rational, let $\sigma :\widetilde{V} \to V$ be the minimal resolution of $p$ over $k$ and let $N$ be the exceptional set of $\sigma$ on $\widetilde{V}$. 
\begin{enumerate}
\renewcommand{\labelenumii}{(\roman{enumii})}
\item Assume that the type of $p$ is an type $A_1$ on $V_{\kc}$. Then: 
\begin{enumerate}
\item $p$ is of {\it type $A_1^+$} on $V$ if $N(k) \not= \emptyset$. 
\item $p$ is of {\it type $A_1^{++}$} on $V$ if $N(k) = \emptyset$. 
\end{enumerate}
\item Assume that $p$ is an type $A_n$ for $n\ge 2$ on $V_{\kc}$. Then: 
\begin{enumerate}
\item $p$ is of {\it type $A_n^-$} on $V$ if $\rho _k(\widetilde{V}) - \rho _k(V) = n$. 
\item $p$ is of {\it type $A_n^+$} on $V$ if $\rho _k(\widetilde{V}) - \rho _k(V) < n$ and $N(k) \not= \emptyset$. 
\item $p$ is of {\it type $A_n^{++}$} on $V$ if $\rho _k(\widetilde{V}) - \rho _k(V) < n$ and $N(k) = \emptyset$. 
\end{enumerate}
\item Assume that $p$ is an type $X_n$ on $V_{\kc}$, where $X_n$ means $D_n$ for $n\ge 4$ or $E_n$ for $n=6$. Then: 
\begin{enumerate}
\item $p$ is of {\it type $X_n^-$} on $V$ if $\rho _k(\widetilde{V}) - \rho _k(V) = n$. 
\item $p$ is of {\it type $X_n^+$} on $V$ if $\rho _k(\widetilde{V}) - \rho _k(V) < n$. 
\end{enumerate}
\end{enumerate}
\end{defin}
\begin{rem}
If $k=\bR$, then all types of Du Val singularities over $k$ correspond to all types of {\it real Du Val singularities} in {\cite[\S 9]{Kol99}} except for type $A_1$. 
Meanwhile, although \cite{Kol99} defines both of Du Val singularities of type $A_1^+$ and type $A_1^-$, whereas in Definition \ref{def:duval}, we do not prepare the notation for type $A_1^-$ intentionally in consideration of the assertion (4) in Theorem \ref{main(1-3)}.  
\end{rem}
\begin{eg}
Let $\widetilde{V}$ be a $k$-form of the Hirzebruch surface $\bF _2$ of degree two over $k$. 
Notice that the minimal section $M$ on $\widetilde{V}_{\kc}$, which is a $(-2)$-curve, is defined over $k$. 
Hence, we obtain the contraction $\sigma :\widetilde{V} \to V$ of $M$ defined over $k$. 
Then $p := \sigma (M) \in V_{\kc}$ is a Du Val singular point of type $A_1$. 
Now, if $M$ is a non-trivial $k$-form of $\bP ^1_{\kc}$, then $M(k) = \emptyset$ (see Lemma \ref{Severi-Brauer} below). 
Namely, if $M \simeq \bP ^1_k$ (resp. $M \not\simeq \bP ^1_k$), then $p \in S$ is of type $A_1^+$ (resp. $A_1^{++}$). 
\end{eg}

\subsection{Types of weak del Pezzo surfaces}\label{2-2}
For any Du Val del Pezzo surface $S$ defined over $k$, recall that its minimal resolution $\wS$ is a weak del Pezzo surface over $k$ and satisfies $(-K_S)^2=(-K_{\wS})^2$, which is called the {\it degree} of $S$ (or $\wS$). 
Conversely, for any weak del Pezzo surface $\wS$ defined over $k$, note that the union of all $(-2)$-curves on $\wS _{\kc}$ is defined over $k$ and can be contracted, hence, we obtain the contraction $\sigma :\wS \to S$ over $k$ of this union, so that $S$ is a Du Val del Pezzo surface over $k$. 
Hence, Du Val del Pezzo surfaces over $k$ are in one-to-one correspondence with weak del Pezzo surfaces over $k$ via minimal resolutions. 
In this subsection, we shall recall a classification of the types of weak del Pezzo surfaces over a field of characteristic zero for later use. 
Moreover, we will introduce the notation for a special kind of singularities included in Du Val del Pezzo surfaces from a different point of view than in \S \S \ref{2-1}. 
\subsubsection{}\label{2-2-1}
Let $\wS$ be a weak del Pezzo surface defined over a field $k$ of characteristic zero. 
Then we associate to $\wS$ ``{\it Degree}'', which is the degree of $\wS$, ``{\it Singularities}'', which means singularities on the surface given by the blow-down of all $(-2)$-curves on $\wS _{\kc}$, and ``{\it \#\,Lines}'', which is the number of all $(-1)$-curves on $\wS _{\kc}$. 
The triplet (Degree, Singularities, \#\,Lines) is called the {\it type} of $\wS$ (see {\cite[\S\S 2.3]{Saw}}). 
The type of weak del Pezzo surfaces is classified, see for instance {\cite{BW79,CT88,Ura83}}, {\cite[\S \S 8.4--8.7 and \S \S 9.2]{Dol12}}, and {\cite[Table 3]{Saw}} for the list in the classification. 

In what follows, we follow the notation on the type of weak del Pezzo surfaces as in {\cite[\S \S 2.3]{Saw}}. 
For the readers' convenience, we give some comments on this notation. 
Almost all types of weak del Pezzo surfaces are determined by only ``Degree'' and ``Singularities'', namely the remaining information on ``\#\,Lines'' would be uniquely determined by the former two. 
In each of the other exceptional types, there are exactly two possibilities about ``\#\,Lines''. 
In order to distinguish such cases, we use the notation something like $(X)_<$ and $(X)_>$, where $X$ means ``Singularities'' and \#\,Lines of $(X)_<$ is less than \#\,Lines of $(X)_>$. 
For example, in case of $(\text{Degree}, \text{Singularities}) = (6,A_1)$, then ``\#\,Lines'' is either $3$ or $4$. 
Hence, if $\text{\#\,Lines}=3$ (resp. $\text{\#\,Lines}=4$), then this type is denoted by $(A_1)_<$ (resp. $(A_1)_>$). 
\subsubsection{}
Let $S$ be a Du Val del Pezzo surface of degree $d \le 2$ defined over $k$ and let $\sigma :\wS \to S$ be the minimal resolution, so that $\wS$ is a weak del Pezzo surface of degree $d$. 

By the classification of types of weak del Pezzo surfaces, assuming that $S_{\kc}$ admits at least one singular point of type $A_{9-2d}$ or at least two singular points, one of which is of type $A_{7-2d}$ and the other of which is of type $A_1$, the type of the weak del Pezzo surface $\wS$ is not uniquely determined by only ``Degree'' and ``Singularities'' if and only if ``Singularities'' of $\wS$ is one of the following: 
\begin{align}\label{(2.1)}
\begin{cases}
\ d=2:&A_5+A_1,\ A_5,\ A_3+2A_1\ \text{or}\ A_3+A_1.\\
\ d=1:&A_7\ \text{or}\ A_5+A_1.
\end{cases}
\end{align}
For the above mentioned cases (\ref{(2.1)}) only, we shall adopt the notation found in {\cite[\S \S 2.2]{CPW16b}} as follows to make the proof more transparent. We note that in (\ref{(2.2)}) each of the left hand side is the notation used in {\cite[\S \S 2.2]{CPW16b}}, meanwhile, each of the right hand side is the one defined in \ref{2-2-1}. For types in (\ref{(2.1)}), we will adopt the ones at the left hand side in (\ref{(2.2)}): 
\begin{align}\label{(2.2)}
\begin{cases}
\ d=2:&(A_5+A_1)'=(A_5+A_1)_<,\ (A_5+A_1)''=(A_5+A_1)_>,\ (A_5)'=(A_5)_<,\\
& (A_5)''=(A_5)_>,\ (A_3+2A_1)'=(A_3+2A_1)_<,\  (A_3+2A_1)''=(A_3+2A_1)_>,\\ 
&(A_3+A_1)'=(A_3+A_1)_<,\ (A_3+A_1)''=(A_3+A_1)_>. \\
\ d=1:&(A_7)'=(A_7)_>,\ (A_7)''=(A_7)_<,\ (A_5+A_1)'=(A_5+A_1)_>,\\
&(A_5+A_1)''=(A_5+A_1)_<. 
\end{cases}
\end{align}

On the other hand, to state our main result exactly, we shall divide the types of $k$-rational Du Val singularity $x \in S_{\kc}$ of type $A_{9-2d}$ as follows by making use of their notation: 
\begin{defin}[cf.\ {\cite{Ura83}}]\label{Urabe}
With the notation as above, the singular point $x$ is of type $(A_{9-2d})'$ (resp. type $(A_{9-2d})''$) if there exists a $(-1)$-curve on $\wS_{\kc}$ meeting the $(-2)$-curve corresponding to the central vertex on the dual graph of the minimal resolution (resp. there does not exist such a $(-1)$-curve on $\wS _{\kc}$). 
\end{defin}
\begin{rem}
If $S_{\kc}$ admits a singular point $x$ of type $A_{9-2d}$, then $\wS$ is of one of the following types: 
\begin{itemize}
\item $d=2$: $A_5+A_2$, $A_5+A_1$, $A_5$. 
\item $d=1$: $A_7+A_1$, $A_7$. 
\end{itemize}
Moreover, if $\wS$ is of $A_{9-2d}+A_d$-type, then the singular point $x$ is of type $(A_{9-2d})'$. 
\end{rem}

\subsection{Properties of weak del Pezzo surfaces}\label{2-3}
In this subsection, we shall recall some properties about mainly weak del Pezzo surfaces defined over algebraically non-closed fields. 
\begin{lem}\label{Severi-Brauer}
Let $V$ be a smooth algebraic variety over $k$ satisfying $V_{\kc} \simeq \bP ^n_{\kc}$. 
If $V$ has a $k$-rational point, then $V \simeq \bP ^n_k$. 
\end{lem}
\begin{proof}
See {\cite[Proposition 4.5.10]{Poo17}}. 
\end{proof}
In this article, we will treat $k$-minimal weak del Pezzo surfaces over $k$. 
Here, a smooth projective surface $W$ defined over $k$ is {\it $k$-minimal} if any birational morphism $W \to W'$ to a smooth projective surface $W'$ over $k$ is an isomorphism. 
Recently, the author studied $k$-minimal weak del Pezzo surfaces ({\cite{Saw}}). 
In particular, we will use the following result later: 
\begin{prop}\label{minimal}
Let $W$ be a weak del Pezzo surface of degree $d$ defined over $k$ satisfying one of the following conditions: 
\begin{itemize}
\item $d=4$ and $W$ is of $(2A_1)_<$-type. 
\item $d=2$ and $W$ is of $A_2$-type. 
\end{itemize}
Then the following assertion hold: 
\begin{enumerate}
\item $W_{\kc}$ is endowed with a structure of Mori conic bundle $\pi :W_{\kc} \to \bP ^1_{\kc}$ over $\kc$ with exactly $(8-d)$-times of singular fibers such that each $(-2)$-curve on $W_{\kc}$ is a section of $\pi$, where each singular fibers of $\pi$ consists of exactly two $(-1)$-curves meeting transversely. 
\item $W$ is $k$-minimal if and only if $\rho _k(W)=2$. 
\item If $W$ is $k$-minimal, then $W$ does not contain any cylinder. 
\end{enumerate}
\end{prop}
\begin{proof}
See {\cite{Saw}}. In particular, the assertion (1) follows from {\cite[Claim 3.8 and Remark 3.9]{Saw}}. 
\end{proof}
The following two lemmas are basic but will play important roles in \S \ref{4} and \S \ref{5}: 
\begin{lem}\label{(-1)-curve}
Let $W$ be a weak del Pezzo surface over $\kc$ and let $D$ be a divisor on $W$. 
If $(D )^2=-1$, $(D \cdot -K_W)=1$ and $(D \cdot M) \ge 0$ for any $(-2)$-curve $M$ on $W$, then there exists a $(-1)$-curve $E$ on $W$ such that $E \sim D$. 
\end{lem}
\begin{proof}
See {\cite[Lemma 8.2.22]{Dol12}}. 
\end{proof}
\begin{lem}\label{(-2)-curve}
Let $W$ be a weak del Pezzo surface over $\kc$ of degree $d$. 
Then the number of $(-2)$-curves on $W$ is less than or equal to $9-d$. 
\end{lem}
\begin{proof}
See {\cite[Proposition 8.2.25]{Dol12}}. 
\end{proof}
Moreover, we also prepare a {\it variant of Corti's inequality}: 
\begin{prop}\label{Corti}
Let $V$ be a smooth algebraic surface over $\kc$, let $\Delta _1$ and $\Delta _2$ be two curves on $V$, which are meeting transversely at a point, say $p$, let $\sL$ be a mobile linear system on $V$, let $a_1$ and $a_2$ be two non-negative rational numbers and let $\mu$ be a positive rational number. 
If: 
\begin{align*}
\left( V,(1-a_1)\Delta _1 + (1-a_2)\Delta _2+\frac{1}{\mu}\sL \right)
\end{align*}
is not log canonical at $p$, then we have: 
\begin{align*}
i(L_1,L_2;p) > 4a_1a_2\mu ^2, 
\end{align*}
where $i(L_1,L_2;p)$ is the local intersection multiplicity at $p$ of two general members $L_1,L_2 \in \sL$. 
\end{prop}
\begin{proof}
See {\cite[Theorem 3.1(1)]{Cor00}}. 
\end{proof}
A variant Corti's inequality is a useful tool for proving the absence of cylinders in smooth surfaces over algebraically non-closed fields (cf. {\cite{DK18, Saw}}). 
This article again uses this inequality in proving the absence of cylinders (see Lemmas \ref{Bs} and \ref{Corti(2)}). 

\subsection{Properties with respect to cylinders}\label{2-4}

In this subsection, we shall present some beneficial tools about cylinders for later use. 
\begin{defin}
Let $X$ be an algebraic variety over $k$ and let $U$ be a cylinder of $X$. 
Then the closed subset $X \backslash U$ on $X$ is called the {\it boundary} of $U$. 
\end{defin}
We think that the following lemma is well known but could not find a proof in the literature. 
Hence, we shall give the proof of this lemma for the readers' convenience. 
\begin{lem}\label{no cycle}
Let $V$ be a smooth projective surface over $k$ and let $U \simeq \bA ^1_k \times Z$ be a cylinder of $V$. 
Then the boundary divisor of $U$, say $D$, has no cycle. 
\end{lem}
\begin{proof}
If $D$ has a cycle, then $D_{\kc}$ also has a cycle, hence, we may assume $k=\kc$. 
The closures in $V$ of fibers of the projection $pr_Z : U \simeq \bA ^1_k \times Z \to Z$ yields a linear system on $V$, say $\sL$, hence we have the rational map $\Phi _{\sL} : V \dashrightarrow \overline{Z}$ to a projective model $\overline{Z}$ of the closure of $Z$ in $V$. 
Note that $\Bs (\sL)$ consists of at most one point by the configuration of $\sL$. 
Let $\psi : \bar{V} \to V$ be the shortest succession of blow-ups the point on $\Bs (\sL )$ and its infinitely near points such that the proper transform of $\sL$ is free of base points to give rise to a morphism $\varphi := \Phi _{\sL} \circ \psi : \bar{V} \to \overline{Z}$, where we shall define $\varphi := \Phi _{\sL}$ if $\Bs (\sL ) = \emptyset$. 
Hence $\varphi$ is a $\bP ^1$-fibration, moreover, $\psi ^{\ast}(D)_{\rm red.}$ is the union of a section and all singular fibers of $\varphi$. 
Thus, if $D$ has a cycle, then some singular fibers of $\varphi$ also have a cycle. 
However it is impossible, because, in general, it is known that any singular fiber of $\bP ^1$-fibration does not have a cycle (see, e.g., {\cite[Lemma 12.5]{Miy94}}). 
This completes the proof. 
\end{proof}
The following two lemmas are not difficult but will play an important role in constructing the cylinders. 
\begin{lem}\label{cylinder F0}
Let $V$ be a $k$-form of $\bP ^1_{\kc} \times \bP ^1_{\kc}$ and let $F_1$ and $F_2$ be two $0$-curves on $V_{\kc} $ with $(F_1 \cdot F_2) =1$. 
Namely, $\Pic (V_{\kc}) = \bZ [F_1] \oplus \bZ [F_2]$. 
Let $C$ be a geometrically rational curve on $V$ with $C_{\kc} \sim F_1+F_2$. 
If $C$ has a $k$-rational point, say $p$, then there exists  a cylinder $U \simeq \bA ^1_k \times \bA ^1_{\ast ,k}$ on $V$ satisfying $C \cap U = \emptyset$. 
\end{lem}
\begin{proof}
Let $C_1$ and $C_2$ be two rational curves on $V_{\kc}$ satisfying $p \in C_i$ and $C_i \sim F_i$ for $i=1,2$. 
Note that each of these curves uniquely exists and the union $C_1 + C_2$ is defined over $k$. 
Let $\varphi :V' \to V$ be a blowing-up at $p$, let $E$ be the exceptional curve of $\varphi$, and let us put $C' := \varphi ^{-1}_{\ast}(C)$ and $C_i' := \varphi ^{-1}_{\ast}(C_i)$ for $i=1,2$. 
Noticing that $E \simeq \bP ^1_k$ and $C_1' + C_2'$ is defined over $k$, by Lemma \ref{Severi-Brauer} we obtain the contraction $\psi :V' \to \bP ^2_k$ of $C_1'+C_2'$ over $k$, which maps $C' \cup E$ onto a pairs of lines in $\bP ^2_k$. 
Hence, we have an isomorphism $U := V \backslash (C \cup C_1 \cup C_2) \simeq \bP ^2_k \backslash \psi (C' \cup E) \simeq \bA ^1_k \times \bA ^1_{\ast ,k}$. 
In particular, $C \cap U = \emptyset$ by construction of $U$. 
\end{proof}
\begin{lem}\label{cylinder F2}
Let $V$ be a $k$-form of $\bF _2$, let $M$ be a $(-2)$-curve on $V_{\kc}$, let $F$ be a $0$-curve on $V_{\kc}$. 
Namely, $\Pic (V_{\kc}) = \bZ [M] \oplus \bZ [F]$. 
Let $C$ be a geometrically rational curve on $V$ with $(C_{\kc} \cdot M) \le 1$ and $(C_{\kc} \cdot F)=1$. 
If $C$ has a $k$-rational point, say $p$, then there exists  a cylinder $U \simeq \bA ^1_k \times \bA ^1_{\ast ,k}$ on $V$ satisfying $(M \cup C) \cap U = \emptyset$. 
\end{lem}
\begin{proof}
Since $V$ has the $k$-rational point $p$, we know that $V$ is the trivial $k$-form of $\bF _2$ by using Lemma \ref{Severi-Brauer}, i.e., $V \simeq \bP (\sO _{\bP ^1_k} \oplus \sO _{\bP ^1_k}(2))$. 
Meanwhile, $M$ is automatically defined over $k$ and there exists a unique rational curve $F_0$ on $V$ with $F_{0,\kc} \sim F$ passing through $p$. 
If $(C \cdot M) =0$, then $V \backslash (M \cup F_0 \cup C) \simeq \bA ^1_k \times \bA ^1_{\ast ,k}$ because the pair $(V,M+F_0+C)$ is a minimal normal compactification of $\bA ^1_k \times \bA ^1_{\ast ,k}$ (see {\cite{Suz79}} or {\cite{Koj02}}). 
In what follows, we consider the case $(C \cdot M) =1$. 
Then $p$ is the intersection point of $M$, $F_0$ and $C$. 
Let $\varphi :V' \to V$ be a blowing-up at $p$, let $E$ be the exceptional curve of $\varphi$, and let us put $M' := \varphi ^{-1}_{\ast}(M)$, $F' := \varphi ^{-1}_{\ast}(F_0)$ and $C' := \varphi ^{-1}_{\ast}(C)$. 
Since $F'$ is a $k$-form of a $(-1)$-curve and is defined over $k$, we obtain the contraction $\psi :V' \to V'' \simeq \bF _3$ of $F'$. 
Then $V \backslash (M \cup F_0 \cup C) \simeq V'' \backslash \psi (M' \cup C' \cup E) \simeq \bA ^1_k \times \bA ^1_{\ast ,k}$ because the pair $(V'',\psi _{\ast}(M'+C'+E))$ is a minimal normal compactification of $\bA ^1_k \times \bA ^1_{\ast ,k}$ (see {\cite{Suz79}} or {\cite{Koj02}}). 
\end{proof}
At the end of this subsection, we shall present the fact about cylinders in Du Val del Pezzo surfaces of Picard rank one. 
Let $S$ be a Du Val del Pezzo surface over $k$ with $\rho _k(S)=1$ and of degree $d$. 
Suppose that $S$ contains a cylinder, say $U \simeq \bA ^1_k \times Z$, where $Z$ is a smooth affine curve defined over $k$. 
The closures in $S$ of fibers of the projection $pr_Z : U \simeq \bA ^1_k \times Z \to Z$ yields a linear system, say $\sL$, on $S$. 
Then the following lemma can be shown by applying the proof of {\cite[Proposition 9]{DK18}}: 
\begin{lem}\label{Bs}
Let the notation be the same as above. 
Then $\Bs (\sL )$ consists of exactly a $k$-rational point, say $p$. 
Moreover, if $d \le 4$, then $p$ is a singular point on $S$. 
\end{lem}
\begin{proof}
Note that $\Bs (\sL)$ consists of at most one $k$-rational point by the configuration of $\sL$. 
Hence, we shall prove $\Bs (\sL) \not= \emptyset$. 
However, this assertion follows from $\rho _k(S)=1$. 
Indeed, any general members $L_1$ and $L_2$ of $\sL$ must intersect. 
Thus, there exists a unique $k$-rational point $p$ on $S_{\kc}$ such that $\Bs (\sL) = \{ p\}$. 

In order to prove the last assertion, suppose that $d \le 4$ and $p$ is a smooth point on $S$. 
Let $\sigma : \wS \to S$ be the minimal resolution over $k$. 
$\wS$ also contains a cylinder $\wU := \sigma ^{-1}(U) \simeq U$. 
The closures in $\wS$ of fibers of the projection $pr_Z : \widetilde{U} \simeq \bA ^1_k \times Z \to Z$ yields a linear system, say $\wsL$, on $\wS$. 
Then, $\sigma ^{-1}(p)$ is a $k$-rational point, say $\wtp$, on $\wS$. 
On the other hand, by $\rho _k(S)=1$, we have $\wsL \sim _{\bQ} a(-K_{\wS})$ for some $a \in \bQ _{>0}$. 
Indeed, we see $\wsL = \sigma ^{-1}_{\ast}(\sL ) = \sigma ^{\ast}(\sL )$, since $p$ is a smooth point on $S$. 
In addition to, we see $-K_{\wS} = \sigma ^{\ast}(-K_S)$ by construction of $\sigma$. 
Thus, we can obtain a contradiction by the argument similar to {\cite[Proposition 9]{DK18}}. 
In fact, show that $(\wS , \frac{1}{a}\wsL )$ is not log canonical at $\widetilde{p}$, moreover, derive contradiction by Corti's inequality (see Proposition \ref{Corti}). 
Therefore, $p$ must be a singular point on $S$ if $d \le 4$. 
\end{proof}

\section{Degree $3$ or higher}\label{3}

In this section, we shall prove Theorems \ref{main(1-1)} and \ref{main(1-2)}. 
Let $S$ be a Du Val del Pezzo surface over $k$ of degree $d \ge 3$ and let $\sigma : \wS \to S$ be the minimal resolution over $k$, so that $\wS$ is a weak del Pezzo surface. 
Hence, $\wS$ is a $k$-form of the Hirzebruch surface $\bF _2$ of degree $2$ or there exists the birational morphism $\wS _{\kc} \to \bP ^2_{\kc}$ over $\kc$. 
In particular, we can explicitly observe the configuration of $(-1)$-curves and $(-2)$-curves on $\wS _{\kc}$ (see {\cite{BW79,CT88}} or {\cite[\S \S 8.4--8.6 and \S \S 9.2]{Dol12}}). 

\subsection{Classification of Du Val del Pezzo surfaces of Picard rank one}\label{3-1}

The purpose of this subsection is that we shall classify Du Val del Pezzo surfaces of Picard rank one and of degree $\ge 3$. 

\begin{lem}\label{outside}
If there exists a $(-1)$-curve $E$ on $\wS _{\kc}$, which does not meet any $(-2)$-curve on $\wS _{\kc}$, such that either $E$ is defined over $k$ or the $\Gal$-orbit of $E$ is a disjoint union, then $\rho _k(S) > 1$. 
\end{lem}
\begin{proof}
We shall take this $(-1)$-curve $E$. 
By assumption, the direct image of the $\Gal$-orbit of $E$ via $\sigma$ can be contracted defined over $k$. 
Hence, there exists a blow-down $\tau : S \to S'$ over $k$, so that $\rho _k(S)>1$. 
\end{proof}
\begin{lem}\label{less (-2)-curves}
Assume that any $(-1)$-curve and $(-2)$-curve on $\wS _{\kc}$ are defined over $k$, respectively. 
If the number of all $(-2)$-curves on $\wS _{\kc}$ is less than $9-d$, then $\rho _k(S) >1$. 
\end{lem}
\begin{proof}
Indeed, we obtain $\rho _k(S) > (10-d) - (9-d) = 1$ since $\rho _k(\wS )= \rho _{\kc}(\wS _{\kc}) =10-d$ and $\rho _k(\wS  ) - \rho _k(S) < 9-d$ by assumption. 
\end{proof}
We shall view explicitly to these results of Lemmas \ref{outside} and \ref{less (-2)-curves} combined with the classification of weak del Pezzo surfaces of degree $\ge 3$ (see {\cite{BW79,CT88}} or {\cite[\S \S 8.4--8.6 and \S \S 9.2]{Dol12}}). 
The types satisfying the condition of Lemma \ref{outside} are as follows (see also Example \ref{A3+2A1andA4+A1}): 
\begin{itemize}
\item $d=7$ and $A_1$-type. 
\item $d=6$ and $(A_1)_>$-type. 
\item $d=5$ and $A_2+A_1$, $A_2$, $A_1$-type. 
\item $d=4$ and $A_4$, $A_2+A_1$-type. 
\item $d=3$ and $D_5$, $A_3+2A_1$, $A_4$, $A_3+A_1$, $A_2+2A_1$, $A_3$, $A_2+A_1$, $2A_1$-type. 
\end{itemize}
Similarly, the types satisfying the condition of Lemma \ref{less (-2)-curves} are as follows (see also Example \ref{A3+2A1andA4+A1}): 
\begin{itemize}
\item $d=6$ and $2A_1$-type. 
\item $d=5$ and $A_3$-type. 
\item $d=3$ and $A_4+A_1$-type. 
\end{itemize}
\begin{eg}\label{A3+2A1andA4+A1}
\begin{enumerate}
\item Assume that $d=3$ and $\wS$ is of $A_3+2A_1$-type. 
Then the dual graph corresponding to the union of all $(-2)$-curves and all $(-1)$-curves on $\wS _{\kc}$ is as follows: 
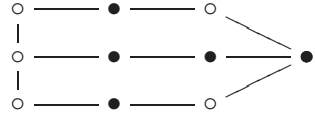
\begin{align*}
\xygraph{
\bullet -[l] \bullet -[l] \bullet -[l] \circ
(-[]!{+(0,.5)} \circ -[r] \bullet -[r] \circ -[]!{+(1,-.5)} \bullet , -[]!{+(0,-.5)} \circ -[r] \bullet -[r] \circ -[]!{+(1,.5)} \bullet )
}
\end{align*}
Here, ``$\circ$" means a $(-2)$-curve and ``$\bullet$" means a $(-1)$-curve. 
Hence, there exists a unique $(-1)$-curve on $\wS _{\kc}$ defined over $k$, which does not meet any $(-2)$-curve on $\wS _{\kc}$. 
Thus, we obtain $\rho _k(S)>1$ by Lemma \ref{outside}. 
\item Assume that $d=3$ and $\wS$ is of $A_4+A_1$-type. 
Then the dual graph corresponding to the union of all $(-2)$-curves and all $(-1)$-curves on $\wS _{\kc}$ is as follows: 
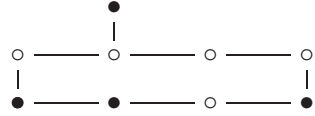
\begin{align*}
\xygraph{
\bullet -[r] \bullet -[r] \circ -[r] \bullet -[]!{+(0,.5)} \circ -[l] \circ -[l] \circ (-[]!{+(0,.5)} \bullet ,-[l] \circ -[]!{+(0,-.5)} \bullet)
}
\end{align*}
Here, ``$\circ$" means a $(-2)$-curve and ``$\bullet$" means a $(-1)$-curve. 
Hence, each $(-1)$-curve and $(-2)$-curve on $\wS _{\kc}$ is defined over $k$. 
Moreover, the number of all $(-2)$-curves on $\wS _{\kc}$ is less than $9-3=6$. 
Thus, we obtain $\rho _k(S)>1$ by Lemma \ref{less (-2)-curves}. 
\end{enumerate}
\end{eg}
Thus, we do not have to deal with these types. Furthermore, the following two lemmas hold: 
\begin{lem}
Assume that $d=5$ and $\wS$ is of $2A_1$-type, then $\rho _k(S) >1$. 
\end{lem}
\begin{proof}
Note that the dual graph corresponding to the union of all $(-2)$-curves and all $(-1)$-curves meeting a $(-2)$-curve, on $\wS _{\kc}$ is as follows: 
\begin{align*}
\xygraph{
\bullet -[r] \circ -[r] \bullet -[r] \circ -[r] \bullet
}
\end{align*}
Here, ``$\circ$" means a $(-2)$-curve and ``$\bullet$" means a $(-1)$-curve. 
Hence, we obtain a birational morphism $\tau :\wS \to \bP ^2_k$ over $k$ such that $\rho _k(\wS ) \ge 2+\rho _k(\bP ^2_k) = 3$. If $\rho _k(\wS ) \ge 4$, then we see $\rho _k(S) \ge 2$ by $\rho _k(\wS ) \le \rho _k(S) + 2$. 
If $\rho _k(\wS )=3$, then we see $\rho _k(S) =2$ since $\rho _k(\wS ) = \rho _k(S) +1$ by construction of $\tau$. 
\end{proof}
\begin{lem}
Assume that $d=4$. 
If $\wS$ is of $(A_3)_>$ or $(2A_1)_>$-type, then $\rho _k(S) >1$. 
\end{lem}
\begin{proof}
Note that the dual graph corresponding to the union of all $(-2)$-curves and all $(-1)$-curves meeting a $(-2)$-curve on $\wS _{\kc}$ is as follows according to types of $\wS$: 
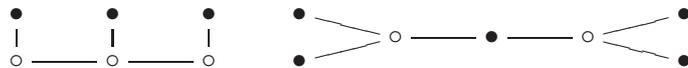
\begin{align*}
\xygraph{
(-[]!{+(0,.25)} \bullet , -[]!{+(0,-.25)} \circ -[r] \circ (-[]!{+(0,.5)} \bullet ,-[r] \circ -[]!{+(0,.5)} \bullet ))
}\qquad 
\xygraph{
\circ ( -[r] \bullet -[r] \circ ( - []!{+(1,.25)} \bullet , - []!{+(1,-.25)} \bullet[]!{+(0,0)} {}), - []!{+(-1,.25)} \bullet , - []!{+(-1,-.25)} \bullet))
}
\end{align*}
Here, ``$\circ$" means a $(-2)$-curve and ``$\bullet$" means a $(-1)$-curve. 
We consider two cases separately: 

{\it $(A_3)_>$-type:} 
In this case, we obtain a birational morphism $\tau :\wS \to \bF _1$ over $k$ such that $\rho _k(\wS ) \ge 2+\rho _k(\bF _1) = 4$. If $\rho _k(\wS ) \ge 5$, then we see $\rho _k(S) \ge 2$ by $\rho _k(\wS ) \le \rho _k(S) + 3$. 
If $\rho _k(\wS )=4$, then we see $\rho _k(S) =2$ since $\rho _k(\wS ) = \rho _k(S) +2$ by construction of $\tau$. 

{\it $(2A_1)_>$-type:} 
In this case, we obtain a birational morphism $\tau :\wS \to \bF _1$ over $k$ such that $\rho _k(\wS ) \ge 1+\rho _k(\bF _1) = 3$. If $\rho _k(\wS ) \ge 4$, then we see $\rho _k(S) \ge 2$ by $\rho _k(\wS ) \le \rho _k(S) + 2$. 
If $\rho _k(\wS )=3$, then we see $\rho _k(S) =2$ since $\rho _k(\wS ) = \rho _k(S) +1$ by construction of $\tau$. 
\end{proof}

In what follows, we shall consider the remaining cases. 
For any remaining type of weak del Pezzo surfaces, there exists certainly a Du Val del Pezzo surface $S$ over $k$ with $\rho _k(S)=1$ such that its minimal resolution $\wS$ is of this type. 
Then we see that the Picard number $\rho _k(\wS )$ is as in Table \ref{list(1-1)} (resp. \ref{list(1-2)}) according to the type of $\wS$. 
Here, Table \ref{list(1-1)} (resp. Table \ref{list(1-2)}) summarizes all types that $S_{\kc}$ admits (resp. does not admit) a singular point, which is $k$-rational. 
\begin{table}[htbp]
\begin{tabular}{c}

\begin{minipage}[c]{1\hsize}
\begin{center}
\caption{Types of $\wS$ in Theorems \ref{main(1-1)} and \ref{main(1-2)} (I)}\label{list(1-1)}
 \begin{tabular}{|c|c|c|c||c|c|c|c|}
 \hline

\multirow{2}{*}{$d$} & Type & \multirow{2}{*}{$n^{\circ}$} & \multirow{2}{*}{Dual graph} & \multirow{2}{*}{$d$} & Type & \multirow{2}{*}{$n^{\circ}$} & \multirow{2}{*}{Dual graph} \\ 

 & $\rho _k(\wS )$ & & & & $\rho _k(\wS )$ & & \\ \hline \hline

\multirow{2}{*}{$8$} & $A_1$ & \multirow{2}{*}{$9^{\circ}$} & 
\multirow{2}{*}{$\xygraph{
	\circ ([]!{+(0,-.3)} {M})
}$} &

\multirow{2}{*}{$6$} & $A_2+A_1$ & \multirow{2}{*}{$1^{\circ}$} & 
\multirow{2}{*}{$\xygraph{
( -[]!{+(0,-.25)} \circ ([]!{+(+.3,0)} {L}), 
-[]!{+(0,.25)} \bullet ([]!{+(0,0)} {}) -[]!{+(-.5,0)} \circ ([]!{+(0,0)} {}) -[]!{+(-.5,0)} \circ ([]!{+(0,0)} {}))
}$} \rule[0mm]{0mm}{5mm} \\ 

 & $2$ & & & & $4$ & & \rule[0mm]{0mm}{5mm} \\ \hline

\multirow{2}{*}{$6$} & $A_2$ & \multirow{2}{*}{$6^{\circ}$} & 
\multirow{2}{*}{$\xygraph{
\circ ([]!{+(0,-.3)} {F})
( -[]!{+(.5,0)} \circ ([]!{+(0,-.3)} {M}),
 - []!{+(-.5,.25)} \bullet ([]!{+(0,0)} {}), - []!{+(-.5,-.25)} \bullet ([]!{+(0,0)} {}) )
}$} &

\multirow{2}{*}{$6$} & $(A_1)_<$ & \multirow{2}{*}{$1^{\circ}$} & 
\multirow{2}{*}{$\xygraph{
\circ ([]!{+(+.3,0)} {L}) 
(-[]!{+(-.5,.25)} \bullet ([]!{+(0,0)} {}), -[]!{+(-.5,0)} \bullet ([]!{+(0,0)} {}), -[]!{+(-.5,-.25)} \bullet ([]!{+(0,0)} {}) )
}$} \rule[0mm]{0mm}{5mm} \\ 

 & $3$ & & & & $2$ & & \rule[0mm]{0mm}{5mm} \\ \hline

\multirow{2}{*}{$5$} & $A_4$ & \multirow{2}{*}{$1^{\circ}$} & 
\multirow{2}{*}{$\xygraph{
(- []!{+(0,.25)} \circ ([]!{+(0,0)} {})(-[]!{+(-.5,0)} \circ ([]!{+(0,0)} {})-[]!{+(-.5,0)} \circ ([]!{+(0,0)} {}),-[]!{+(.5,0)} \bullet ([]!{+(0,0)} {})),- []!{+(0,-.25)} \circ ([]!{+(.3,0)} {L}))
}$} &

\multirow{2}{*}{$4$} & $D_5$ & \multirow{2}{*}{$1^{\circ}$} & 
\multirow{2}{*}{$\xygraph{
(- []!{+(0,.25)} \circ ([]!{+(0,0)} {})(-[]!{+(-.5,0)} \circ ([]!{+(0,0)} {})-[]!{+(-.5,0)} \circ ([]!{+(0,0)} {}),-[]!{+(.5,0)} \circ ([]!{+(0,0)} {})-[]!{+(.5,0)} \bullet ([]!{+(0,0)} {})),- []!{+(0,-.25)} \circ ([]!{+(.3,0)} {L}))
}$} \rule[0mm]{0mm}{5mm} \\ 

 & $5$ & & & & $6$ & & \rule[0mm]{0mm}{5mm} \\ \hline

\multirow{2}{*}{$4$} & $A_3+2A_1$ & \multirow{2}{*}{$10^{\circ}$} & 
\multirow{2}{*}{$\xygraph{
(-[]!{+(0,.25)} \circ , -[]!{+(0,-.25)} \bullet -[]!{+(.5,0)} \circ ([]!{+(0,.3)} {F_1}) -[]!{+(.5,0)} \circ ([]!{+(0,.3)} {M}) -[]!{+(.5,0)} \circ ([]!{+(0,.3)} {F_2}) -[]!{+(.5,0)} \bullet -[]!{+(0,.5)} \circ )
}$} &

\multirow{2}{*}{$4$} & $D_4$ & \multirow{2}{*}{$6^{\circ}$} & 
\multirow{2}{*}{$\xygraph{
\circ ([]!{+(0,-.3)} {F})
(- []!{+(-.5,.25)} \circ ([]!{+(0,0)} {}) -[]!{+(-.5,0)} \bullet ([]!{+(0,0)} {}),- []!{+(-.5,-.25)} \circ ([]!{+(0,0)} {}) -[]!{+(-.5,0)} \bullet ([]!{+(0,0)} {}),-[]!{+(.5,0)} \circ ([]!{+(0,-.3)} {M}))
}$} \rule[0mm]{0mm}{5mm} \\ 

 & $4$ or $6$ & & & & $4$ & & \rule[0mm]{0mm}{5mm} \\ \hline

\multirow{2}{*}{$4$} & $A_3+A_1$ & \multirow{2}{*}{$2^{\circ}$} & 
\multirow{2}{*}{$\xygraph{
\circ ([]!{+(0,-.3)} {L_1}) -[]!{+(.5,0)} \bullet ([]!{+(0,0)} {}) -[]!{+(.5,0)} 
\circ ([]!{+(0,0)} {}) -[]!{+(.5,0)} \circ ([]!{+(0,0)} {}) -[]!{+(.5,0)} 
\circ ([]!{+(0,-.3)} {L_2}) ( - []!{+(.5,.25)} \bullet ([]!{+(0,0)} {}), - []!{+(.5,-.25)} \bullet ([]!{+(0,0)} {}))
}$} &

\multirow{2}{*}{$4$} & $A_2+2A_1$ & \multirow{2}{*}{$4^{\circ}$} & 
\multirow{2}{*}{$\xygraph{
\circ ([]!{+(0,0)} {}) -[]!{+(.5,0)} \bullet ([]!{+(0,0)} {}) 
-[]!{+(.5,0)} \circ ([]!{+(0,-.3)} {F_1}) -[]!{+(.5,0)} \circ ([]!{+(0,-.3)} {F_2})
-[]!{+(.5,0)} \bullet ([]!{+(0,0)} {}) -[]!{+(.5,0)} \circ ([]!{+(0,0)} {})
}$} \rule[0mm]{0mm}{5mm} \\ 

 & $5$ & & & & $3$ & & \rule[0mm]{0mm}{5mm} \\ \hline

\multirow{2}{*}{$4$} & $4A_1$ & \multirow{2}{*}{$8^{\circ}$} & 
\multirow{2}{*}{$\xygraph{
\circ ([]!{+(0,-.3)} {M})
}$ $\xygraph{
\circ ([]!{+(0,0)} {}) -[]!{+(.5,0)} \bullet ([]!{+(0,0)} {}) -[]!{+(.5,0)} \circ ([]!{+(0,-.3)} {C}) -[]!{+(.5,0)} \bullet ([]!{+(0,0)} {}) -[]!{+(.5,0)} \circ ([]!{+(0,0)} {})
}$} &

\multirow{2}{*}{$4$} & $(A_3)_<$ & \multirow{2}{*}{$10^{\circ}$} & 
\multirow{2}{*}{$\xygraph{
\circ ([]!{+(0,-.3)} {F_1}) ( -[]!{+(.5,0)} \circ ([]!{+(0,-.3)} {M})-[]!{+(.5,0)} \circ ([]!{+(0,-.3)} {F_2}) ( - []!{+(.5,.25)} \bullet ([]!{+(0,0)} {}), - []!{+(.5,-.25)} \bullet[]!{+(0,0)} {}), - []!{+(-.5,.25)} \bullet ([]!{+(0,0)} {}), - []!{+(-.5,-.25)} \bullet[]!{+(0,0)} {}))
}$} \rule[0mm]{0mm}{5mm} \\ 

 & $4$ or $5$ & & & & $3$ or $4$ & & \rule[0mm]{0mm}{5mm} \\ \hline

\multirow{2}{*}{$4$} & $3A_1$ & \multirow{2}{*}{$5^{\circ}$} & 
\multirow{2}{*}{$\xygraph{
\circ ([]!{+(0,0)} {}) -[]!{+(.5,0)} \bullet ([]!{+(0,0)} {}) 
-[]!{+(.5,0)} \circ ([]!{+(0,-.3)} {C}) -[]!{+(.5,0)} \bullet ([]!{+(0,0)} {}) -[]!{+(.5,0)} \circ ([]!{+(0,0)} {})
}$} &

\multirow{2}{*}{$4$} & $A_2$ & \multirow{2}{*}{$4^{\circ}$} & 
\multirow{2}{*}{$\xygraph{
\circ ([]!{+(0,-.3)} {F_1}) ( -[]!{+(.5,0)} \circ ([]!{+(0,-.3)} {F_2}) ( -[]!{+(.5,.25)} \bullet ([]!{+(0,0)} {}), -[]!{+(.5,-.25)} \bullet[]!{+(0,0)} {}), - []!{+(-.5,.25)} \bullet ([]!{+(0,0)} {}), - []!{+(-.5,-.25)} \bullet[]!{+(0,0)} {}))
}$} \rule[0mm]{0mm}{5mm} \\ 

 & $3$ & & & & $2$ & & \rule[0mm]{0mm}{5mm} \\ \hline

\multirow{2}{*}{$4$} & $(2A_1)_<$ & \multirow{2}{*}{$8^{\circ}$} & 
\multirow{2}{*}{$\xygraph{
\circ ([]!{+(0,-.3)} {M})
}$ $\xygraph{
\circ ([]!{+(0,-.3)} {C})
 ((-[]!{+(.5,.25)} \bullet ([]!{+(0,0)} {}), -[]!{+(.5,-.25)} \bullet ([]!{+(0,0)} {})), -[]!{+(-.5,.25)} \bullet ([]!{+(0,0)} {}), -[]!{+(-.5,-.25)} \bullet ([]!{+(0,0)} {}))
}$} &

\multirow{2}{*}{$4$} & $A_1$ & \multirow{2}{*}{$5^{\circ}$} & 
\multirow{2}{*}{$\xygraph{
\circ ([]!{+(0,-.3)} {C})
 ((-[]!{+(.5,.25)} \bullet ([]!{+(0,0)} {}), -[]!{+(.5,-.25)} \bullet ([]!{+(0,0)} {})), -[]!{+(-.5,.25)} \bullet ([]!{+(0,0)} {}), -[]!{+(-.5,-.25)} \bullet ([]!{+(0,0)} {}))
}$} \rule[0mm]{0mm}{5mm} \\ 

 & $3$ & & & & $2$ & & \rule[0mm]{0mm}{5mm} \\ \hline

\multirow{2}{*}{$3$} & $E_6$ & \multirow{2}{*}{$1^{\circ}$} & 
\multirow{2}{*}{$\xygraph{
(-[]!{+(0,.25)} \circ ([]!{+(0,0)} {})(-[]!{+(-.5,0)} \circ ([]!{+(0,0)} {})-[]!{+(-.5,0)} \circ ([]!{+(0,0)} {}),-[]!{+(.5,0)} \circ ([]!{+(0,0)} {})-[]!{+(.5,0)} \circ ([]!{+(0,0)} {})-[]!{+(.5,0)} \circ ([]!{+(0,0)} {})-[]!{+(0,-.5)} \bullet ([]!{+(0,0)} {})),- []!{+(0,-.25)} \circ ([]!{+(.3,0)} {L}))
}$} &

\multirow{2}{*}{$3$} & $A_5+A_1$ & \multirow{2}{*}{$2^{\circ}$} & 
\multirow{2}{*}{$\xygraph{
(-[]!{+(0,.25)} \circ ([]!{+(0,0)} {})(-[]!{+(-.5,0)} \circ ([]!{+(0,-.3)} {L_1}),-[]!{+(.5,0)} \circ ([]!{+(0,0)} {})-[]!{+(.5,0)} \circ ([]!{+(0,0)} {})-[]!{+(.5,0)} \circ ([]!{+(0,-.3)} {L_2})-[]!{+(.5,0)} \bullet ([]!{+(0,0)} {})-[]!{+(0,-.5)} \circ ([]!{+(0,0)} {})),- []!{+(0,-.25)} \bullet ([]!{+(0,0)} {}))
}$} \rule[0mm]{0mm}{5mm} \\ 

 & $7$ & & & & $7$ & & \rule[0mm]{0mm}{5mm} \\ \hline

\multirow{2}{*}{$3$} & $3A_2$ & \multirow{2}{*}{$2^{\circ}$} & 
\multirow{2}{*}{$\xygraph{
(-[]!{+(0,.25)} \bullet ([]!{+(0,0)} {})-[]!{+(-.5,0)} \circ ([]!{+(0,0)} {})-[]!{+(-.5,0)} \circ ([]!{+(0,0)} {}),-[]!{+(0,-.25)} \circ ([]!{+(-.3,0)} {L_1})-[]!{+(.5,0)} \circ ([]!{+(.3,0)} {L_2})-[]!{+(0,.5)} \bullet ([]!{+(0,0)} {})-[]!{+(.5,0)} \circ ([]!{+(0,0)} {})-[]!{+(.5,0)} \circ ([]!{+(0,0)} {}))
}$} &

\multirow{2}{*}{$3$} & $A_5$ & \multirow{2}{*}{$2^{\circ}$} & 
\multirow{2}{*}{$\xygraph{
\circ ([]!{+(0,-.3)} {L_1}) -[]!{+(.5,0)} \circ ([]!{+(0,0)} {})
 (-[]!{+(0,-.4)} \bullet ([]!{+(0,0)} {}), -[]!{+(.5,0)} \circ ([]!{+(0,0)} {}) -[]!{+(.5,0)} \circ ([]!{+(0,0)} {}) -[]!{+(.5,0)} \circ ([]!{+(0,-.3)} {L_2}) (-[]!{+(.5,.25)} \bullet ([]!{+(0,0)} {}), -[]!{+(.5,-.25)} \bullet ([]!{+(0,0)} {})))
}$} \rule[0mm]{0mm}{5mm} \\ 

 & $4$ or $7$ & & & & $6$ & & \rule[0mm]{0mm}{5mm} \\ \hline

\multirow{2}{*}{$3$} & $2A_2+A_1$ & \multirow{2}{*}{$3^{\circ}$} & 
\multirow{2}{*}{$\xygraph{
(-[]!{+(0,.25)} \circ ([]!{+(0,0)} {})-[]!{+(-.5,0)} \circ ([]!{+(0,0)} {}),-[]!{+(0,-.25)} \bullet ([]!{+(0,0)} {})-[]!{+(.5,0)} \circ ([]!{+(0,.3)} {Q})-[]!{+(.5,0)} \bullet ([]!{+(0,0)} {})-[]!{+(0,.5)} \circ ([]!{+(0,0)} {})-[]!{+(.5,0)} \circ ([]!{+(0,0)} {}))
}$} &

\multirow{2}{*}{$3$} & $D_4$ & \multirow{2}{*}{$1^{\circ}$} & 
\multirow{2}{*}{$\xygraph{
\circ ([]!{+(.3,0)} {L})
 (-[]!{+(-.5,.25)} \circ ([]!{+(0,0)} {})-[]!{+(-.5,0)} \bullet ([]!{+(0,0)} {}), -[]!{+(-.5,0)} \circ ([]!{+(0,0)} {})-[]!{+(-.5,0)} \bullet ([]!{+(0,0)} {}), -[]!{+(-.5,-.25)} \circ ([]!{+(0,0)} {})-[]!{+(-.5,0)} \bullet ([]!{+(0,0)} {}))
}$} \rule[0mm]{0mm}{5mm} \\ 

 & $4$ & & & & $3$ & & \rule[0mm]{0mm}{5mm} \\ \hline

\multirow{2}{*}{$3$} & $2A_2$ & \multirow{2}{*}{$7^{\circ}$} & 
\multirow{2}{*}{$\xygraph{
\circ ([]!{+(0,-.3)} {M}) -[]!{+(.5,0)} \circ ([]!{+(0,-.3)} {F}) -[]!{+(.5,0)} \bullet ([]!{+(0,0)} {}) -[]!{+(.5,0)} \circ ([]!{+(0,0)} {}) -[]!{+(.5,0)} \circ ([]!{+(0,-.3)} {C}) 
( -[]!{+(.5,0)} \bullet ([]!{+(0,0)} {}), - []!{+(.5,-.25)} \bullet ([]!{+(0,0)} {}), - []!{+(.5,.25)} \bullet ([]!{+(0,0)} {}) )
}$} &

\multirow{2}{*}{$3$} & $4A_1$ & \multirow{2}{*}{$3^{\circ}$} & 
\multirow{2}{*}{$\xygraph{
\circ ([]!{+(.3,0)} {Q})
 (-[]!{+(-.5,.25)} \bullet ([]!{+(0,0)} {})-[]!{+(-.5,0)} \circ ([]!{+(0,0)} {}), -[]!{+(-.5,0)} \bullet ([]!{+(0,0)} {})-[]!{+(-.5,0)} \circ ([]!{+(0,0)} {}), -[]!{+(-.5,-.25)} \bullet ([]!{+(0,0)} {})-[]!{+(-.5,0)} \circ ([]!{+(0,0)} {}))
}$} \rule[0mm]{0mm}{5mm} \\ 

 & $5$ & & & & $3$ & & \rule[0mm]{0mm}{5mm} \\ \hline

\multirow{2}{*}{$3$} & $A_2$ & \multirow{2}{*}{$2^{\circ}$} & 
\multirow{2}{*}{$\xygraph{
\circ ([]!{+(0,-.3)} {L_1})
 (-[]!{+(.5,0)} \circ ([]!{+(0,-.3)} {L_2}) (-[]!{+(.5,.25)} \bullet ([]!{+(0,0)} {}), -[]!{+(.5,0)} \bullet ([]!{+(0,0)} {}), -[]!{+(.5,-.25)} \bullet ([]!{+(0,0)} {})), -[]!{+(-.5,.25)} \bullet ([]!{+(0,0)} {}), -[]!{+(-.5,0)} \bullet ([]!{+(0,0)} {}), -[]!{+(-.5,-.25)} \bullet ([]!{+(0,0)} {}))
}$} &

\multirow{2}{*}{$3$} & $A_1$ & \multirow{2}{*}{$3^{\circ}$} & 
\multirow{2}{*}{$\xygraph{
\circ ([]!{+(0,-.3)} {Q})
 ((-[]!{+(.5,.25)} \bullet ([]!{+(0,0)} {}), -[]!{+(.5,0)} \bullet ([]!{+(0,0)} {}), -[]!{+(.5,-.25)} \bullet ([]!{+(0,0)} {})), -[]!{+(-.5,.25)} \bullet ([]!{+(0,0)} {}), -[]!{+(-.5,0)} \bullet ([]!{+(0,0)} {}), -[]!{+(-.5,-.25)} \bullet ([]!{+(0,0)} {}))
}$} \rule[0mm]{0mm}{5mm} \\ 

 & $2$ or $3$ & & & & $2$ & & \rule[0mm]{0mm}{5mm} \\ \hline

\end{tabular}
\end{center}
\end{minipage}
\\

\begin{minipage}[c]{1\hsize}
\begin{center}
\caption{Types of $\wS$ in Theorems \ref{main(1-1)} and \ref{main(1-2)} (II)}\label{list(1-2)}
 \begin{tabular}{|c|c|c|c||c|c|c|c||c|c|c|c|} \hline
$d$ & Type & $\rho _k(\wS )$ & $V$ & $d$ & Type & $\rho _k(\wS )$ & $V$ & $d$ & Type & $\rho _k(\wS )$ & $V$ \\ \hline \hline

$4$ & $4A_1$ & $2$ or $3$ & $S_8$ & 
$4$ & $(2A_1)_<$ & $2$ & $W_4$ & 
$3$ & $3A_2$ & $2$ & $S_6$ \rule[0mm]{0mm}{5mm} \\ \hline
$3$ & $2A_2$ & $3$ & $W_4$ & 
$3$ & $4A_1$ & $2$ & $S_9$ &
$3$ & $3A_1$ & $2$ & $S_6$ \rule[0mm]{0mm}{5mm}  \\ \hline
\end{tabular}
\end{center}
\end{minipage}

\end{tabular}
\end{table}

Now, we will present the remark of Tables \ref{list(1-1)} and \ref{list(1-2)}. 

For a weak del Pezzo surface $\wS$ such that the triplet $(d,\text{Type}, \rho _k(\wS))$ is one of the list in Table \ref{list(1-1)}, ``$n^{\circ}$'' and ``Dual graph'' in Table \ref{list(1-1)} present the explicit construction a birational morphism $\tau: \wS \to V$, where $V$ is defined by the following according to the number of $n^{\circ}$: 
\begin{itemize}
\item $V$ is a $k$-form of $\bP ^2_{\kc}$, i.e., $V_{\kc} \simeq \bP ^2_{\kc}$, if $n^{\circ} =1^{\circ}$, $2^{\circ}$ or $3^{\circ}$. 
\item $V$ is a $k$-form of $\bP ^1_{\kc} \times \bP ^1_{\kc}$, i.e., $V_{\kc} \simeq \bP ^1_{\kc} \times \bP ^1_{\kc}$, if $n^{\circ} =4^{\circ}$ or $5^{\circ}$. 
\item $V$ is a $k$-form of $\bF _2$, i.e., $V_{\kc} \simeq \bF _2$, if $n^{\circ} =6^{\circ}$, $7^{\circ}$, $8^{\circ}$, $9^{\circ}$ or $10^{\circ}$. 
\end{itemize}
Moreover, ``Dual graph'' in Table \ref{list(1-1)} is a dual graph corresponding to the union of all $(-2)$-curves and some $(-1)$-curves on $\wS _{\kc}$, which is clearly defined over $k$, where ``$\circ$" means a $(-2)$-curve and ``$\bullet$" means a $(-1)$-curve. 
The birational morphism $\tau$ is then defined by the compositions of the successive contractions of the $(-1)$-curves corresponding to all vertices $\bullet$ in the dual graph in Table \ref{list(1-1)} and that of the proper transform of the branch components such that all curves corresponding to vertices with no label in the dual graph in Table \ref{list(1-1)} are contracted by $\tau$, according to the type of $\wS$. 
Note that, by construction, $\tau$ is defined over $k$. 
This birational morphism $\tau$ will be used for the explicit construction of cylinders in $S$ in \S \S \ref{3-2}. 

Similarly, for a weak del Pezzo surface $\wS$ such that the triplet $(d,\text{Type}, \rho _k(\wS ))$ is one of the list in Table \ref{list(1-2)}, there exists a birational morphism $\tau :\wS \to V$ over $k$ such that $V$ is that as in Table \ref{list(1-2)}. 
Here, in Table \ref{list(1-2)}, $S_{d'}$ means a smooth del Pezzo surface of degree $d'$ with $\rho _k(S_{d'})=1$ and $W_4$ means a weak del Pezzo surface of degree $4$ and of $(2A_1)_<$-type with $\rho _k(W_4)=2$. 
Notice that $W_4$ is $k$-minimal by Proposition \ref{minimal}. 

The following example presents how to determine the value of $\rho _k(\wS )$ according to types of weak del Pezzo surfaces. 
By the argument in this example or the argument similar, the lists of Tables \ref{list(1-1)} and \ref{list(1-2)} are constructed. 
\begin{eg}\label{4A1}
Assume that $\rho _k(S)=1$, $d=3$ and $\wS$ is of $4A_1$-type. 
Then we shall show that $\rho _k(\wS )$ is equal to $2$ or $3$ (see also Tables \ref{list(1-1)} and \ref{list(1-2)}). 
We will consider whether $S_{\kc}$ admits a singular point of type $A_1$, which is $k$-rational or not, in what follows. 

{\it $S_{\kc}$ does not admit any singular point of type $A_1$ which is $k$-rational:} 
It is known that $\wS _{\kc}$ can be constructed by the blow-up at the intersection points of four lines in a general linear position on $\bP ^2_{\kc}$ (see {\cite[\S \S 9.2]{Dol12}}). 
Notice that the union of six $(-1)$-curves on $\wS _{\kc}$ corresponding to these points is defined over $k$, so is this blow-up. 
In other words, there exists the blow-up $\tau :\wS \to S_9$ over $k$ at the intersection points $\{ x_{i,j}\} _{1\le i<j\le 4}$ of four lines $L_1,\dots ,L_4$ in a general linear position on $\bP ^2_{\kc}$ such that the union of these lines is defined over $k$, where $S_9$ is a $k$-form of $\bP ^2_{\kc}$ and $x_{i,j}$ is the intersection point on $L_i$ and $L_j$. 
Supposing that $\rho _k(\wS ) - \rho _k(S) \ge 2$, without loss of generality, $L_1$ and $L_2$ (resp. $L_3$ and $L_4$) are exchanged by the $\Gal$-action. 
Namely, $\rho _k(\wS ) - \rho _k(S)=2$. 
Since $x_{1,2}$ and $x_{3,4}$ are $k$-rational points, we see $\rho _k(\wS ) -\rho _k(S_9) \ge 3$. 
Thus, we obtain $1= \rho _k(S) \ge 2$, which is absurd. 
Thus, we see $\rho _k(\wS ) =2$. 

{\it $S_{\kc}$ admits a singular point of type $A_1$ which is $k$-rational:} 
By assumption, we can take the $(-2)$-curve $M$ on $\wS _{\kc}$, which is defined over $k$. 
Then we can take the birational morphism $\tau :\wS \to V$, which is the compositions of the successive contractions of the $(-1)$-curves corresponding to all vertices $\bullet$ in the dual graph in Table \ref{list(1-1)} and that of the proper transform of the branch components such that all curves corresponding to vertices with no label in the dual graph in Table \ref{list(1-1)} are contracted by $\tau$, where $V$ is a $k$-form of $\bP ^2_{\kc}$. 
Note that, by construction, $\tau$ is defined over $k$. 
Put $m:= \rho _k(\wS) - \rho _k(S)$. 
By construction of $\tau$, we see $\rho _k(\wS ) - \rho _k(V) = 2(m-1)$. 
Thus, we obtain $\rho _k(\wS ) =3$ because of $\rho _k(\wS ) = m-1 = 2(m-1)-1$ by $\rho _k(S)=\rho _k(V) =1$. 
This means that all $(-2)$-curves except for $M$ on $\wS _{\kc}$ are exchanged by the $\Gal$-action. 
\end{eg}

\subsection{Proof of Theorems \ref{main(1-1)} and \ref{main(1-2)}}\label{3-2}
In this subsection, we shall show Theorems \ref{main(1-1)} and \ref{main(1-2)}. 
With the notation as above, assume further that $\rho _k(S)=1$. 

At first, we shall show the ``only if" part in Theorem \ref{main(1-2)}. 
Assume that $d$ is equal to $3$ or $4$ and $S$ contains a cylinder, say $U \simeq \bA ^1_k \times Z$. 
The closures in $S$ of fibers of the projection $pr_Z : U \simeq \bA ^1_k \times Z \to Z$ yields a linear system, say $\sL$, on $S$. By Lemma \ref{Bs}, $\Bs (\sL )$ consists of only one singular point on $S$, which is $k$-rational, say $p$. 
In order to prove the ``only if" part in Theorem \ref{main(1-2)}, we shall show that the singularity of $p$ is not type $A_1^{++}$ on $S$. 
\begin{lem}\label{lem4.1}
Let the notation and the assumptions be the same as above. 
If the singular point $p$ is of type $A_1$ on $S_{\kc}$, then $p$ is not of type $A_1^{++}$ on $S$. 
\end{lem}
\begin{proof}
Since $U _{\kc}$ is smooth, $\wU := \sigma ^{-1}(U) \simeq U$ is a cylinder on $\wS$. 
The closures in $\wS$ of fibers of the projection $pr_Z : \wU \simeq \bA ^1_k \times Z \to Z$ yields a linear system, say $\wsL$, on $\wS$. 
By the assumption, the reducible exceptional locus over $\kc$ of the minimal resolution at $p$ consists of only one $(-2)$-curve, say $M$. 
Notice that $M$ is defined over $k$. 
By construction of $\wsL$, we see that a general member of $\wsL _{\kc}$ does not meet any $(-2)$-curve other than $M$ on $\wS _{\kc}$. 
Hence, we can write $\wsL \sim _{\bQ} a(-K_{\wS}) - bM$ for some $a,b \in \bQ$. 
Noting that the degree $d$ of $S$ is equal to $3$ or $4$, we have $(\wsL )^2 = da^2-2b^2 \not= 0$ because of $a,b \in \bQ$. 
Thus, $\Bs (\wsL ) \not= \emptyset$. 
In particular, $\Bs (\wsL )$ consists of exactly one $k$-rational point lying on $M$. 
Thus, we obtain $M(k) \not= \emptyset$, which implies that $p$ is not of type $A_1^{++}$ on $S$. 
\end{proof}
By Lemma \ref{lem4.1}, this completes the proof of the ``only if" part in Theorem \ref{main(1-2)}. 

Next, in order to show Theorem \ref{main(1-1)} and the ``if" part in Theorem \ref{main(1-2)}, we shall assume that $S_{\kc}$ has a singular point, which is $k$-rational, such that it is not of type $A_1^{++}$ on $S$ if $d$ is equal to $3$ or $4$. 
By using Table \ref{list(1-1)}, we can construct the birational morphism $\tau :\wS \to V$ as in \S \S \ref{3-1}. 
Let $N$ be the divisor consisting of the union of all $(-2)$-curves on $\wS _{\kc}$ and let $E$ be the reduced exceptional divisor of $\tau$. 
Then the support $\Supp (N+E)$ corresponds to the dual graph as in Table \ref{list(1-1)} according to the type of $\wS$. 
Noting that the number of $n^{\circ}$ is determined depending on the type of $\wS$ by using Table \ref{list(1-1)}, we shall construct a cylinder $\wU$ on $\wS$ according to the number of $n^{\circ}$: 

{\it $n^{\circ} = 1^{\circ}$:} 
In this case, we see that $V \simeq \bP ^2_k$ and the image of the vertex with a label written $L$ via $\tau$ is a line on $V \simeq \bP ^2_k$, say $L$. 
Put $\wU := \wS \backslash \Supp (N+E)$. Then we see $\wU \simeq V \backslash L \simeq \bA ^2_k$. 

{\it $n^{\circ} = 2^{\circ}$:} 
In this case, we see that $V \simeq \bP ^2_k$ and the images of two vertices with labels written $L_1$ and $L_2$ via $\tau$ are distinct two lines on $V \simeq \bP ^2_k$, say $L_1$ and $L_2$. 
Put $\wU := \wS \backslash \Supp (N+E)$. 
Then $\wU \simeq V \backslash (L_1 \cup L_2) \simeq \bA ^1_k \times C_{(1)}$. 
Furthermore, $\wU \simeq \bA ^1_k \times \bA ^1_{\ast ,k}$ only if $L_1$ and $L_2$ are defined over $k$. 

{\it $n^{\circ} = 3^{\circ}$:} 
In this case, $V$ is a $k$-form of $\bP ^2_{\kc}$. 
Note that $S$ has a singular point of type $A_1^+$ by the configuration of curves in $\wS _{\kc}$ and assumption. 
Hence, $\wS _{\kc}$ has a $k$-rational point lying on $\Supp (N)$, so does $V$. 
Thus, $V \simeq \bP ^2_k$ by Lemma \ref{Severi-Brauer}. 
Meanwhile, the image of the vertex with a label written $Q$ via $\tau$ is an irreducible conic on $V \simeq \bP ^2_k$, say $Q$. 
Notice that $Q$ admits a $k$-rational point, so that $Q \simeq \bP ^1_k$ by Lemma \ref{Severi-Brauer}. 
Let $L$ be a line on $V$ such that $L$ and $Q$ tangentially meet at a general $k$-rational point. 
Noting that $\tau ^{-1}_{\ast}(L)$ is defined over $k$, set $\wU := \wS \backslash \Supp (N+E+\tau ^{-1}_{\ast}(L))$. 
Then $\wU$ is certainly the cylinder on $\wS$ since $\wU \simeq V \backslash (Q \cup L) \simeq \bA ^1_k \times \bA ^1_{\ast ,k}$. 

{\it $n^{\circ} = 4^{\circ}$:} 
In this case, $V$ is a $k$-form of $\bP ^1_{\kc} \times \bP ^1_{\kc}$. 
Note that two curves corresponding to the vertices with labels written $F_1$ and $F_2$ meet transversely at a point. 
Since this point is a $k$-rational point, so is the image, say $x$, via $\tau _{\kc}$. 
Moreover, the images of two vertices with labels written $F_1$ and $F_2$ via $\tau _{\kc}$ are distinct two curves such that they pass through $x$ and are closed fibers of the first and second projection $V_{\kc} \simeq \bP ^1_{\kc} \times \bP ^1_{\kc} \to \bP ^1_{\kc}$, say $F_1$ and $F_2$, respectively. 
Note that the union $F_1 + F_2$ is defined over $k$. 
Set $\wU := \wS \backslash \Supp (N+E)$. 
Then $\wU$ is certainly the cylinder on $\wS$ since $\wU \simeq V \backslash (F_1 \cup F_2) \simeq \bA ^2_k$ (see {\cite[Proposition 12]{DK18}}). 

{\it $n^{\circ} = 5^{\circ}$:} 
In this case, $V$ is a $k$-form of $\bP ^1_{\kc} \times \bP ^1_{\kc}$. 
Namely, $\Pic (V_{\kc}) = \bZ [F_1] \oplus \bZ [F_2]$, where $F_1$ and $F_2$ are general fibers of the first and second projection $V_{\kc} \simeq \bP ^1_{\kc} \times \bP ^1_{\kc} \to \bP ^1_{\kc}$, respectively. 
Meanwhile, the image of the vertex with a label written $C$ via $\tau _{\kc}$ is a geometrically rational curve, say $C$, with $C \sim F_1+F_2$. 
Note that $S$ has a singular point of type $A_1^+$ by the assumption. 
Hence, $\Supp (N)$ admits a $k$-rational point, say $\wtp$, so does $C$. 
By Lemma \ref{cylinder F0}, $V$ contains a cylinder, whose boundary includes $C$. 
We shall take $\wU$ to be the pullback of this cylinder by $\tau$. 
Then $\wU$ is a cylinder on $\wS$ satisfying $\wU \cap \Supp (N+E) = \emptyset$. 

{\it $n^{\circ} = 6^{\circ}$:} 
In this case, we see that $V \simeq \bF _2$ and the images of two vertices with labels written $M$ and $F$ via $\tau$ are the minimal section and a closed fiber of the $\bP ^1$-bundle $V \simeq \bF _2 \to \bP ^1_k$, say $M$ and $F$, respectively. 
Set $\wU := \wS \backslash \Supp (N+E)$. 
Then $\wU$ is certainly the cylinder on $\wS$ since $\wU \simeq V \backslash (M \cup F) \simeq \bA ^2_k$. 

{\it $n^{\circ} = 7^{\circ}$:} 
In this case, $V$ is a $k$-form of $\bF _2$. 
Note that $S_{\kc}$ has a singular point of type $A_2$, which is $k$-rational, by the configuration of curves on $\wS _{\kc}$ and assumption. 
Hence, $S$ has a $k$-rational point lying on $\Supp (N)$, so does $V$. 
Thus, $V \simeq \bF _2$ by using Lemma \ref{Severi-Brauer}. 
Meanwhile, the images of two vertices with labels written $M$ and $F$ are those as in $6^{\circ}$, say $M$ and $F$, moreover, the image of the vertex with a label written $C$ via $\tau$ is a rational curve on $V$, say $C$, with $C \sim M+2F$. 
By Lemma \ref{cylinder F2}, $V$ contains a cylinder, whose boundary includes $M \cup F \cup C$. 
We shall take $\wU$ to be the pullback of this cylinder by $\tau$. 
Then $\wU$ is a cylinder on $\wS$ satisfying $\wU \cap \Supp (N+E) = \emptyset$. 

{\it $n^{\circ} = 8^{\circ}$:} 
In this case, $V$ is a $k$-form of $\bF _2$. 
Note that $S$ has a singular point of type $A_1^+$ by the configuration of curves on $\wS _{\kc}$ and assumption. 
Thus, $V \simeq \bF _2$ by an argument similar to the case of $7^{\circ}$. 
Moreover, we can assume that the image of the vertex with a label written $M$ via $\sigma$ is a singular point of type $A_1^+$. 
Meanwhile, the images of two vertices with labels written $M$ and $C$ via $\tau$ are those as in $7^{\circ}$, say $M$ and $C$, respectively. 
Then $C$ admits a $k$-rational point. 
By Lemma \ref{cylinder F2}, $V$ contains a cylinder, whose boundary includes $M \cup C$. 
We shall take $\wU$ to be the pullback of this cylinder by $\tau$. 
Then $\wU$ is a cylinder on $\wS$ satisfying $\wU \cap \Supp (N+E) = \emptyset$. 

{\it $n^{\circ} = 9^{\circ}$:} 
In this case, $V=\wS$ and $V$ is a $k$-form of $\bF _2$. 
Hence, $\wS$ contains a cylinder $\wU$, so that $\wU \cap \Supp (N) = \emptyset$ (see {\cite[Corollary 4.5]{Saw}}). 

{\it $n^{\circ} = 10^{\circ}$:} 
In this case, $V$ is a $k$-form of $\bF _2$ and  the images of the vertices with labels written $M$ and $F_i$ via $\tau$ are a $k$-form of the minimal section and $k$-forms of closed fibers of the $\bP ^1$-bundle $V_{\kc} \simeq \bF _2 \to \bP ^1_{\kc}$, say $M$ and $F_i$, respectively. 
Then $V$ contains a cylinder, whose boundary includes $M \cup F_1 \cup F_2$ (see {\cite[Corollary 4.5]{Saw}}). 
We shall take $\wU$ to be the pullback of this cylinder by $\tau$. 
Then $\wU$ is a cylinder on $\wS$ satisfying $\wU \cap \Supp (N+E) = \emptyset$. 

For all cases, we obtain a cylinder $\wU $ on $\wS$ such that $\wU \cap \Supp (N) = \emptyset$. 
Therefore, $S$ contains the cylinder $\sigma (\wU ) \simeq \wU$. 
This completes the proof of Theorem \ref{main(1-1)} and the ``if" part in Theorem \ref{main(1-2)}. 
\begin{rem}\label{cylinder rem}
We shall state some remarks on the above argument. 
\begin{enumerate}
\item In these cases $n^{\circ} = 1^{\circ}$, $4^{\circ}$ or $6^{\circ}$, then $S$ always contains the affine plane $\bA ^2_k$ (compare the fact that the Du Val del Pezzo surface over $\bC$ with $\rho _{\bC}(S)=1$ and of degree $d \ge 3$ contains $\bC ^2$ if and only if the pair of the degree and the singularities of this surface is $(8,A_1)$, $(6,A_2+A_1)$, $(5,A_4)$, $(4,D_5)$ or $(3,E_6)$, see {\cite{MZ88}}). 
\item In these cases $n^{\circ} = 9^{\circ}$ or $10^{\circ}$, then $\wS$ need not to have a $k$-rational point, where note that $\wS$ has a $k$-rational point if and only if $V$ is a trivial $k$-form. 
However, $S$ always contains a cylinder, say $U \simeq \bA ^1_k \times Z$ (compare the fact that any smooth del Pezzo surface over $k$ with $\rho _k(S)=1$ containing a cylinder admits $k$-rational points, see {\cite{DK18}}). 
This implies that $Z$ need not be $k$-rational. 
\end{enumerate}
\end{rem}

\section{Divisors on weak del Pezzo surfaces}\label{4}

Let $S$ be a Du Val del Pezzo surface over $k$ and let $\sigma :\wS \to S$ be the minimal resolution over $k$, so that $\wS$ is a weak del Pezzo surface over $k$. 
In this section, we will study the property of some $\bQ$-divisors on $\wS _{\kc}$ and look for some $(-1)$-curves on $\wS _{\kc}$, which are $\bQ$-linearly equivalent to some divisors generated by the anti-canonical divisor and some $(-2)$-curves. 

\subsection{Properties of $\bQ$-divisors composed of $(-2)$-curves}\label{4-1}

In this subsection, let $x$ be a singular point of type $A_n$, $D_5$ or $E_6$ on $S_{\kc}$, which is $k$-rational, let $M_1,\dots ,M_n$ be the irreducible components of the exceptional set on $\wS _{\kc}$ by the minimal resolution at $x$ on $S_{\kc}$. 
Assume that the dual graph of $\sum _{j=1}^nM_j$ is the following graph according to the singularity type of $x$ on $S_{\kc}$: 
\begin{itemize}
\item Type $A_n$: 
\begin{align}\label{graphA}
\begin{split}
\xygraph{
     \circ ([]!{+(0,+.3)} {M_1}) 
 -[r] \circ ([]!{+(0,+.3)} {M_2}) 
 -[r] \cdots 
 -[r] \circ ([]!{+(0,+.3)} {M_n})}
\end{split}
\end{align}
\item Type $D_5$: 
\begin{align}\label{graphD}
\begin{split}
\xygraph{
	\circ ([]!{+(-.4,0)} {M_1})
 - []!{+(1,-.5)} \circ ([]!{+(0,+.3)} {M_3})
( -[r] \circ ([]!{+(0,+.3)} {M_4})
 -[r] \circ ([]!{+(0,+.3)} {M_5}) ,
 - []!{+(-1,-.5)} \circ ([]!{+(-.4,0)} {M_2}) )
}
\end{split}
\end{align}
\item Type $E_6$: 
\begin{align}\label{graphE}
\begin{split}
\xygraph{
	\circ ([]!{+(0,+.3)} {M_1})
 -[r] \circ ([]!{+(0,+.3)} {M_3})
 - []!{+(1,-.5)} \circ ([]!{+(0,+.3)} {M_5})
( -[r] \circ ([]!{+(0,+.3)} {M_6}) ,
 - []!{+(-1,-.5)} \circ ([]!{+(0,+.3)} {M_4}) -[l] \circ ([]!{+(0,+.3)} {M_2}) )
}
\end{split}
\end{align}
\end{itemize}
Let $M$ be a $\bQ$-divisor on $\wS _{\kc}$, which is generated by $M_1,\dots ,M_n$, so that: 
\begin{align*}
M = \sum _{j=1}^nb_jM_j
\end{align*}
for some $b_1,\dots ,b_n \in \bQ$. 
\begin{lem}\label{A-1}
Assume that the singular point $x$ is of type $A_n$ on $S_{\kc}$. 
Let $j_0$ be an integer with $1 \le j_0 \le n$. 
If $(-M \cdot M_j) = \delta _{j_0,j}$, then we have:   
\begin{align*}
M = \frac{n-j_0+1}{n+1} \sum _{j=1}^{j_0}jM_j + \frac{j_0}{n+1} \sum _{j=1}^{n-j_0}jM_{n-j+1}
\end{align*} 
and: 
\begin{align*}
(M)^2 = -\frac{(n-j_0+1)j_0}{n+1}. 
\end{align*}
\end{lem}
\begin{proof}
For all cases, we can easily show because it is enough to directly compute some intersection numbers. 
\end{proof}
In Lemma \ref{A-1}, if $(-M \cdot M_j) = \delta _{j_0,j}$, then the value of $(M)^2$ is explicitly summarized in Table \ref{A-list} depending on the values of $n$ and $j_0$: 
\begin{table}[htbp]
\caption{The value of $(M)^2$ in Lemma \ref{A-1}}\label{A-list} 
\begin{center}
\begin{tabular}{|c||c|c|c|c|c|c|c|c|} \hline

$n \backslash j_0$ & $1$ & $2$ & $3$ & $4$ & $5$ & $6$ & $7$ & $8$ \\ \hline \hline

$1$ & $-\frac{1}{2}$ & & & & & & & \\ \hline
$2$ & $-\frac{2}{3}$ & $-\frac{2}{3}$ & & & & & & \\ \hline
$3$ & $-\frac{3}{4}$ & $-1$ & $-\frac{3}{4}$ & & & & & \\ \hline
$4$ & $-\frac{4}{5}$ & $-\frac{6}{5}$ & $-\frac{6}{5}$ & $-\frac{4}{5}$ & & & & \\ \hline
$5$ & $-\frac{5}{6}$ & $-\frac{4}{3}$ & $-\frac{3}{2}$ & $-\frac{4}{3}$ & $-\frac{5}{6}$ & & & \\ \hline
$6$ & $-\frac{6}{7}$ & $-\frac{10}{7}$ & $-\frac{12}{7}$ & $-\frac{12}{7}$ & $-\frac{10}{7}$ & $-\frac{6}{7}$ & & \\ \hline
$7$ & $-\frac{7}{8}$ & $-\frac{3}{2}$ & $-\frac{15}{8}$ & $-2$ & $-\frac{15}{8}$ & $-\frac{3}{2}$ & $-\frac{7}{8}$ & \\ \hline
$8$ & $-\frac{8}{9}$ & $-\frac{14}{9}$ & $-2$ & $-\frac{20}{9}$ & $-\frac{20}{9}$ & $-2$ & $-\frac{14}{9}$ & $-\frac{8}{9}$ \\ \hline
\end{tabular}
\end{center}
\end{table}
\begin{lem}\label{D-1}
Assume that the singular point $x$ is of type $D_5$ on $S_{\kc}$. 
\begin{enumerate}
\item If $(-M \cdot M_j) = \delta _{1,j} + \delta _{2,j}$, then we have: 
\begin{align*}
M = 2M_1+2M_2+3M_3+2M_4+M_5
\end{align*} 
and $(M)^2 = -4$.
\item If $(-M \cdot M_j) = \delta _{1,j}$, then we have: 
\begin{align*}
M = \frac{5}{4}M_1+\frac{3}{4}M_2+\frac{3}{2}M_3+M_4+\frac{1}{2}M_5
\end{align*}
and $(M)^2 = -\frac{5}{4}$. 
\item If $(-M \cdot M_j) = \delta _{3,j}$, then we have: 
\begin{align*}
M = \frac{3}{2}M_1+\frac{3}{2}M_2+3M_3+2M_4+M_5
\end{align*}
and $(M)^2 = -3$. 
\item If $(-M \cdot M_j) = \delta _{4,j}$, then we have: 
\begin{align*}
M = M_1+M_2+2M_3+2M_4+M_5
\end{align*}
and $(M)^2 = -2$. 
\item If $(-M \cdot M_j) = \delta _{5,j}$, then we have: 
\begin{align*}
M = \frac{1}{2}M_1+\frac{1}{2}M_2+M_3+M_4+M_5
\end{align*}
and $(M)^2 = -1$. 
\end{enumerate}
\end{lem}
\begin{proof}
For all cases, we can easily show because it is enough to directly compute some intersection numbers. 
\end{proof}
\begin{lem}\label{E-1}
Assume that the singular point $x$ is of type $E_6$ on $S_{\kc}$. 
\begin{enumerate}
\item If $(-M \cdot M_j) = \delta _{1,j} + \delta _{2,j}$, then we have: 
\begin{align*}
M = 2M_1+2M_2+3M_3+3M_4+4M_5+M_6
\end{align*} 
and $(M)^2 = -4$. 
\item If $(-M \cdot M_j) = \delta _{3,j} + \delta _{4,j}$, then we have: 
\begin{align*}
M = 3M_1+3M_2+6M_3+6M_4+8M_5+4M_6
\end{align*} 
and $(M)^2 = -12$. 
\item If $(-M \cdot M_j) = \delta _{1,j}$, then we have: 
\begin{align*}
M = \frac{4}{3}M_1+\frac{2}{3}M_2+\frac{5}{3}M_3+\frac{4}{3}M_4+2M_5+M_6
\end{align*} 
and $(M)^2 = -\frac{4}{3}$. 
\item If $(-M \cdot M_j) = \delta _{3,j}$, then we have: 
\begin{align*}
M = \frac{5}{3}M_1+\frac{4}{3}M_2+\frac{10}{3}M_3+\frac{8}{3}M_4+4M_5+2M_6
\end{align*} 
and $(M)^2 = -\frac{10}{3}$. 
\item If $(-M \cdot M_j) = \delta _{5,j}$, then we have: 
\begin{align*}
M =2M_1+2M_2+4M_3+4M_4+6M_5+3M_6
\end{align*} 
and $(M)^2 = -6$. 
\item If $(-M \cdot M_j) = \delta _{6,j}$, then we have: 
\begin{align*}
M =M_1+M_2+2M_3+2M_4+3M_5+2M_6
\end{align*} 
and $(M)^2 = -2$. 
\end{enumerate}
\end{lem}
\begin{proof}
For all cases, we can easily show because it is enough to directly compute some intersection numbers. 
\end{proof}
\begin{lem}\label{ADE-(-1)}
Assume that $(-K_{\wS})^2 = 1$ and one of the following conditions holds: 
\begin{enumerate}
\item $M = \sum _{j=1}^nM_j$ and the dual graph of $M$ is the same as in (\ref{graphA}). 
\item $n=5$, $M = M_1+M_2+2M_3+2M_4+M_5$ and the dual graph of $M$ is the same as in (\ref{graphD}). 
\item $n=6$, $M = M_1+M_2+2M_3+2M_4+3M_5+2M_6$ and the dual graph of $M$ is the same as in (\ref{graphE}). 
\end{enumerate}
Then there exists a $(-1)$-curve $E$ on $\wS _{\kc}$ such that $E \sim -K_{\wS _{\kc}} - M$ and $E$ is defined over $k$. 
\end{lem}
\begin{proof}
It is easily seen that $(-K_{\wS _{\kc}} - M )^2 = -1$ and $(-K_{\wS _{\kc}} - M \cdot -K_{\wS _{\kc}}) = 1$. 
Moreover, in the cases of (1), (2) and (3), we obtain $(-K_{\wS _{\kc}} - M \cdot M_j) = \delta _{j,1} + \delta _{j,n}$, $(-K_{\wS _{\kc}} - M \cdot M_j) = \delta _{j,4}$ and $(-K_{\wS _{\kc}} - M \cdot M_j) = \delta _{j,6}$, respectively (cf. Lemmas \ref{A-1}, \ref{D-1}(4) and \ref{E-1}(6)). 
Meanwhile, $(-K_{\wS _{\kc}} - M \cdot M')=0$ for every $(-2)$-curve $M'$ on $\wS _{\kc}$ other than the irreducible components of $M$. 
Hence, by Lemma \ref{(-1)-curve}, there exists a $(-1)$-curve $E$ on $\wS _{\kc}$ such that $E \sim -K_{\wS _{\kc}} - M$. 
In what follows, we shall show that $E$ is defined over $k$. 
Suppose that there exists an irreducible curve $E'$ other than $E$ on $\wS _{\kc}$ such that $E$ and $E'$ lie in the same $\Gal$-orbit. 
Then $E' \sim -K_{\wS _{\kc}} -M$ since the divisor $-K_{\wS _{\kc}} - M$ is defined over $k$ by the configuration of irreducible components of $M$. 
Hence, we have $0 \le (E \cdot E') =(-K_{\wS _{\kc}}-M)^2=-1$, which is absurd. 
Thus, $E$ must be defined over $k$. 
This completes the proof. 
\end{proof}
\begin{lem}\label{ADE-1}
Assume that $b_1,\dots ,b_n \in \bZ$. 
Then the following assertions hold: 
\begin{enumerate}
\item $(M)^2$ is a non-positive even integer. 
\item If the singular point $x$ is of type $A_n$ on $S$ and $b_j \ge 1$ for any $j$, then $(M)^2 \le -2$, moreover, $(M)^2 = -2$ if and only if $b_j=1$ for any $j=1,\dots ,n$. 
\item If the singular point $x$ is of type $A_n$ on $S$ with $n \ge 3$, $b_1,b_n \ge 1$ and $b_j \ge 2$ for any $j=2,\dots ,n-1$, then $(M)^2 \le -4$, moreover, $(M)^2 = -4$ if and only if $b_1,b_n = 1$ and $b_j = 2$ for any $j=2,\dots ,n-1$. 
\item If the singular point $x$ is of type $A_n$ on $S$ with $n \ge 5$, $b_1,b_n \ge 1$, $b_2,b_{n-1} \ge 2$ and $b_j \ge 3$ for any $j=3,\dots ,n-2$, then $(M)^2 \le -6$, moreover, $(M)^2 = -6$ if and only if $b_1,b_n = 1$, $b_2,b_{n-1} = 2$ and $b_j = 3$ for any $j=3,\dots ,n-2$.  
\item If the singular point $x$ is of type $D_5$ on $S_{\kc}$ and $b_1,b_2,b_4 \ge 2$, $b_3 \ge 3$ and $b_5 \ge 1$, then $(M)^2 \le -4$, moreover, $(M)^2 = -4$ if and only if $b_1,b_2,b_4=2$, $b_3=3$ and $b_5 =1$. 
\item If the singular point $x$ is of type $E_6$ on $S_{\kc}$ and $b_1,b_2,b_6 \ge 2$, $b_3,b_4 \ge 3$ and $b_5 \ge 4$, then $(M)^2 \le -4$, moreover, $(M)^2 = -4$ if and only if $b_1,b_2,b_6=2$, $b_3,b_4=3$ and $b_5 =4$. 
\end{enumerate}
\end{lem}
\begin{proof}
In (1), since any irreducible component of $M$ is a $(-2)$-curve and any coefficient of $M$ is an integer, it is clear that $(M)^2$ is an even number. 
We shall show that $(M)^2$ is non-positive according to the singularity type of $x$ on $S_{\kc}$: 
\begin{itemize}
\item {\it Type $A_n$:} we have: 
\begin{align}\label{A-self}
(M)^2 = -(b_1^2+b_n^2) - \sum _{j=1}^{n-1}(b_j-b_{j+1})^2. 
\end{align}
\item {\it Type $D_5$:} we have: 
\begin{align}\label{D-self}
(M)^2 = -\frac{1}{2}(2b_1-b_3)^2 -\frac{1}{2}(2b_2-b_3)^2 -(b_3-b_4)^2 -(b_4-b_5)^2 -b_5^2. 
\end{align}
\item {\it Type $E_6$:} we have: 
\begin{align}\label{E-self}
\begin{aligned}
(M)^2 = -\frac{1}{2}(2b_1-b_2)^2 &-\frac{1}{2}(2b_2-b_4)^2 \\
&-\frac{1}{6}(3b_3-2b_5)^2  -\frac{1}{6}(3b_4-2b_5)^2 -\frac{1}{6}(2b_5-3b_6)^2-\frac{1}{2}b_6^2. 
\end{aligned}
\end{align}
\end{itemize}
Therefore, for all cases, we see that $(M)^2$ is non-positive. 
This completes the proof of (1). 

In (2), (3) and (4), it is easy to show by (\ref{A-self}). 

In (5), if $b_5 > 1$ then it is easy to see $(M)^2 <-4$ by assumption and (\ref{D-self}). 
Hence, we assume $b_5=1$ in what follows. 
Now, if $b_4 >2$, then we also see $(M)^2 <-4$ by an argument similar to the above. 
Hence, we also assume $b_4=2$ in what follows. 
By sequentially replacing $b_4$ in the argument by $b_3$, $b_2$ and $b_1$, we obtain the assertion. 

In (6), it can be shown by an argument similar to (5) using (\ref{E-self}) instead of (\ref{D-self}). 
\end{proof}

\subsection{Construction of $(-1)$-curves on weak del Pezzo surface}\label{4-2}

In this subsection, let $d$ be the degree of $\wS$, let $x_1,\dots ,x_{r'}$ be all singular points on $S_{\kc}$ let $M_{i,1},\dots ,M_{i,n(i)}$ be all irreducible components of the exceptional set $\sigma ^{-1}(x_i)$ for $i=1,\dots ,r'$. 
Here, we assume that $x_1 \in S_{\kc}$ is of type $A_{n(1)}$ with $n(1) \ge 2$ (resp. either $x_1 \in S_{\kc}$ is of type $A_{n(1)}$ with $n(1) \ge 4$ or of type $D_5$ or $E_6$) if $d=2$ (resp. $d=1$). 
Moreover, letting $r$ be a positive integer with $r \le r'$, we also assume that the dual graph of $\sum _{i=1}^r\sum_{j=1}^{n(i)}M_{i,j}$ is one of the following graphs (\ref{graphAA}), (\ref{graphDA}) and (\ref{graphEA}): 
\begin{align}\label{graphAA}
\begin{split}
\xygraph{
     \circ ([]!{+(0,+.3)} {M_{i,1}}) 
 -[r] \circ ([]!{+(0,+.3)} {M_{i,2}}) 
 -[r] \cdots 
 -[r] \circ ([]!{+(0,+.3)} {M_{i,n(i)}})}\quad \text{for}\ i=1,\dots ,r 
\end{split}
\end{align}
\begin{align}\label{graphDA}
\begin{split}
\xygraph{ !~:{@{}}
	\circ ([]!{+(-.5,0)} {M_{1,1}})
 - []!{+(1,-.5)} \circ ([]!{+(0,+.3)} {M_{1,3}})
( -[r] \circ ([]!{+(0,+.3)} {M_{1,4}})
 -[r] \circ ([]!{+(0,+.3)} {M_{1,5}})
 : [r] \circ ([]!{+(0,+.3)} {M_{2,1}}) 
 -[r] \circ ([]!{+(0,+.3)} {M_{2,2}}) 
 -[r] \cdots 
 -[r] \circ ([]!{+(0,+.3)} {M_{2,n}}) ,
 - []!{+(-1,-.5)} \circ ([]!{+(-.5,0)} {M_{1,2}}) )
}
\end{split}
\end{align}
\begin{align}\label{graphEA}
\begin{split}
\xygraph{ !~:{@{}}
	\circ ([]!{+(0,+.3)} {M_{1,1}})
 -[r] \circ ([]!{+(0,+.3)} {M_{1,3}})
 - []!{+(1,-.5)} \circ ([]!{+(0,+.3)} {M_{1,5}})
( -[r] \circ ([]!{+(0,+.3)} {M_{1,6}})
 : [r] \circ ([]!{+(0,+.3)} {M_{2,1}}) 
 -[r] \circ ([]!{+(0,+.3)} {M_{2,2}}) 
 -[r] \cdots 
 -[r] \circ ([]!{+(0,+.3)} {M_{2,n}}),
 - []!{+(-1,-.5)} \circ ([]!{+(0,+.3)} {M_{1,4}}) -[l] \circ ([]!{+(0,+.3)} {M_{1,2}}) )}
\end{split}
\end{align}
Here, in (\ref{graphAA}), we shall assume $(d,r) = (2,2),(2,1),(1,3),(1,2)$ or $(1,1)$. 
Furthermore, in (\ref{graphDA}) (resp. (\ref{graphEA})), we immediately obtain $r=2$ and $n(1)=5$ (resp. $n(1)=6$) by the configuration of curves, moreover, we shall assume $d=1$ and put $n(2):=n$. 

Let $D$ be the divisor on $\wS _{\kc}$ given by one of the lists in Table \ref{div D} according to the above cases of the dual graph and the pair $(d,r)$. 
\begin{table}[t]
\caption{Divisor $D$ in \S \S \ref{4-2}}\label{div D}
\begin{center}
\begin{tabular}{|c|c|c||l|} \hline
Case & $d$ & $r$ & \multicolumn{1}{|c|}{Irreducible decomposition of $D$} \\ \hline \hline

(a) & $2$ & $2$ & $(-K_{\wS _{\kc}}) - \sum _{i=1}^2 \sum _{j=1}^{n(i)}M_{i,j}$ \\ \hline
(b) & $2$ & $1$ & $(-K_{\wS _{\kc}}) + (M_{1,1} + M_{1,n(1)}) - 2\sum _{j=1}^{n(1)}M_{1,j}$.  \\ \hline
(c) & $1$ & $3$ & $2(-K_{\wS _{\kc}}) - \sum _{i=1}^3 \sum _{j=1}^{n(i)}M_{i,j}$ \\ \hline
(d) & $1$ & $2$ & $2(-K_{\wS _{\kc}}) + (M_{1,1} + M_{1,n(1)}) - 2\sum _{j=1}^{n(1)}M_{1,j} - \sum _{j=1}^{n(2)}M_{2,j}$ \\ \hline
(e) & $1$ & $1$ & $2(-K_{\wS _{\kc}}) + 2(M_{1,1} + M_{1,n(1)}) + (M_{1,2} + M_{1,n(1)-1}) - 3\sum _{j=1}^{n(1)}M_{1,j}$ \\ \hline
(f) & $1$ & $2$ & $2(-K_{\wS _{\kc}}) -(2M_{1,1}+2M_{1,2}+3M_{1,3}+2M_{1,4}+M_{1,5})-\sum_{j=1}^nM_{2,j}$ \\ \hline
(g) & $1$ & $2$ & $2(-K_{\wS _{\kc}}) -(2M_{1,1}+2M_{1,2}+3M_{1,3}+3M_{1,4}+4M_{1,5}+M_{1,6})-\sum_{j=1}^nM_{2,j}$ \\ \hline
\end{tabular}
\end{center}
\end{table}
Here, the dual graph of $\sum _{i=1}^r\sum_{j=1}^{n(i)}M_{i,j}$ is as in (\ref{graphAA}) (resp. (\ref{graphDA}), (\ref{graphEA})) if the case of $D$ is either (a), (b), (c), (d) or (e) (resp. (f), (g)). 
Moreover, we assume $n(1) \ge 4$ (resp. $n(1) \ge 6$) if the case of $D$ is either (b) or (d) (resp. (e)). 

For all cases, we see $(D)^2 =-2$ and $(D \cdot -K_{\wS _{\kc}}) =2$ by construction, moreover, we have the value of $(D 
\cdot M_{i,j})$, which is the following according to the cases: 
\begin{enumerate}
\item[{(a):}] $(D \cdot M_{i,j})= \delta _{j,1} + \delta _{j,n(i)}$ for $i=1,2$. 
\item[{(b):}] $(D \cdot M_{i,j})= \delta _{j,2} + \delta _{j,n(1)-1}$. 
\item[{(c):}] $(D \cdot M_{i,j})= \delta _{j,1} + \delta _{j,n(i)}$ for $i=1,2,3$. 
\item[{(d):}] $(D \cdot M_{i,j})= \delta _{i,1}(\delta _{j,2} + \delta _{j,n(1)-1}) + \delta _{i,2}(\delta _{j,1} + \delta _{j,n(2)})$ for $i=1,2$. 
\item[{(e):}] $(D \cdot M_{i,j})= \delta _{j,3} + \delta _{j,n(1)-2}$. 
\item[{(f):}] $(D \cdot M_{i,j})= \delta _{i,1}(\delta _{j,1} + \delta _{j,2}) + \delta _{i,2}(\delta _{j,1} + \delta _{j,n})$ for $i=1,2$. 
\item[{(g):}] $(D \cdot M_{i,j})= \delta _{i,1}(\delta _{j,1} + \delta _{j,2}) + \delta _{i,2}(\delta _{j,1} + \delta _{j,n})$ for $i=1,2$. 
\end{enumerate}
The purpose of this subsection is that we show Proposition \ref{ADE-prop}. 
For the following two lemmas, we only treat the case {\rm (a)} since other cases can be shown by a similar argument. 
\begin{lem}\label{A-lem(1)}
$\dim |D| \ge 0$. 
\end{lem}
\begin{proof}
By the Riemann-Roch theorem and $(D \cdot D -K_{\wS _{\kc}})=0$, we have $\chi (\wS _{\kc} , \sO _{\wS _{\kc}}(D)) = \chi (\wS _{\kc} , \sO _{\wS _{\kc}})$. 
Moreover, by the Serre duality theorem and the rationality of $\wS _{\kc}$, we obtain $h^2(\wS _{\kc} , \sO _{\wS _{\kc}}(D)) = h^0(\wS _{\kc} , \sO _{\wS _{\kc}}(K_{\wS _{\kc}} -D)) = 0$. 
Thus, we have $\dim |D| = h^0(\wS _{\kc} , \sO _{\wS _{\kc}}(D)) -1 \ge \chi (\wS _{\kc} , \sO _{\wS _{\kc}}(D)) -1 = \chi (\wS _{\kc} , \sO _{\wS _{\kc}})-1 = 0$ because of the rationality of $\wS _{\kc}$. 
\end{proof}
By Lemma \ref{A-lem(1)}, there exist two effective divisors $D^{(1)}$ and $D^{(2)}$ on $\wS _{\kc}$ such that $D \sim D^{(1)}+D^{(2)}$
and  each irreducible component $C_1$ (resp. $C_2$) of $D^{(1)}$ (resp. $D^{(2)}$) satisfies $(C_1 \cdot -K_{\wS _{\kc}})>0$ (resp. $(C_2 \cdot -K_{\wS _{\kc}})=0$). 
Note that $D^{(2)}$ is an effective divisor, which consists of $(-2)$-curves on $\wS _{\kc}$, since $\wS _{\kc}$ is a weak del Pezzo surface. 
\begin{lem}\label{A-lem(2)}
$(D^{(1)})^2 \le -2$. 
\end{lem}
\begin{proof}
By $D^{(1)} \sim D - D^{(2)}$, we can write $D^{(1)} \sim (-K_{\wS _{\kc}}) - \sum _{i=1}^2 \sum _{j=1}^{n(i)} b_{i,j}M_{i,j} -M'$, where $b_{i,j}$ is an integer for $i=1,2$ and $j=1,\dots ,n(i)$, and $M'$ is an effective divisor consisting of $(-2)$-curves $\{ M_{i,j}\} _{r < i \le r',\ 1 \le j \le n(i)}$.  
By $D^{(2)} \sim D -D^{(1)}$, we have $D^{(2)} \sim \sum _{i=1}^2 \sum _{j=1}^{n(i)} (b_{i,j}-1)M_{i,j}+M'$. 
Hence, we see $b_{i,j} \ge 1$ for $i=1,2$ and $j=1,\dots ,n(i)$ since $D^{(2)}$ is an effective divisor. 
Thus, we obtain $(D^{(1)})^2 \le (-K_{\wS _{\kc}})^2+\sum _{i=1}^2\left( \sum _{j=1}^{n(i)}b_{i,j}M_{i,j} \right) ^2\le 2 +2 \cdot (-2) = -2$ by Lemma \ref{ADE-1}(2). 
\end{proof}
\begin{rem}\label{rem of A-lem(2)}
The proof of Lemma \ref{A-lem(2)} uses Lemma \ref{ADE-1}(2). 
On the other hand, in the case of (b) (resp. (d), (e), (f), (g)), we should use Lemma \ref{ADE-1}(3) (resp. both (2) and (3), (4), both (2) and (5), both (2) and (6)) instead of Lemma \ref{ADE-1}(2). 
\end{rem}
The following proposition is the main result in \S \ref{4} and will play an important role in \S \ref{5}: 
\begin{prop}\label{ADE-prop}
With the notation as above, the following assertions hold: 
\begin{enumerate}
\item $D$ satisfies one of the following two conditions: 
\begin{enumerate}
\item[{\rm (A)}:] There exist two $(-1)$-curves $E_1$ and $E_2$ on $\wS _{\kc}$ satisfying $D^{(1)} =E_1+E_2$ and $(E_1 \cdot E_2) = 0$. 
\item[{\rm (B)}:] There exists a $(-1)$-curve $E$ on $\wS _{\kc}$ satisfying $D^{(1)} = 2E$. 
\end{enumerate}
\item If $D$ satisfies the condition {\rm (A)}, then we have $D \sim D^{(1)}$. 
\item We write $D^{(1)} \sim \frac{2}{d}(-K_{\wS _{\kc}}) - \sum _{i=1}^{r'}\sum _{j=1}^{n(i)}b_{i,j}M_{i,j}$, where each $b_{i,j}$ is a non-negative integer. 
Then: 
\begin{itemize}
\item For any $i \le r$, $b_{i,j} \not= 0$ for some $j$. 
\item For each $i$, if $b_{i,j} \not= 0$ for some $j$, then $(E \cdot M_{i,1}+\dots +M_{i,n(i)})=1$. 
\end{itemize}
\item If $D$ is of the case ${\rm (f)}$ or ${\rm (g)}$, then $D$ satisfies the condition ${\rm (A)}$. 
\item Assume that $D$ satisfies {\rm (B)}, and write $E \sim _{\bQ} \frac{1}{d}(-K_{\wS _{\kc}}) - \sum _{i=1}^{r'}M_i$, where $M_i$ is an effective $\bQ$-divisor consisting of $M_{i,1},\dots ,M_{i,n(i)}$. 
Letting $s$ be the number of $\bQ$-divisors $M_i$ as $M_i \not= 0$, then $s \le 2$. 
Hence, if $D$ is of the case ${\rm (c)}$, then $D$ satisfies the condition ${\rm (A)}$. 
\item Assume that $D$ satisfies the condition {\rm (A)}. If any irreducible component $E$ of $D^{(1)}$ is contained in $\bQ [-K_{\wS _{\kc}}] \oplus \left( \bigoplus _{i=1}^{r} \bigoplus _{j=1}^{n(i)} \bQ [M_{i,j}] \right)$, then each $n(i)$ is one of the following according to the case of $D$: 
\begin{itemize}
\item In the case of {\rm (a)}, then $\{ n(1),n(2)\} = \{ 5,2\}$ or $\{ 3,3\}$. 
\item In the case of {\rm (b)}, then $n(1)=7$. 
\item In the case of {\rm (c)}, then $n(1)=5$ and $\{ n(2),n(3)\} = \{ 2,1\}$. 
\item In the case of {\rm (d)}, then $(n(1),n(2)) = (7,1)$, $(5,2)$ or $(4,4)$. 
\item In the case of {\rm (e)}, then $n(1)=8$. 
\item In the case of {\rm (f)}, then $n(2)=3$ (it is clear that $n(1)=5$). 
\item In the case of {\rm (g)}, then $n(2)=2$ (it is clear that $n(1)=6$). 
\end{itemize}
\item Assume that $D$ satisfies the condition {\rm (B)}. If the case of $D$ is {\rm (a)} or {\rm (d)}, i.e., $r=2$, then the irreducible component $E$ of $D^{(1)}$ is contained in $\bQ [-K_{\wS _{\kc}}] \oplus \left( \bigoplus _{i=1}^{2} \bigoplus _{j=1}^{n(i)} \bQ [M_{i,j}] \right)$ and each $n(i)$ is as follows according to the case of $D$: 
\begin{itemize}
\item In the case of {\rm (a)}, then $(n(1),n(2))=(3,1)$. 
\item In the case of {\rm (d)}, then $(n(1),n(2))=(5,1)$. 
\end{itemize}
\item Assume that $D$ satisfies the condition {\rm (B)}. If the case of $D$ is {\rm (b)} or {\rm (e)}, i.e., $r=1$, and the irreducible component $E$ of $D^{(1)}$ is contained in $\bQ [-K_{\wS _{\kc}}] \oplus \left( \bigoplus _{j=1}^{n(1)} \bQ [M_{1,j}] \right)$, then each $n(i)$ is as follows according to the case of $D$: 
\begin{itemize}
\item In the case of {\rm (b)}, then $n(1)=5$. 
\item In the case of {\rm (e)}, then $n(1)=7$. 
\end{itemize}
\end{enumerate}
\end{prop}
\begin{proof}
In (1), note that $D^{(1)}$ consists of at most two irreducible components by $(D^{(1)} \cdot -K_{\wS _{\kc}})=2$. 
Hence, we see that this assertion follows from Lemma \ref{A-lem(2)}. 

In (2), assuming that $D$ satisfies the condition {\rm (A)}, we have $(D^{(1)})^2 = -2$. 
Hence, we see that this assertion follows from Lemma \ref{ADE-1} according to the case of $D$ (cf. Remark \ref{rem of A-lem(2)}). 

In (3), this proof is a bit long and is needed a technical argument. 
Hence, we will present this proof in \S \S \ref{4-3}. 

In what follows, we present the proof under the assumption that (3) is valid. 

In (4), we only treat the case where $D$ is of ${\rm (f)}$, the other cases are similar and left to the reader.
Suppose on the contrary that $D$ satisfies the condition ${\rm (B)}$. 
In other words, there exists a $(-1)$-curve $E$ on $\wS _{\kc}$ such that $D^{(1)} = 2E$. 
Then by (3) there uniquely exists $j' \in \{ 1,\dots ,5\}$ and $j'' \in \{ 1,\dots ,n(2)\}$ such that $(E \cdot M_{1,j}) = \delta _{j,j'}$ and $(E \cdot M_{2,j}) = \delta _{j,j''}$, respectively. 
Since $D^{(1)}$ is a $\bZ$-divisor, $j' \not= 1,2$ by Lemma \ref{D-1}. 
(Note that we shall use Lemma \ref{E-1}, when we treat the case ${\rm (g)}$ instead of the case $(f)$. )
On the other hand, we write $E \sim _{\bQ} (-K_{\wS _{\kc}}) - \sum _{i=1}^{r'}M_i$, where $M_i$ is an effective $\bQ$-divisor consisting of $M_{i,1},\dots ,M_{i,n(i)}$. 
Then $(M_i)^2 \le 0$ by using Lemma \ref{ADE-1}(1), moreover, $(M_1)^2 \le -2$ and $(M_2)^2<0$ by Lemmas \ref{A-1} and \ref{D-1}. 
Hence, we have $-1 = (E)^2  < 1 + (-2) = -1$, which is absurd. 

In (5), by the assumption of $E$, we have: 
\begin{align}\label{(5)-1}
-1 = (E)^2 = \frac{1}{d} + \sum _{i=1}^{r'}(M_i)^2
\end{align}
Here, if $M_i \not= 0$, we see $(M_i)^2 \le -\frac{1}{2}$ by (3) and Lemmas \ref{A-1}, \ref{D-1} and \ref{E-1} (see also Table \ref{A-list}). 
Furthermore, $(M_1)^2 \le -\frac{2}{3}$ by virtue of $n(1)>1$. 
Hence, we have: 
\begin{align}\label{(5)-2}
\frac{1}{d} + \sum _{i=1}^{r'}(M_i)^2 \le \frac{1}{d} - \frac{2}{3} -(s-1)\cdot \frac{1}{2}
\end{align}
Two formulas (\ref{(5)-1}) and (\ref{(5)-2}) imply $s \le \frac{2}{d}+\frac{5}{3}$. 
Since $s$ is an integer, we thus obtain $s \le 2$ and $s \le 3$ if $d=2$ and $d=1$, respectively. 
In what follows, we consider the case $d=1$ and suppose $s=3$. 
Then we may assume $M_i \not= 0$ for $i=1,2,3$. 
Notice that each singularity on $S_{\kc}$ corresponding to $\sum _{j=1}^{n(i)}M_{i,j}$ is of type $A_{n(i)}$ for $i=1,2,3$ by virtue of (1) and (4), moreover, note $n(1) \ge 4$. 
By looking for the triplet $\{(M_1)^2,(M_2)^2,(M_3)^2\}$ with $(M_1)^2+(M_2)^2+(M_3)^2 = -2$ in Table \ref{A-list}, the triplet is only $\{ -\frac{5}{6},-\frac{2}{3},-\frac{1}{2}\}$, moreover, $n(1)=5$ and $\{ n(2),n(3)\}=\{2,1\}$. 
Hence, we may assume: 
\begin{align*}
E \sim _{\bQ} (-K_{\wS _{\kc}}) - \sum _{j=1}^5\frac{6-j}{6}M_{1,j} - \sum _{j=1}^2\frac{3-j}{3}M_{2,j} - \frac{1}{2}M_{3,1}. 
\end{align*}
However, this contradicts that $D^{(1)}=2E$ is a $\bZ$-divisor. 

In (6), assume that $D$ satisfies the condition {\rm (A)} and $E \in \bQ [-K_{\wS _{\kc}}] \oplus \left( \bigoplus _{i=1}^{r} \bigoplus _{j=1}^{n(i)} \bQ [M_{i,j}] \right)$. 
Hence, we can write $E \sim _{\bQ} \frac{1}{2}(-K_{\wS _{\kc}}) -\sum _{i=1}^rM_i$ by noticing $(E \cdot -K_{\wS _{\kc}}) = 1$, where $M_i$ is an effective $\bQ$-divisor generated by $M_{i,1},\dots ,M_{i,n(i)}$ for $i=1,\dots ,r$. 
Then we have: 
\begin{align}\label{(A)}
-1 = (E)^2 = \frac{1}{d} + \sum _{i=1}^{r}(M_i)^2. 
\end{align}
We shall look for the combination of the values of $(M_1)^2,\dots ,(M_{r})^2$ such that the equality (\ref{(A)}) holds, using directly Table \ref{A-list} and Lemmas \ref{D-1} and \ref{E-1} according to each case. 
As an example, we will explain the case of (a). 
Note that the equality (\ref{(A)}) means $(M_1)^2 + (M_2)^2 = -\frac{3}{2}$ by $d=2$ and $r=2$. 
Since $(D \cdot M_{i,j})=\delta _{j,1}+\delta_{j,n(i)}$, we may assume that $(E_1 \cdot M_{i,j})=\delta _{j,1}$ for $i=1,2$ by virtue of (2) and (3). 
Hence, we shall look at the row of $j_0=1$ in Table \ref{A-list}. 
Then it is easy to see that the equality (\ref{(A)}) holds if and only if $\{ (M_1)^2,(M_2)^2\} =\{ -\frac{5}{6},-\frac{2}{3}\}$ or $\{ -\frac{3}{4},-\frac{3}{4}\}$. 
This means that $\{ n(1),n(2)\} =\{ 5,2\}$ or $\{ 3,3\}$ by Table \ref{A-list}. 
The other cases are left to the reader because these can be shown by an argument similar to the above argument. 

In (7), assume that $D$ satisfies the condition {\rm (B)}, in other words, there exists a $(-1)$-curve $E$ on $\wS _{\kc}$ such that $D^{(1)} = 2E$. 
Then $E \in \bQ [-K_{\wS _{\kc}}] \oplus \left( \bigoplus _{i=1}^2\bigoplus _{j=1}^{n(i)}\bQ [M_{i,j}] \right)$ by virtue of (3) and (5). 
In particular, we write $E \sim _{\bQ} (-K_{\wS _{\kc}}) - \sum _{i=1}^2M_i$, where $M_i$ is an effective $\bQ$-divisor generated by $M_{i,1},\dots ,M_{i,n(i)}$. 
Hence, we have: 
\begin{align}\label{(B)}
-1 = (E)^2 = \frac{1}{d} + \sum _{i=1}^2(M_i)^2. 
\end{align}
We shall look for the combination of the values of $(M_1)^2$ and $(M_2)^2$ such that the equality (\ref{(B)}) holds and $|D-2E| \not= \emptyset$, according to each case. 
However, this argument can be shown by an argument similar to (6) and is left to the reader. 

In (8), this proof can be shown by an argument similar to (7) and is left to the reader. 
\end{proof}
Now, we shall present the following example about the application of Proposition \ref{ADE-prop}: 
\begin{eg}\label{ex of prop(4-2)}
With the notation as above, assume further that $d=2$ and $\wS$ is of $A_5+A_2$-type. 
Let $M_{1,1},\dots ,M_{1,5}$, $M_{2,1}$ and $M_{2,2}$ be all $(-2)$-curves on $\wS _{\kc}$ with the configuration as in (\ref{graphAA}). 
Then we shall consider two divisors $D_{1,5}$ and $D_3$ on $\wS _{\kc}$ given by: 
\begin{align*}
D_{1,5} &:= -K_{\wS _{\kc}} - \sum _{j=1}^5M_{1,j} - \sum _{j=1}^2M_{2,j},\\
D_3 &:= -K_{\wS _{\kc}} -M_{1,1}-2(M_{1,2}+M_{1,3}+M_{1,4})-M_{1,5}. 
\end{align*}
Notice that $D_{1,5}$ and $D_3$ are divisors as in ${\rm (a)}$ and ${\rm (b)}$ in Table \ref{div D}, respectively. 
Hence, since $D_{1,5}$ satisfies the condition ${\rm (A)}$ by Proposition \ref{ADE-prop}(1),(2) and (7), there exist two $(-1)$-curves $E_1$ and $E_5$ on $\wS _{\kc}$ such that $D_{1,5} \sim D_{1,5}^{(1)}:=E_1+E_5$. 
Moreover, $D_3$ satisfies the condition either ${\rm (A)}$ or ${\rm (B)}$. 
However, $D_3$ does not satisfy the condition ${\rm (A)}$. 
Indeed, otherwise, since there exist two $(-1)$-curves $E_2$ and $E_4$ on $\wS _{\kc}$ such that $D_3 \sim D_3^{(1)}:=E_2+E_4$. 
Hence, we obtain the compositions $\tau :\wS _{\kc} \to V$ of successive contractions of $E_2+E_4$, that of the images of $M_{1,2}+M_{1,4}$ and finally that of the images of $M_{1,1}+M_{1,5}$ over $\kc$, so that the weighted dual graphs of $\sum _{j=1}^5M_{1,j}+\sum _{j=1}^2M_{2,j}+D_{1,5}^{(1)}+D_3^{(1)}$ and its image via $\tau$ are as follows, where ``$\circ$", ``$\bullet$" and ``$\diamond$" mean a $(-2)$-curve, a $(-1)$-curve and a $0$-curve, respectively: 
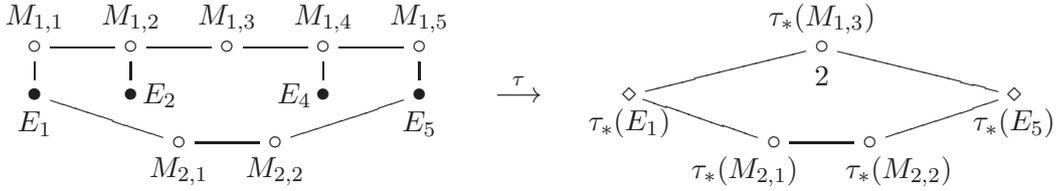
\begin{align*}
\xygraph{
\bullet ([]!{+(0,-.3)} {E_1}) - []!{+(0,.5)} \circ ([]!{+(0,+.3)} {M_{1,1}}) -[r] \circ ([]!{+(0,+.3)} {M_{1,2}})
(- []!{+(0,-.5)} \bullet ([]!{+(+.3,0)} {E_2}), -[r] \circ ([]!{+(0,+.3)} {M_{1,3}}) -[r] \circ ([]!{+(0,+.3)} {M_{1,4}})
(- []!{+(0,-.5)} \bullet ([]!{+(-.3,0)} {E_4}), -[r] \circ ([]!{+(0,+.3)} {M_{1,5}}) - []!{+(0,-.5)} \bullet ([]!{+(0,-.3)} {E_5}) - []!{+(-1.5,-.5)} \circ ([]!{+(0,-.3)} {M_{2,2}})
-[l] \circ ([]!{+(0,-.3)} {M_{2,1}}) - []!{+(-1.5,.5)} \circ 
))}
\quad \overset{\tau}{\longrightarrow} \quad
\xygraph{
\diamond ([]!{+(0,-.3)} {\tau_{\ast}(E_1)}) - []!{+(2,.5)} \circ ([]!{+(0,+.3)} {\tau _{\ast}(M_{1,3})}) ([]!{+(0,-.3)} {2}) - []!{+(2,-.5)} \diamond ([]!{+(0,-.3)} {\tau _{\ast}(E_5)}) - []!{+(-1.5,-.5)} \circ ([]!{+(+.3,-.3)} {\tau _{\ast}(M_{2,2})})
-[l] \circ ([]!{+(-.3,-.3)} {\tau _{\ast}(M_{2,1})}) - []!{+(-1.5,.5)} \ 
}
\end{align*}
Then $(-K_V)^2=8$ and $V$ contains two $(-2)$-curves $\tau _{\ast}(M_{2,1})$ and $\tau _{\ast}(M_{2,2})$. 
This is a contradiction. 
Thus, $D_3$ satisfies the condition ${\rm (B)}$. 
In other words, there exists a $(-1)$-curve $E_3$ on $\wS _{\kc}$ such that $D^{(1)}_3=2E_3$. 
In particular, we know $(E_3 \cdot M_{i,j}) = \delta _{1,i} \delta _{3,j}$. 
Since $E_1+E_5$ and $E_3$ are defined over $k$, we see that $\wS _{\kc}$ contains a union defined over $k$ of curves corresponding to the following dual graph, where ``$\circ$"and  ``$\bullet$" mean a $(-2)$-curve and a $(-1)$-curve, respectively: 
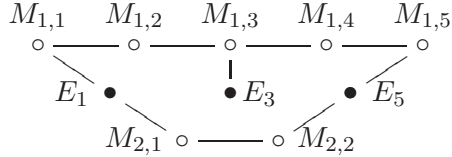
\begin{align*}
\xygraph{\bullet ([]!{+(-.4,0)} {E_1}) (
        - []!{+(-.75,.5)} \circ ([]!{+(0,+.3)} {M_{1,1}}) -[r] \circ ([]!{+(0,+.3)} {M_{1,2}}) -[r] \circ ([]!{+(0,+.3)} {M_{1,3}}) (- []!{+(0,-.5)} \bullet ([]!{+(+.3,0)} {E_3}),-[r]  \circ ([]!{+(0,+.3)} {M_{1,4}}) -[r] \circ ([]!{+(0,+.3)} {M_{1,5}}) - []!{+(-.75,-.5)} \bullet ([]!{+(+.4,0)} {E_5})),
        - []!{+(.75,-.5)} \circ ([]!{+(-.5,0)} {M_{2,1}}) -[r] \circ ([]!{+(+.5,0)} {M_{2,2}}) - []!{+(.75,.5)} \circ )}
\end{align*}
\end{eg}

\subsection{Proof of Proposition \ref{ADE-prop}(3)}\label{4-3}

In this subsection, we shall prove Proposition \ref{ADE-prop}(3). 
With the notation as in Proposition \ref{ADE-prop}(3), notice that $E$ is a $(-1)$-curve on $\wS _{\kc}$ by Proposition \ref{ADE-prop}(1). 
Since $D^{(1)} \sim D -D^{(2)}$, we can write $E \sim _{\bQ} \frac{1}{d}(-K_{\wS _{\kc}}) -\sum _{i=1}^{r'}M_i$, where each $M_i$ is an effective $\bQ$-divisor generated by $M_{i,1},\dots ,M_{i,n(i)}$. 
In particular, we note $M_i \not= 0$ for every $i=1,\dots ,r$. 
\begin{lem}\label{(3)lem-1}
Let $D_1$ and $D_2$ be two $\bQ$-divisors on $\wS _{\kc}$ generated by $M_{i,1},\dots ,M_{i,n(i)}$. 
If $(D_1 \cdot M_{i,j})=(D_2 \cdot M_{i,j})$ for any $j=1,\dots ,n(i)$, then $D_1=D_2$. 
\end{lem}
\begin{proof}
It is enough to show when we assume $D_2=0$. 
We shall write $D_1 = \sum _{j=1}^{n(i)} b_{i,j}M_{i,j}$ for some $b_{i,j} \in \bQ$. 
By assumption, we have the following linear simultaneous equation: 
\begin{align*}
\left[ \begin{array}{c} (D_1 \cdot M_{i,1}) \\ \vdots \\ (D_1 \cdot M_{i,n(i)}) \end{array} \right]
=A
\left[ \begin{array}{c} b_{i,1} \\ \vdots \\ b_{i,n(i)} \end{array} \right]
=
\left[ \begin{array}{c} 0 \\ \vdots \\ 0 \end{array} \right]
,
\end{align*}
where $A$ is the intersection matrix with respect to $M_{i,1},\dots ,M_{i,n(i)}$, i.e., $A = ((M_{i,j} \cdot M_{i,j'}))_{1\le j,\, j' \le n(i)}$. 
Since it is well known that intersection matrix is negative definite ({\cite{Mum61}}), we obtain $b_{i,j}=0$ for any $j=1,\dots ,n(i)$, which means $D_1=0$. 
\end{proof}
\begin{lem}\label{(3)lem-2}
$(E \cdot M_{i,1}+ \dots +M_{i,n(i)}) \le 1$ for $i=1,\dots ,r'$. 
\end{lem}
\begin{proof}
Let $\Delta _{i,j}$ be the $\bQ$-divisor, which is a $\bQ$-linear combination of $M_{i,1},\dots ,M_{i,n(i)}$, with $(\Delta _{i,j} \cdot M_{i,j'})=\delta _{j,j'}$ for $j,j'=1,\dots ,n(i)$ on $\wS _{\kc}$. 
Note that such a $\bQ$-divisor $\Delta _{i,j}$ is certainly exists and each coefficient of $\Delta _{i,j}$ can be determined by Lemmas \ref{A-1}, \ref{D-1} and \ref{E-1}. 
In particular, any coefficient of $\Delta _{i,j}$ is less than or equal to $-\frac{1}{2}$. 
Hence, we have $(\Delta _{i,j} \cdot \Delta _{i,j'}) \le -\frac{1}{2}$ for any $j,j'=1,\dots ,n(i)$, where the equal sign holds if and only if $n(i)=1$. 
On the other hand, by Lemma \ref{(3)lem-1}, we obtain: 
\begin{align*}
M_i = (M_i \cdot M_{i,j})\Delta _{i,j} = (E \cdot M_{i,j})\Delta _{i,j}
\end{align*}
for any $j$ by virtue of $(E-M_i \cdot M_{i,j})=0$. 
Meanwhile, by using Lemma \ref{ADE-1}(1), we note $(M_i)^2<0$ if $M_i \not= 0$. 
Suppose that $(E \cdot M_{i,1}+ \dots +M_{i,n(i)}) \ge 2$ for some $i \in \{ 1,\dots ,r'\}$. 
Notice $(E \cdot M_{i,j}) \ge 0$ for any $j$. 
If there exists $j_0$ such that $(E \cdot M_{i,j_0}) \ge 2$, then we have: 
\begin{align*}
(E)^2 \le \frac{1}{d} +(E \cdot M_{i,j_0})^2(\Delta _{i,j_0})^2 \le 1 -2 = -1,
\end{align*}
furthermore, we see $(E)^2<-1$ by virtue of $n(i)\ge 2$ or both $n(i)=1$ and $i>1$. 
This is absurd as $(E)^2=-1$. 
Otherwise, by hypothesis there exist two integers $j_1$ and $j_2$ such that $(E \cdot M_{i,j_1})=(E \cdot M_{i,j_2})=1$. 
By virtue of $n(i) \ge 2$, we have: 
\begin{align*}
(E)^2 \le \frac{1}{d} +(\Delta _{i,j_1})^2 + (\Delta _{i,j_2})^2 + 2(\Delta _{i,j_1} \cdot \Delta _{i,j_2}) < 1 -\frac{1}{2} -\frac{1}{2} -1= -1,
\end{align*}
which is absurd as $(E)^2=-1$. 
\end{proof}
\begin{lem}\label{(3)lem-3}
Assume that $D$ satisfies the condition {\rm (B)}. 
For $i =1,\dots ,r'$, $(E \cdot M_{i,1}+ \dots +M_{i,n(i)}) \ge 1$ if $M_i \not= 0$. 
\end{lem}
\begin{proof}
Suppose $(E \cdot M_{i,1}+ \dots +M_{i,n(i)}) =0$ for some $i \in \{1,\dots ,r'\}$. 
Then we note $(E \cdot M_{i,j})=0$ for any $j=1,\dots ,n(i)$. 
Hence, we obtain $M_i = 0$ by Lemma \ref{(3)lem-1}. 
\end{proof}
Proposition \ref{ADE-prop}(3) can be shown as follows:  
\begin{proof}[Proof of Proposition \ref{ADE-prop}(3)]
The first assertion of Proposition \ref{ADE-prop}(3) follows immediately from the beginning of \S \S \ref{4-3}. 
Hence, we shall prove the second assertion of this in what follows. 
In this proof, we will consider two cases separately: 

In the case that $D$ satisfies the condition {\rm (A)}. 
In other words, there exists a $(-1)$-curve $E'$ on $\wS _{\kc}$ such that $D^{(1)}=E+E'$ and $E \not= E'$. 
Furthermore, we see $D \sim E+E'$ by Proposition \ref{ADE-prop}(2). 
By construction of $D$, we see $(E+E' \cdot M_{i,1}+ \dots +M_{i,n(i)}) = (D \cdot M_{i,1}+ \dots +M_{i,n(i)})=2$. 
Hence, we obtain $(E \cdot M_{i,1}+ \dots +M_{i,n(i)}) = (E' \cdot M_{i,1}+ \dots +M_{i,n(i)})=1$ for $i=1,\dots ,r$ by Lemma \ref{(3)lem-2}. 

In the case that $D$ satisfies the condition {\rm (B)}. 
In other words, we can write $D^{(1)}=2E$. 
Hence, we obtain $(E \cdot M_{i,1}+ \dots +M_{i,n(i)})=1$ for $i=1,\dots ,r$ by Lemmas \ref{(3)lem-2} and \ref{(3)lem-3}. 
\end{proof}

\section{Degree $2$ or lower}\label{5}
In this section, we shall show Theorem \ref{main(1-3)}. 
Unless otherwise stated, let $S$ be a Du Val del Pezzo surface over $k$ with $\rho _k(S)=1$ and of degree $d \le 2$ and let $\sigma : \wS \to S$ be the minimal resolution over $k$, so that $\wS$ is a weak del Pezzo surface of degree $d$ over $k$. 

\subsection{Base locus with respect to cylinder}\label{5-1}

In this subsection, we shall study the base locus with respect to a cylinder on $S$. 

Supposing that $S$ contains a cylinder, say $U \simeq \bA ^1_k \times Z$, where $Z$ is a smooth affine curve defined over $k$, the closures in $S$ of fibers of the projection $pr_Z : U \simeq \bA ^1_k \times Z \to Z$ yields a linear system, say $\sL$, on $S$. 
By Lemma \ref{Bs} we see that $\Bs (\sL )$ consists of exactly one $k$-rational point, say $p$, which is a singular point on $S_{\kc}$. 
On the other hand, $\wU := \sigma ^{-1}(U) \simeq U$ is a cylinder on $\wS$ since $U _{\kc}$ is smooth. 
The closures in $\wS$ of fibers of the projection $pr_Z : \wU \simeq \bA ^1_k \times Z \to Z$ yields a linear system, say $\wsL$, on $\wS$. 
The purpose of this subsection is to show the following proposition: 
\begin{prop}\label{Bs(3)}
With the notation and the assumptions as above, assume further that one of the following conditions holds: 
\begin{enumerate}
\item $d=2$ and the singular point $p$ is of type $A_n$ on $S_{\kc}$ but not of type $A_n^-$ on $S$ for some $n=1,\dots ,6$. 
\item $d=1$ and the singular point $p$ is of type $A_n$ on $S_{\kc}$ but not of type $A_n^-$ on $S$ for some $n=1,\dots ,8$. 
\item $d=1$ and the singular point $p$ is of type $D_5^+$ on $S_{\kc}$. 
\item $d=1$ and the singular point $p$ is of type $E_6^+$ on $S_{\kc}$. 
\end{enumerate}
Then $\Bs (\wsL )$ consists of only one $k$-rational point. 
In particular, $p$ is not of type $A_n^{++}$ on $S$ except for only one case $(d,n)=(2,7)$. 
\end{prop}
In what follows, we shall prove Proposition \ref{Bs(3)}. 
Let $M_1,\dots ,M_n$ be all irreducible components of the exceptional set over $\kc$ of $\sigma _{\kc}$ at $p$ such that the dual graph of $M_1,\dots ,M_n$ is that as in (\ref{graphA}), (\ref{graphD}) or (\ref{graphE}) according to the singularity type of $p$ on $S_{\kc}$. 
Now, the following two lemmas hold: 
\begin{lem}\label{0-curve(1)}
Assume that the singular point $p$ is of type $A_n$ on $S_{\kc}$. Then: 
\begin{enumerate}
\item If $d=2$, then there exists a curve $C$ on $\wS$ such that $C \sim (-K_{\wS}) - (M_1+\dots +M_n)$. 
Hence, $M_1+ \cdots + M_n +C$ is a cycle. 
\item If $d=1$ and $n \ge 3$, then there exists a curve $C$ on $\wS$ such that $C \sim 2(-K_{\wS}) - 2(M_1+\dots +M_n)+(M_1+M_n)$. 
Hence, $M_2+ \cdots + M_{n-1} +C$ is a cycle. 
\end{enumerate}
\end{lem}
\begin{proof}
In (1), let $D$ be the divisor on $\wS _{\kc}$ defined by $D := (-K_{\wS _{\kc}}) - (M_1+\dots +M_n)$. 
By construction, we have $(D)^2 =0$ and $(D \cdot -K_{\wS _{\kc}}) =2$. 
Hence, we see $\dim |D| \ge 1$ by the Riemann-Roch theorem. 
Thus, there exists a curve $C$ on $\wS _{\kc}$ such that $C \sim D$ and $C$ is defined over $k$. 
Namely, $(C \cdot M_j) = (D \cdot M_j) = \delta _{j,1} + \delta _{j,n}$. 
This completes the proof of (1). 

In (2), it can be shown by the argument similar to (1). 
\end{proof}
\begin{lem}\label{0-curve(2)}
Assume that $d=1$ and the singular point $p$ is of type $D_5$ on $S_{\kc}$. 
Then there exists a curve $C$ on $\wS _{\kc}$ such that $C \sim 2(-K_{\wS _{\kc}}) -(2M_1+2M_2+3M_3+2M_4+M_5)$. 
Hence, $M_1+M_2+M_3+C$ is a cycle. 
\end{lem}
\begin{proof}
This lemma can be shown by the argument similar to Lemma \ref{0-curve(1)}. 
\end{proof}
\begin{proof}[Proof of Proposition \ref{Bs(3)}]
Let $\wL$ be a general member of $\wsL$. 
Since $\Bs ( \sL ) = \{ p\}$, we see that $\wL$ meets $M_i$ for some $1 \le i \le n$. 
By construction of $\wsL$, if $\wL$ meets two distinct irreducible components $M_i$ and $M_j$, then $\Bs (\wsL ) = M_i \cap M_j \not= \emptyset$. 
In what follows, we thus assume that $\wL$ meets exactly one irreducible component, say $M_{i_0}$. 
Notice that $M_{i_0}$ is defined over $k$. 
Let $a$ and $b$ be two positive rational numbers such that $da = ( \wsL \cdot -K_{\wS})$ and $2b = (\wsL \cdot M_{i_0})$. 

Now, we notice that $\Bs (\wsL ) \not= \emptyset$ if $(\wsL )^2 \not= 0$. 
Hence, we shall show $(\wsL )^2 \not= 0$ according to the conditions (1)--(4) in Proposition \ref{Bs(3)} in what follows: 

In (1) or (2), by the configuration of a dual graph of  $M_1+\dots +M_n$, we see that $n$ is odd and $i_0$ is equal to $\mn$. 
In particular, $M_{i_0}$ corresponds to the central vertex in this dual graph. 
Thus, by Lemma \ref{A-1}, we have $\wsL \sim _{\bQ} a(-K_{\wS})- bM$, where $M = \sum _{j=1}^{i_0 -1}j(M_j+M_{n+j-1}) + i_0 M_{i_0} $. 
Moreover, we obtain $(\wsL )^2 = da^2-2 i_0b^2$. 
Suppose that $(\wsL )^2 = 0$. 
Note that $i_0 \le 4$ because of $n \le 8$. 
Hence, we obtain $(d,i_0)=(1,2)$ or $(2,1)$ since $a,b$ are rational numbers. 
In particular, we have $a=(3-d)b$. 
However, the curve $C$ on $\wS _{\kc}$, which is that as in Lemma \ref{0-curve(1)}, then satisfies $(\wsL \cdot C)=0$. 
This implies that the boundary of $\wU _{\kc}$ contains a cycle, which contradicts Lemma \ref{no cycle}. 
Therefore, we see $(\wsL )^2 \not= 0$. 

In (3), since $p$ is of type $D_5^+$, note that $M_1$ and $M_2$ lie in the same $\Gal$-orbit, on the other hand, $M_3,M_4$ and $M_5$ are defined over $k$, respectively. 
Hence, $i_0$ is equal to $3$, $4$ or $5$. 
Thus, by Lemma \ref{D-1}, we have $\wsL \sim _{\bQ} a(-K_{\wS})- bM$ and $(\wsL )^2 = a^2 + (M)^2b^2$, where $M$ is that as in Lemma \ref{D-1}(3), (4) or (5) according to the number of $i_0$. 
In particular, $(M)^2=-3$, $-2$ and $-1$ if $i_0=3$, $4$ and $5$, respectively. 
Suppose that $(\wsL )^2 = 0$. 
Then $(M)^2=-1$ since $a,b$ are rational numbers. 
Hence, we see $i_0=5$ and $a=b$ by Lemma \ref{D-1}(3), (4) and (5). 
However, the curve $C$ on $\wS$, which is that as in Lemma \ref{0-curve(2)}(1), then satisfies $(\wsL \cdot C)=0$. 
This implies that the boundary of $\wU _{\kc}$ contains a cycle, which contradicts Lemma \ref{no cycle}. 
Therefore, we see $(\wsL )^2 \not= 0$. 

In (4), since $p$ is of type $E_6^+$, note that $M_1$ and $M_2$ (resp. $M_3$ and $M_4$) lie in the same $\Gal$-orbit, on the other hand, $M_5$ and $M_6$ are defined over $k$, respectively. 
Hence, $i_0$ is equal to $5$ or $6$. 
Thus, by Lemma \ref{E-1}, we have $\wsL \sim _{\bQ} a(-K_{\wS})- bM$ and $(\wsL )^2 = a^2 + (M)^2b^2$, where $M$ is that as in Lemma \ref{E-1}(5) or (6) according to the number of $i_0$. 
In particular, $(M)^2=-6$ and $-2$ if $i_0=5$ and $6$, respectively. 
Thus, we see $(\wsL )^2 \not= 0$ since $a,b$ are rational numbers. 
\end{proof}

\subsection{Proof of Theorem \ref{main(1-3)}(1)--(3)}\label{5-2}
In this subsection, we shall show Theorem \ref{main(1-3)}(1)--(3). 
In order to prove Theorem \ref{main(1-3)}(1) and (2), we will use Table \ref{list(2)}. 
In fact, in this proof, we mainly consider the two morphisms over $k$. 
One is the minimal resolution $\sigma :\wS \to S$ over $k$ and the other is the contraction $\tau : \wS \to W_{d'}$ over $k$ of the union of $(-1)$-curves, which can be determined by the dual graph in Table \ref{list(2)} according to the type of $\wS$ (the detailed configuration of $\tau$ will be treated in the following Lemmas \ref{construct cylinder(1)}, \ref{construct cylinder(2)} and \ref{lem(5-2-2)-1}). 
By construction of $\tau$, we will know that $W_{d'}$ contains a cylinder, whose boundary includes the exceptional set of $\tau$. 
Hence, the pullback of this cylinder by $\tau$ is also a cylinder in $\wS$, moreover, we will see that this boundary includes the union of all $(-2)$-curves, which is clearly defined over $k$. 
Thus, the image of this cylinder via $\sigma$ is a cylinder in $S$, namely, we see that $S$ certainly contains a cylinder. 

In what follows, we shall state the notation in Table \ref{list(2)}. 
Letting $\tau :\wS \to W_{d'}$ be the morphism as above depending on the type of $\wS _{\kc}$, we then see that $W_{d'}$ is a weak del Pezzo surface. 
Then ``$d'$" and ``Type of $W_{d'}$" in Table \ref{list(2)} mean the degree and the type of $W_{d'}$ according to the type of $\wS$, respectively. 
On the other hand, ``$\rho _k(\wS )$" in Table \ref{list(2)} means the Picard number of $\wS$ according to the type of $\wS$. 
Notice that this can be obtained by the Picard number of $W_{d'}$, which is explicitly given (see Table \ref{list(1-1)}), and the construction of $\tau$. 
Moreover, ``Dual graph" in Table \ref{list(2)} means a dual graph on $\wS _{\kc}$ according to the type of $\wS$, where ``$\circ$" means a $(-2)$-curve, ``$\bullet$" means a $(-1)$-curve and ``$\xygraph{(-[]!{+(.5,0)} \bigodot ([]!{+(0,0)} {})}$" means either ``$\xygraph{-[]!{+(.5,0)} \bullet ([]!{+(0,0)} {})-[]!{+(.5,0)} \circ ([]!{+(0,0)} {})}$" or ``$\xygraph{((-[]!{+(.5,.15)} \bullet ([]!{+(0,0)} {}),-[]!{+(.5,-.15)} \bullet ([]!{+(0,0)} {})}$", which can be determined according to the type of $\wS$. 
Note that the union of curves on $\wS _{\kc}$ corresponding to all vertices on this dual graph is certainly defined over $k$ by the configuration of $W_{d'}$ (see Table \ref{list(1-1)}) and the construction of $\tau$. 
\begin{table}[htbp]
\caption{Types of $\wS$ in Theorem \ref{main(1-3)}(1) and (2)}\label{list(2)}
\begin{center}
 \begin{tabular}{|c|c|c|c|c|c|}
 \hline
$d$ & Type of $\wS$ & $\rho _k(\wS )$ & Dual graph & $d'$ & Type of $W_{d'}$ \\ \hline \hline

$2$ & $D_4$ & $3$, $4$ or $5$ & \multirow{4}{*}{
$\xygraph{
\circ ([]!{+(.3,0)})
 (-[]!{+(-.5,.5)} \circ ([]!{+(+.3,0)} {}) -[]!{+(-.5,0)} \bigodot ([]!{+(+.3,0)} {}), -[]!{+(-.5,0)} \circ ([]!{+(+.3,0)} {}) -[]!{+(-.5,0)} \bigodot ([]!{+(+.3,0)} {}), -[]!{+(-.5,-.5)} \circ ([]!{+(+.3,0)} {}) -[]!{+(-.5,0)} \bigodot ([]!{+(+.3,0)} {}))
}$}
 & \multirow{4}{*}{8} & \multirow{4}{*}{$A_1$} \\ \cline{1-3}

$2$ & $D_4+A_1$ & $5$ or $6$ & 
 & & \\ \cline{1-3}

$2$ & $D_4+2A_1$ & $5$ or $7$ & 
 & & \\ \cline{1-3}

$2$ & $D_4+3A_1$ & $4$, $6$ or $8$ & 
 & & \\ \hline \hline

$2$ & $A_6$ & $4$ & 
$\xygraph{
(-[]!{+(0,.25)} \circ ([]!{+(0,0)} {}) (-[]!{+(-.5,0)} \circ ([]!{+(0,0)} {}),-[]!{+(.5,0)} \circ ([]!{+(0,0)} {})-[]!{+(.5,0)} \circ ([]!{+(0,0)} {})-[]!{+(.5,0)} \circ ([]!{+(0,0)} {})(-[]!{+(0,-.5)} \bullet ([]!{+(0,0)} {}),-[]!{+(.5,0)} \circ ([]!{+(0,0)} {}))),-[]!{+(0,-.25)} \bullet ([]!{+(0,0)} {}))
}$
 & $4$ & $A_2+2A_1$ \\ \hline

$2$ & $A_7$ & $5$ or $8$ & 
$\xygraph{
(-[]!{+(0,.25)} \circ ([]!{+(0,0)} {}) (-[]!{+(-.5,0)} \circ ([]!{+(0,0)} {}),-[]!{+(.5,0)} \circ ([]!{+(0,0)} {})-[]!{+(.5,0)} \circ ([]!{+(0,0)} {})-[]!{+(.5,0)} \circ ([]!{+(0,0)} {})-[]!{+(.5,0)} \circ ([]!{+(0,0)} {})(-[]!{+(0,-.5)} \bullet ([]!{+(0,0)} {}),-[]!{+(.5,0)} \circ ([]!{+(0,0)} {}))),-[]!{+(0,-.25)} \bullet ([]!{+(0,0)} {}))
}$
 & $4$ & $A_3+2A_1$ \\ \hline

$2$ & $D_5$ & $5$ & \multirow{2}{*}{
$\xygraph{
\circ ([]!{+(.3,0)})
 (-[]!{+(.5,0)} \circ ([]!{+(0,0)} {}) -[]!{+(.5,0)} \circ ([]!{+(0,0)} {}) -[]!{+(.5,0)} \bigodot ([]!{+(0,0)} {}),
 -[]!{+(-.5,.25)} \circ ([]!{+(0,0)} {}) -[]!{+(-.5,0)} \bullet ([]!{+(0,0)} {}),
 -[]!{+(-.5,-.25)} \circ ([]!{+(0,0)} {}) -[]!{+(-.5,0)} \bullet ([]!{+(0,0)} {}))
}$}
 & $4$ & $(A_3)_<$ \\ \cline{1-3} \cline{5-6}

$2$ & $D_5+A_1$ & $6$ & 
 & $4$ & $A_3+A_1$ \\ \hline

$2$ & $D_6$ & $7$ & \multirow{2}{*}{
$\xygraph{
\circ ([]!{+(.3,0)})
 (-[]!{+(.5,0)} \circ ([]!{+(0,0)} {}) -[]!{+(.5,0)} \circ ([]!{+(0,0)} {}) -[]!{+(.5,0)} \circ ([]!{+(0,0)} {}) -[]!{+(.5,0)} \bigodot ([]!{+(0,0)} {}),
 -[]!{+(-.5,.25)} \circ ([]!{+(0,0)} {}) -[]!{+(-.5,0)} \bullet ([]!{+(0,0)} {}), -[]!{+(-.5,-.25)} \circ ([]!{+(0,0)} {}))
}$}
 & $3$ & $A_5$ \\ \cline{1-3} \cline{5-6}

$2$ & $D_6+A_1$ & $8$ & 
 & $3$ & $A_5+A_1$ \\ \hline

$2$ & $E_6$ & $5$ & 
$\xygraph{
\circ ([]!{+(.3,0)})
 (-[]!{+(.5,0)} \circ ([]!{+(0,0)} {}),
 -[]!{+(-.5,.25)} \circ ([]!{+(0,0)} {}) -[]!{+(-.5,0)} \circ ([]!{+(0,0)} {})-[]!{+(-.5,0)} \bullet ([]!{+(0,0)} {}), -[]!{+(-.5,-.25)} \circ ([]!{+(0,0)} {})-[]!{+(-.5,0)} \circ ([]!{+(0,0)} {})-[]!{+(-.5,0)} \bullet ([]!{+(0,0)} {}))
}$
 & $4$ & $D_4$ \\ \hline

$2$ & $E_7$ & $8$ & 
$\xygraph{
(-[]!{+(0,.25)} \circ ([]!{+(0,0)} {}) (-[]!{+(-.5,0)} \circ ([]!{+(0,0)} {})-[]!{+(-.5,0)} \circ ([]!{+(0,0)} {}),-[]!{+(.5,0)} \circ ([]!{+(0,0)} {})-[]!{+(.5,0)} \circ ([]!{+(0,0)} {})-[]!{+(.5,0)} \circ ([]!{+(0,0)} {})-[]!{+(0,-.5)} \bullet ([]!{+(0,0)} {})),-[]!{+(0,-.25)} \circ ([]!{+(0,0)} {}))
}$
 & $3$ & $E_6$ \\ \hline \hline

$1$ & $A_8$ & $5$ or $9$ & 
$\xygraph{
(-[]!{+(0,.25)} \circ ([]!{+(0,0)} {}) (-[]!{+(-.5,0)} \circ ([]!{+(0,0)} {})-[]!{+(-.5,0)} \circ ([]!{+(0,0)} {}),-[]!{+(.5,0)} \circ ([]!{+(0,0)} {})-[]!{+(.5,0)} \circ ([]!{+(0,0)} {})-[]!{+(.5,0)} \circ ([]!{+(0,0)} {})(-[]!{+(0,-.5)} \bullet ([]!{+(0,0)} {}),-[]!{+(.5,0)} \circ ([]!{+(0,0)} {})-[]!{+(.5,0)} \circ ([]!{+(0,0)} {}))),-[]!{+(0,-.25)} \bullet ([]!{+(0,0)} {}))
}$
 & $3$ & $3A_2$ \\ \hline

$1$ & $D_6$ & $6$ or $7$ & \multirow{3}{*}{
$\xygraph{
\circ ([]!{+(.3,0)})
 (-[]!{+(-.5,.5)} \circ ([]!{+(0,0)} {}) -[]!{+(-1,0)} \circ ([]!{+(0,0)} {})(-[]!{+(0,-.5)} \bullet ([]!{+(0,0)} {}),-[]!{+(-.5,0)} \circ ([]!{+(0,0)} {})), -[]!{+(-.5,0)} \circ ([]!{+(0,0)} {}) -[]!{+(-.5,0)} \bigodot ([]!{+(0,0)} {}), -[]!{+(-.5,-.5)} \circ ([]!{+(0,0)} {}) -[]!{+(-.5,0)} \bigodot ([]!{+(0,0)} {}))
}$}
 & $2$ & $D_4+A_1$ \rule[0mm]{0mm}{5mm} \\ \cline{1-3} \cline{5-6}

$1$ & $D_6+A_1$ & $8$ & 
 & $2$ & $D_4+2A_1$ \rule[0mm]{0mm}{5mm} \\ \cline{1-3} \cline{5-6}

$1$ & $D_6+2A_1$ & $7$ or $9$ & 
 & $2$ & $D_4+3A_1$ \rule[0mm]{0mm}{5mm} \\ \hline

$1$ & $D_7$ & $7$ & 
$\xygraph{
\circ ([]!{+(.3,0)})
 (-[]!{+(.5,0)} \circ ([]!{+(0,0)} {}) -[]!{+(.5,0)} \circ ([]!{+(0,0)} {}) -[]!{+(.5,0)} \circ ([]!{+(0,0)} {}) (-[]!{+(.5,.25)} \circ ([]!{+(0,0)} {}),-[]!{+(.5,-.25)} \bullet ([]!{+(0,0)} {})),
 -[]!{+(-.5,.25)} \circ ([]!{+(0,0)} {}), -[]!{+(-.5,-.25)} \circ ([]!{+(0,0)} {}))
}$
 & $2$ & $D_5+A_1$ \\ \hline

$1$ & $D_8$ & $9$ & 
$\xygraph{
\circ ([]!{+(.3,0)})
 (-[]!{+(.5,0)} \circ ([]!{+(0,0)} {}) -[]!{+(.5,0)} \circ ([]!{+(0,0)} {}) -[]!{+(.5,0)} \circ ([]!{+(0,0)} {}) -[]!{+(.5,0)} \circ ([]!{+(0,0)} {})(-[]!{+(.5,.25)} \circ ([]!{+(0,0)} {}),-[]!{+(.5,-.25)} \bullet ([]!{+(0,0)} {})),
 -[]!{+(-.5,.25)} \circ ([]!{+(0,0)} {}), -[]!{+(-.5,-.25)} \circ ([]!{+(0,0)} {}))
}$
 & $2$ & $D_6+A_1$ \\ \hline

$1$ & $E_7$ & $8$ & \multirow{2}{*}{
$\xygraph{
\circ ([]!{+(.3,0)})
 (-[]!{+(.5,0)} \circ ([]!{+(0,0)} {}) -[]!{+(.5,0)} \circ ([]!{+(+.3,0)} {}) -[]!{+(0,0)} \circ ([]!{+(+.3,0)} {}) -[]!{+(.5,0)} \bigodot ([]!{+(0,0)} {}),
 -[]!{+(-.5,.25)} \circ ([]!{+(0,0)} {}) -[]!{+(-.5,0)} \circ ([]!{+(0,0)} {}) -[]!{+(-.5,0)} \bullet ([]!{+(0,0)} {}), -[]!{+(-.5,-.25)} \circ ([]!{+(0,0)} {}))
}$}
 & $2$ & $D_6$ \\ \cline{1-3} \cline{5-6}

$1$ & $E_7+A_1$ & $9$ & 
 & $2$ & $D_6+A_1$ \\ \hline

$1$ & $E_8$ & $9$ & 
$\xygraph{
(-[]!{+(0,.25)} \circ ([]!{+(0,0)} {}) (-[]!{+(-.5,0)} \circ ([]!{+(0,0)} {})-[]!{+(-.5,0)} \circ ([]!{+(0,0)} {}),-[]!{+(.5,0)} \circ ([]!{+(0,0)} {})-[]!{+(.5,0)} \circ ([]!{+(0,0)} {})-[]!{+(.5,0)} \circ ([]!{+(0,0)} {})-[]!{+(.5,0)} \circ ([]!{+(0,0)} {})-[]!{+(0,-.5)} \bullet ([]!{+(0,0)} {})),-[]!{+(0,-.25)} \circ ([]!{+(0,0)} {}))
}$
 & $2$ & $E_7$ \\ \hline \hline

$2$ & $(A_5)''$ & $4$ & 
$\xygraph{
(-[]!{+(0,.25)} \circ ([]!{+(0,0)} {}) (-[]!{+(-.5,0)} \circ ([]!{+(0,0)} {}),-[]!{+(.5,0)} \circ ([]!{+(0,0)} {})-[]!{+(.5,0)} \circ ([]!{+(0,0)} {})(-[]!{+(0,-.5)} \bullet ([]!{+(0,0)} {}),-[]!{+(.5,0)} \circ ([]!{+(0,0)} {}))),-[]!{+(0,-.25)} \bullet ([]!{+(0,0)} {}))
}$
 & $4$ & $3A_1$ \\ \hline

$1$ & $(A_7)''$ & $5$ & 
$\xygraph{
(-[]!{+(0,.25)} \circ ([]!{+(0,0)} {}) (-[]!{+(-.5,0)} \circ ([]!{+(0,0)} {})-[]!{+(-.5,0)} \circ ([]!{+(0,0)} {}),-[]!{+(.5,0)} \circ ([]!{+(0,0)} {})-[]!{+(.5,0)} \circ ([]!{+(0,0)} {})(-[]!{+(0,-.5)} \bullet ([]!{+(0,0)} {}),-[]!{+(.5,0)} \circ ([]!{+(0,0)} {})-[]!{+(.5,0)} \circ ([]!{+(0,0)} {}))),-[]!{+(0,-.25)} \bullet ([]!{+(0,0)} {}))
}$
 & $3$ & $2A_2+A_1$ \\ \hline \hline

$2$ & $(A_5+A_1)''$ & $5$ & 
$\xygraph{
(-[]!{+(0,.25)} \circ ([]!{+(0,0)} {}) (-[]!{+(-.5,0)} \circ ([]!{+(0,0)} {}) []!{+(-.5,0)} \circ,-[]!{+(.5,0)} \circ ([]!{+(0,0)} {})-[]!{+(.5,0)} \circ ([]!{+(0,0)} {})(-[]!{+(0,-.5)} \bullet ([]!{+(0,0)} {}),-[]!{+(.5,0)} \circ ([]!{+(0,0)} {}))),-[]!{+(0,-.25)} \bullet ([]!{+(0,0)} {}))
}$
 & $4$ & $4A_1$ \\ \hline
\end{tabular}
 \end{center}
\end{table}
\subsubsection{}
At first, we shall show Theorem \ref{main(1-3)}(1). 
\begin{lem}\label{construct cylinder(1)}
Let the notation be the same as above. 
If $d=2$ and $S_{\kc}$ has a singular point of type $D_4$, then $S$ contains a cylinder. 
\end{lem}
\begin{proof}
Let $x$ be a singular point of type $D_4$ on $S_{\kc}$. 
Note that $x$ is $k$-rational on $S _{\kc}$ by Lemma \ref{(-2)-curve}.
Moreover, we see that $\wS$ is of $D_4+nA_1$-type for $n=0,1,2,3$ and $\wS (k) \not= \emptyset$ by the configuration of curves in $\wS _{\kc}$ (see also Table \ref{list(2)}). 
Let $\widetilde{E}$ be the union of reduced curves corresponding to three subgraphs $\xygraph{\bigodot ([]!{+(0,0)} {}) -[]!{+(.75,0)} \circ ([]!{+(0,0)} {})}$ in the dual graph in Table \ref{list(2)}. 
Notice that $\widetilde{E}$ is defined over $k$. 
Then we obtain the birational morphism $\tau :\wS \to W_8$ over $k$ such that $W_8$ is a $k$-form of the Hirzebruch surface $\bF _2$ of degree $2$ and the direct image $\tau _{\ast}(\widetilde{E})_{\kc}$ is the disjoint union of three closed fibers, say $F_1$, $F_2$ and $F_3$, of the $\bP ^1$-bundle $W_{8,\kc} \simeq \bF _2 \to \bP ^1_{\kc}$. 
In particular, we see $W_8 \simeq \bF _2$ by using Lemma \ref{Severi-Brauer} because of $\wS (k) \not= \emptyset$ (see also \S \ref{3}). 
Hence, $\wU := \wS \backslash \widetilde{E}$ is certainly the cylinder on $\wS$ since $\wU \simeq \bF _2 \backslash  (M \cup F_1 \cup F_2 \cup F_3) \simeq \bA ^1_k \times C_{(2)}$ (for the definition of $C_{(2)}$, see Notation and Conventions), where $M$ is the $(-2)$-curve on $\bF _2$. 
Therefore, we see that $S$ contains a cylinder $\sigma (\wU ) \simeq \wU$. 
\end{proof}
\begin{lem}\label{construct cylinder(2)}
Let the notation be the same as above. 
If $d=2$ (resp. $d=1$) and $S_{\kc}$ has a singular point of type $A_6$, $A_7$, $D_5$, $D_6$, $E_6$ or $E_7$ (resp. type $A_8$, $D_6$, $D_7$, $D_8$, $E_7$ or $E_8$), then $S$ contains a cylinder. 
\end{lem}
\begin{proof}
Let $x$ be a singular point of the type of the one of the above list on $S_{\kc}$. 
Note that $x$ is $k$-rational on $S _{\kc}$ by Lemma \ref{(-2)-curve}. 
Let $\widetilde{E}$ be the union of the $(-1)$-curves corresponding to all vertices $\bullet$ in the Table \ref{list(2)} according to the type of $\wS$. 
Notice that $\widetilde{E}$ is defined over $k$ and $\widetilde{E}_{\kc}$ is either irreducible or disjoint. 
Hence, we obtain the contraction $\tau :\wS \to W_{d'}$ of $\widetilde{E}$ defined over $k$, so that $W_{d'}$ is a weak del Pezzo surface of degree $d' \in \{ 2,3,4\}$, where $d'$ is determined according to the type of $\wS$. 
If $d' \in \{ 3,4\}$, then $W_{d'}$ contains a cylinder, whose boundary includes $\tau _{\ast}(\widetilde{E})$, by the argument in \S \ref{3} (see also Table \ref{list(1-1)}). 
Thus, the pullback of this cylinder by $\tau$, say $\wU$, is a cylinder in $\wS$ such that this boundary includes the union of all $(-2)$-curves on $\wS _{\kc}$, which is defined over $k$. 
Therefore, we see that $S$ contains a cylinder $\sigma (\wU ) \simeq \wU$. 
If $d'=2$, then $W_{d'}$ is one of the list in Table \ref{list(2)} and contains a cylinder, whose boundary includes $\tau _{\ast}(\widetilde{E})$, by the above argument. 
Namely, the above argument can work as well even if $d'=2$. 
This completes the proof. 
\end{proof}
\begin{rem}
We shall state some remarks on Lemma \ref{construct cylinder(2)} (cf. Remark \ref{cylinder rem}). 
Let $x$ be the same as in Lemma \ref{construct cylinder(2)} and assume $d=2$. Then: 
\begin{enumerate}
\item If the singular point $x$ is of type $A_6$, $E_6$ or $E_7$ on $S_{\kc}$, then $S$ always contains the affine plane $\bA ^2_k$ (compare the fact that the Du Val del Pezzo surface over $\bC$ with Picard rank one and of degree $2$ contains $\bC ^2$ if and only if this surface has a singular point of type $E_7$, see {\cite{MZ88}}). 
\item If the singular point $x$ is of type $A_7$ on $S_{\kc}$, then $\wS$ need not be $k$-rational but always contains a cylinder (compare the fact in {\cite{DK18}}). 
\end{enumerate}
\end{rem}
Theorem \ref{main(1-3)}(1) follows from Lemmas \ref{construct cylinder(1)} and \ref{construct cylinder(2)}. 
\subsubsection{}
Secondly, we shall show Theorem \ref{main(1-3)}(2). 
With the notation as above, assume further that $\wS _{\kc}$ has a singular point $x$ of type $(A_{9-2d})''$ (see \S \S \ref{2-1}, for this definition). 
Note that $x$ is $k$-rational on $S _{\kc}$ by Lemma \ref{(-2)-curve}. 
Notice that $\wS$ is only $(A_5)''$ or $(A_5+A_1)''$-type (resp. $(A_7)''$-type) if $d=2$ (resp. $d=1$). 
\begin{lem}\label{lem(5-2-2)-1}
With the notation and the assumptions as above, assume further that $\wS$ is of $(A_{9-2d})''$-type. 
Then $S$ contains a cylinder if and only if $x$ not of type $A_{9-2d}^{++}$ on $S$. 
\end{lem}
\begin{proof}
Assume that $S$ contains a cylinder $U \simeq \bA ^1_k \times Z$, where $Z$ is a smooth affine curve defined over $k$. 
Then $\wS$ contains a cylinder $\wU := \sigma ^{-1}(U) \simeq U$. 
The closures in $\wS$ of fibers of the projection $pr_Z : \wU \simeq \bA ^1_k \times Z \to Z$ yields a linear system, say $\wsL$, on $\wS$. 
By Proposition \ref{Bs(3)}, $\Bs (\wsL ) \not= \emptyset$. 
Thus, $x$ is not of type $A_{9-2d}^{++}$ on $S$ by the assumption and Lemma \ref{Bs}. 

Conversely, assume that $x$ is not of type $A_{9-2d}^{++}$ on $S$. 
Let $M$ be the $(-2)$-curve on $\wS _{\kc}$ corresponding to the central vertex on the dual graph with the minimal resolution at $x$. 
Notice that $M$ is defined over $k$, moreover, $M$ has a $k$-rational point by the assumption. 
Let $\widetilde{E}$ be the union of the $(-1)$-curves corresponding to two vertices $\bullet$ in the Table \ref{list(2)} according to the type of $\wS$. 
Notice that $\widetilde{E}$ is defined over $k$ and $\widetilde{E}_{\kc}$ is disjoint. 
Hence, we obtain the contraction $\tau :\wS \to W_{d+2}$ of $\widetilde{E}$ defined over $k$, so that $W_{d+2}$ is a weak del Pezzo surface of degree $d+2$ and $\tau _{\ast}(M)_{\kc}$ is a $(-2)$-curve. 
Moreover, since $M$ has a $k$-rational point, so does the image via $\tau$. 
Hence, $W_{d+2,\kc}$ contains a $(-2)$-curve with a $k$-rational point. 
This implies that $W_{d+2}$ contains a cylinder, whose boundary includes $\tau _{\ast}(\widetilde{E})$, by using Theorem \ref{main(1-2)} (see also Table \ref{list(1-1)}). 
Thus, the pullback of this cylinder by $\tau$, say $\wU$, is a cylinder in $\wS$ such that this boundary includes the union of all $(-2)$-curves on $\wS _{\kc}$, which is defined over $k$. 
Therefore, we see that $S$ contains a cylinder $\sigma (\wU ) \simeq \wU$. 
\end{proof}
In what follows, we deal with the case that $\wS$ is not of $(A_{9-2d})''$-type. 
Then notice that $d=2$ and $\wS$ is of $(A_5+A_1)''$-type. 
\begin{lem}\label{lem(5-2-2)-2}
With the notation and the assumptions as above, assume further that $d=2$ and $\wS$ is of $(A_5+A_1)''$-type. 
Then $S$ contains a cylinder if and only if $x$ not of type $A_5^{++}$ on $S$. 
\end{lem}
\begin{proof}
Let $M_{1,1},\dots, M_{1,5}$ and $M_{2,1}$ be the $(-2)$-curves on $\wS _{\kc}$ with the configuration as in (\ref{graphAA}). 
By the configuration, $M_{1,3}$ and $M_{2,1}$ are defined over $k$. 
By using Proposition \ref{ADE-prop}, there exist two $(-1)$-curves $E_2$ and $E_4$ on $\wS _{\kc}$ such that $(E_i \cdot M_{1,j}) = \delta _{i,j}$ and $(E_i \cdot M_{2,1})=0$ for $i=2,4$ and $j=1,\dots ,5$, moreover, the union and $E_2+E_4$ are defined over $k$ (cf. Example \ref{ex of prop(4-2)}). 
Let $\tau :\wS \to W_8$ be the compositions of successive contractions of a disjoint union $E_2 + E_4$, that of the images of $M_{1,2} + M_{1,4}$ and finally that of the images of $M_{1,1} + M_{1,5}$. 
By construction, $\tau$ is defined over $k$ and $W_8$ is a $k$-form of the Hirzebruch surface $\bF _2$ of degree $2$. 

From now on, we prove this lemma. 
Assume that $S$ contains a cylinder. 
Let $y$ be the singular point of type $A_1$ on $S_{\kc}$. 
Then we know that either $x$ is not of type $A_5^{++}$ on $S$ or $y$ is not of type $A_1^{++}$ on $S$ by the similar argument to Lemma \ref{lem(5-2-2)-1}. 
In what follows, we may assume that $y$ is not of type $A_1^{++}$ on $S$. 
In other words, $M_{2,1}$ has a $k$-rational point, hence, so does $\tau _{\ast}(M_{2,1})$. 
Namely, $W_8 \simeq \bF _2$. 
Hence, there exists a unique closed fiber of the $\bP ^1$-bundle $\bF _2 \to \bP ^1_k$ passing through this $k$-rational point. 
Let $F$ be the pullback of this fiber by $\tau$. 
Note that the configuration of the dual graph of $\sum _{j=1}^5M_{1,j}+M_{2,1}+E_2+E_4+F$ is as follows, where ``$\circ$", ``$\bullet$" and ``$\diamond$" mean a $(-2)$-curve, a $(-1)$-curve and a $0$-curve, respectively: 
\begin{align*}
\xygraph{
\diamond ([]!{+(-.3,0)} {F}) (- []!{+(0,-.5)} \circ ([]!{+(-.4,0)} {M_{2,1}}), - []!{+(0,.5)} \circ ([]!{+(0,.3)} {M_{1,3}}) 
((-[l] \circ ([]!{+(0,+.3)} {M_{1,2}}) (-[l] \circ ([]!{+(0,+.3)} {M_{1,1}}), - []!{+(0,-.5)} \bullet ([]!{+(-.3,0)} {E_2})),
(-[r] \circ ([]!{+(0,+.3)} {M_{1,4}}) (-[r] \circ ([]!{+(0,+.3)} {M_{1,5}}), - []!{+(0,-.5)} \bullet ([]!{+(-.3,0)} {E_4}))
))}
\end{align*}
In particular, the intersection point of $M_{1,3}$ and $F$ is $k$-rational, namely, $M_{1,3}(k) \not= \emptyset$. 
This implies that $x$ is not of type $A_5^{++}$ on $S$. 

Conversely, assume that $x$ is not of type $A_5^{++}$ on $S$. 
By putting $M := M_{1,3}$, we see that $S$ contains a cylinder by the similar argument to Lemma \ref{lem(5-2-2)-1}. 
\end{proof}
Theorem \ref{main(1-3)}(2) follows from Lemmas \ref{lem(5-2-2)-1} and \ref{lem(5-2-2)-2}. 

\subsubsection{}
Finally, we shall show Theorem \ref{main(1-3)}(3) by using {\cite{CPW16b}}: 
\begin{proof}[Proof of Theorem \ref{main(1-3)}(3)]
This proof can be shown by an argument similar to {\cite[Remark 10]{DK18}}. 
Indeed, supposing that $S$ contains a cylinder $U$, by $\rho _k(S)=1$ we see that $S_{\kc}$ admits an $(-K_{S_{\kc}})$-polar cylinder $U_{\kc}$ (see, e.g., {\cite[Definition 1.3]{CPW16b}}, for the definition), which is a contradiction to {\cite[Theorem 1.5]{CPW16b}}. 
\end{proof}
\begin{rem}
If $d=1$ and $S_{\kc}$ has a singular point of type $D_4$, then we see that $S$ does not contain any cylinder by Theorem \ref{main(1-3)}(3) since $\wS$ is only of $2D_4$, $D_4+A_3$, $D_4+3A_1$, $D_4+A_2$, $D_4+2A_1$, $D_4+A_1$ or $D_4$-type. 
\end{rem}

\subsection{Proof for the ``only if'' part of Theorem \ref{main(1-3)}(4)}\label{5-3}

In this subsection, we shall show the ``only if" part of Theorem \ref{main(1-3)}(4). 
Assume that $S$ does not satisfy any condition on singularities of (1), (2) nor (3) in Theorem \ref{main(1-3)} and contains a cylinder, say $U$. 
Letting $\sL$ be the linear system on $S$, which is the same as in \S \S \ref{5-1}, by Lemma \ref{Bs} we then see that $\Bs (\sL ) = \{ p\}$ such that $p$ is a singular point on $S$, which is $k$-rational. 
In order to show the ``only if" part of Theorem \ref{main(1-3)}(4), we shall prove that the singular point $p$ is of type $A_n^-$, $D_n^-$ or $E_n^-$. 
Letting $\wsL$ be the linear system on $\wS$, which is the same as in \S \S \ref{5-1}, by Proposition \ref{Bs(3)} we see that $\Bs (\wsL )$ consists of only one $k$-rational point, say $\wtp$. 
In other words, the singular point $p$ is not of type $A_n^{++}$ on $S$ for any $n$. 
In what follows, suppose that the singular point $p$ on $S$ is one of the following according to the degree $d$: 
\begin{itemize}
\item $d=2$: type $A_1^+$, $A_2^+$, $A_3^+$, $A_4^+$ or $(A_5^+)'$, 
\item $d=1$: type $A_1^+$, $A_2^+$, $A_3^+$, $A_4^+$, $A_5^+$, $A_6^+$, $(A_7^+)'$, $D_5^+$ or $E_6^+$. 
\end{itemize}
Meanwhile, we will prove Lemmas \ref{sym(1)}, \ref{sym(2)} and \ref{sym(3)}, which contradict the above hypothesis. 
Now, we shall treat the following Lemmas \ref{Corti(2)} and \ref{Corti(3)}, which will play a crucial role to show Lemmas \ref{sym(1)} and \ref{sym(2)}: 
\begin{lem}\label{Corti(2)}
Assume that $\wsL \sim _{\bQ} a(-K_{\wS})-bM$ for some positive rational numbers $a$ and $b$, where $M$ is an effective $\bQ$-divisor on $\wS$ and consists of some irreducible components of exceptional set of $\sigma$. 
Let $\alpha ,\beta$ and $\gamma$ be three positive rational numbers satisfying $a \ge \alpha b$, $\beta = -(M \cdot M_0)$, and $\gamma = -(M)^2$, where $M_0$ is an irreducible component of $M_{\kc}$ passing through $\wtp$. 
Then the following hold: 
\begin{enumerate}
\item If $d=2$, then the following four inequalities do not hold simultaneously: 
\begin{align}\label{(4.1)}
\begin{cases}
\alpha - u > 0 \\
\alpha - u - v \ge 0\\
2 \alpha u + \beta v - \gamma \ge 0 \\
4 u^2 + 4 uv + 2 v^2 - \gamma \le 0
\end{cases}
\end{align}
for any rational numbers $u,v$ with $u \ge 0$. 
\item If $d=1$, then the following four inequalities do not hold simultaneously: 
\begin{align}\label{(4.2)}
\begin{cases}
\alpha - u > 0 \\
\alpha - u - v \ge 0\\
 \alpha u + \beta v - \gamma \ge 0 \\
4 u^2 + 4 uv + 4 v^2 - 3\gamma \le 0
\end{cases}
\end{align}
for any rational numbers $u,v$ with $u \ge 0$. 
\end{enumerate}
\end{lem}
\begin{proof}
We only show (1), because (2) can be shown by the argument similar to (1). 

Suppose that there exist $u \in \bQ _{\ge 0}$ and $v \in \bQ$ such that the all inequalities (\ref{(4.1)}) hold simultaneously. 
By virtue of $\alpha - u > 0$, $\alpha - u - v \ge 0$, $b>0$ and $a \ge \alpha b$, we then see $a-ub>0$ and $1 - \frac{vb}{a-ub} \ge 0$. 
Hence, we have: 
\begin{align}\label{(4,3)}
\begin{split}
\left(\wsL \cdot K_{\wS} + \frac{vb}{a-ub}M_0 + \frac{1}{a-ub}\wsL\right) &= \frac{1}{a-ub}\left\{ -2a(a -ub) + \beta vb^2 +(2a^2-\gamma b^2) \right\} \\
&= \frac{b}{a-ub} (2ua + \beta vb -\gamma b). 
\end{split}
\end{align}
By virtue of $au \ge \alpha ub$ and $2 \alpha u + \beta v - \gamma \ge 0$, we have: 
\begin{align}\label{(4,4)}
\frac{b}{a-ub} (2ua + \beta vb -\gamma b) &\ge \frac{b^2}{a-ub} (2\alpha u + \beta v -\gamma ) \ge 0. 
\end{align}

Notice that the rational map $\Phi _{\wsL}:\wS \dashrightarrow \overline{Z}$ is not a morphism since $\Bs (\wsL ) = \{ \wtp \}$, where $\overline{Z}$ is the smooth projective model of $Z$. 
Let $\psi : \bar{S} \to \wS$ be the shortest succession of blow-ups of $\wtp$ and its infinitely near points such that the proper transform $\bar{\sL} := \psi ^{-1}_{\ast}(\wsL )$ of $\wsL$ is free of base points. 
Note that, $\wtp \in M_0$ and $(\bar{\sL} \cdot \bar{M}_0)=0$ by construction of $\wsL$, where $\bar{M}_0$ is the proper transform $\psi ^{-1}_{\ast}(M_0)$ of $M_0$. 
Letting ${\{ \bar{E}_i \}}_{1\le i \le n}$ be the exceptional divisors of $\psi$ with $\bar{E}_n$ the last exceptional one, which is a section of $\bar{\varphi} := \Phi _{\wsL} \circ \psi$, we have: 
\begin{align}\label{lc}
K_{\bar{S}} + \frac{vb}{a-ub}\bar{M}_0 +  \frac{1}{a-ub}\bar{\sL} = \psi ^{\ast} \left( K_{\wS} +\frac{vb}{a-ub}M_0+ \frac{1}{a-ub}\wsL \right) + \sum _{i=1}^nc_i \bar{E}_i
\end{align}
and
\begin{align}\label{en}
(\bar{\sL} \cdot \bar{E}_i) = \left\{ \begin{array}{ll} 0 & (1 \le i \le n-1) \\ 1 & (i=n) \end{array} \right. 
\end{align}
for some rational numbers $c_1,\dots ,c_n$. 
Note that the general member of $\bar{\sL}$ is a general fiber of the $\bP ^1$-fibration $\bar{\varphi}$. 
Hence, we have: 
\begin{align*}
-2 &= (\bar{\sL} \cdot K_{\bar{S}}) \\
&= \left(\bar{\sL} \cdot K_{\bar{S}} + \frac{vb}{a-ub}\bar{M}_0 + \frac{1}{a-ub}\bar{\sL}\right) \\
&\underset{(\ref{lc})}{=} \left(\bar{\sL} \cdot \psi ^{\ast}\left( K_{\wS} + \frac{vb}{a-ub}M_0 + \frac{1}{a-ub}\wsL \right) \right) + \sum _{i=1}^nc_i (\bar{\sL} \cdot \bar{E}_i)\\
&\underset{(\ref{en})}{=} \left(\wsL \cdot K_{\wS} + \frac{vb}{a-ub}M_0 + \frac{1}{a-ub}\wsL \right) + c_n 
\end{align*}
Thus, $(\wS ,\frac{vb}{a-ub}M_0 + \frac{1}{a-ub}\wsL )$ is not log canonical at $\wtp$ by (\ref{(4,3)}) and (\ref{(4,4)}). 
Furthermore, since $1 - \frac{vb}{a-ub} \ge 0$ and $\frac{1}{a-ub}>0$, we have: 
\begin{align}\label{(4,5)}
({\wsL} )^2 > 4(a-ub)^2\left( 1 - \frac{vb}{a-ub} \right) \iff 0 > 2a^2 -4(2u+v)ab + \left\{ 4u(u+v)+\gamma \right\} b^2. 
\end{align}
by using a variant of Corti's inequality (see Lemma \ref{Corti}). 
On the other hand, we have: 
\begin{align*}
2a^2 -4(2u+v)ab + \left\{ 4u(u+v)+\gamma \right\} b^2 = 2\left\{ a -(2u+v)b \right\} ^2 -(4u^2+4uv+2v^2-\gamma )b^2 \ge 0,
\end{align*}
by $4 u^2 + 4 uv + 2 v^2 - \gamma \le 0$. It is a contradiction to (\ref{(4,5)}). 
\end{proof}
Note that the following Lemma \ref{Corti(3)} is the special case of Lemma \ref{Corti(2)}: 
\begin{lem}\label{Corti(3)}
Let the notation and the assumption as in Lemma \ref{Corti(2)}. 
If $d=2$ and $d=1$, then we obtain $\alpha ^2 < \gamma$ and $3 \alpha ^2 < 4 \gamma$, respectively. 
\end{lem}
\begin{proof}
Suppose otherwise. 
If $d=2$ (resp. $d=1$), then we see that the four inequalities ({\ref{(4.1)}}) (resp. ({\ref{(4.2)}})) hold for $(u,v)=(\frac{\gamma}{2\alpha},0)$ (resp. $(u,v)=(\frac{\gamma}{\alpha},0)$). 
This is the contradiction to Lemma \ref{Corti(2)}. 
\end{proof}
Now, we show Lemmas \ref{sym(1)}, \ref{sym(2)} and \ref{sym(3)}. 
For these lemmas, let $M_1,\dots ,M_n$ be all irreducible components of the exceptional set over $\kc$ of $\sigma _{\kc}$ at $p$ such that the dual graph of $M_1,\dots ,M_n$ is as in (\ref{graphA}), (\ref{graphD}) or (\ref{graphE}) according to the singularities of $p$ on $S_{\kc}$. 
\begin{lem}\label{sym(1)}
With the notation and the assumptions as above, the following assertions hold: 
\begin{enumerate}
\item If $d=2$, then the singular point $p$ is not of type $A_1^+$, $A_3^+$, $A_4^+$ nor $(A_5^+)'$ on $S$. 
\item If $d=1$, then the singular point $p$ is not of type $A_1^+$, $A_2^+$, $A_3^+$, $A_5^+$, $A_6^+$ nor $(A_7^+)'$ on $S$. 
\end{enumerate}
\end{lem}
\begin{proof}
Suppose that the singular point $p$ on $S$ is one of the list in Lemma \ref{sym(1)}. 
We shall write $m := \mn$ for simplicity. 
By noting $\Bs (\wsL ) = \{ \wtp \}$, we see $(\wsL \cdot M_i) = 0$ for any $i$ other than $i=m$ (resp. $i=m,\ m+1$) if $n$ is odd (resp. even). 
Indeed, if $n$ is odd (resp. even), then $\wtp$ lies on $M_m$ (resp. the intersection point of $M_m$ and $M_{m+1}$). 
Hence, we can write $\wsL \sim _{\bQ} a(-K_{\wS}) - bM$ for some $a,b \in \bQ _{>0}$, where $M = \sum _{j=1}^{m -1}j(M_j+M_{n-j+1}) + m(M_m + \dots + M_{n-m+1})$ by Lemma \ref{A-1}. 

Let $\beta$ and $\gamma$ be two rational numbers defined by $\beta := -(M \cdot M_m)$ and $\gamma := -(M)^2$. 
Moreover, let $\alpha$ be the positive number defined by $\alpha := (M \cdot E)$, where $E$ is the $(-1)$-curve on $\wS _{\kc}$ according to the degree $d$ and the singularity type of $p$ on $S_{\kc}$ as follows: 
\begin{itemize}
\item $(d,\text{Singularity})=(2,A_3)$, $(2,A_4)$, $(1,A_5)$, $(1,A_6)$: 
By using Proposition \ref{ADE-prop}, we take a $(-1)$-curve $E$ on $\wS _{\kc}$ such that $(M_j \cdot E) =\delta _{m,j}$ (see also Example \ref{ex of prop(4-2)}). 
\item $(d,\text{Singularity})=(2,A_1)$, $(1,A_3)$: 
Notice that $S_{\kc}$ allows a singular point other than $p$ by the assumption of Theorem \ref{main(1-3)}(4). 
If $S_{\kc}$ admits a cyclic singular point other than $p$, then we take a $(-1)$-curve $E$ on $\wS _{\kc}$ such that $(M_j \cdot E) =\delta _{m,j}$ by an argument similar to the above. 
Otherwise, since $d=1$ and $\wS$ is of $D_5+A_3$-type by the assumption of Theorem \ref{main(1-3)}(4), it is known that there exists a $(-1)$-curve $E$ on $\wS _{\kc}$ such that $(M_j \cdot E) =\delta _{2,j}$ (see, e.g., {\cite[Figure 1]{MZ88}}), so that we take such a $(-1)$-curve $E$. 
\item $(d,\text{Singularity})=(2,(A_5)')$, $(1,(A_7)')$: 
By the configuration of singularity of $p$, we can take the $(-1)$-curve $E$ such that $(M_j \cdot E) =\delta _{m,j}$. 
\item $(d,\text{Singularity})=(1,A_1)$, $(1,A_2)$: 
We take the $(-1)$-curve $E$ as in Lemma \ref{ADE-(-1)}(1). Namely, $(M_1 \cdot E)=2$ (resp. $(M_1 \cdot E)=(M_2 \cdot E)=1$) if $p$ is of type $A_1^+$ (resp. type $A_2^+$) on $S$.  
\end{itemize}
By construction of $\alpha$, we see that $a \ge \alpha b$ because of $0 \le (\wsL \cdot E) = a - \alpha b$. 
Here, the values of $\alpha$, $\beta$ and $\gamma$ are summarized in Table \ref{A+} according to the degree $d$ and the singularity type of $p$ on $S_{\kc}$. 
For all cases except for $(d,\text{Singularity}) = (2,A_1), (1,A_3)$, we thus obtain a contradiction to Lemma \ref{Corti(3)}. 
In what follows, we consider the remaining cases. 
In the case of $(d,\text{Singularity}) = (2,A_1)$, setting $(u,v) := (0,1)$, the inequalities ({\ref{(4.1)}}) hold simultaneously, which contradicts Lemma \ref{Corti(2)}(1). 
In the case of $(d,\text{Singularity}) = (1,A_3)$, setting $(u,v) := (1,1)$, the inequalities ({\ref{(4.2)}}) hold simultaneously, which contradicts Lemma \ref{Corti(2)}(2). 
\end{proof}
\begin{table}[t]

\begin{center}
\caption{Values of $\alpha$, $\beta$ and $\gamma$ in Lemma \ref{sym(1)}}\label{A+}
\begin{tabular}{|c|c||l||c|c|c|} \hline
$d$ & Singularity & \multicolumn{1}{|c||}{Irreducible decomposition of $M$} & $\alpha$ & $\beta$ & $\gamma$ \\ \hline \hline

$2$ & $A_1^+$ & $M_1$ & $1$ & $2$ & $2$ \\ \hline
$2$ & $A_3^+$ & $M_1+2M_2+M_3$ & $2$ & $2$ & $4$ \\ \hline
$2$ & $A_4^+$ & $M_1+2M_2+2M_3+M_4$ & $2$ & $1$ & $4$ \\ \hline
$2$ & $(A_5^+)'$ & $M_1+2M_2+3M_3+2M_4+M_5$ & $3$ & $2$ & $6$ \\ \hline \hline
$1$ & $A_1^+$ & $M_1$ & $2$ & $2$ & $2$ \\ \hline
$1$ & $A_2^+$ & $M_1+M_2$ & $2$ & $1$ & $2$ \\ \hline
$1$ & $A_3^+$ & $M_1+2M_2+M_3$ & $2$ & $2$ & $4$ \\ \hline
$1$ & $A_5^+$ & $M_1+2M_2+3M_3+2M_4+M_5$ & $3$ & $2$ & $6$ \\ \hline
$1$ & $A_6^+$ & $M_1+2M_2+3M_3+3M_4+2M_5+M_6$ & $3$ & $1$ & $6$ \\ \hline
$1$ & $(A_7^+)'$ & $M_1+2M_2+3M_3+4M_4+3M_5+2M_6+M_7$ & $4$ & $2$ & $8$ \\ \hline
\end{tabular}
\end{center}
\end{table}
\begin{rem}
If the pair of the degree $d$ and the singular point $p$ on $S_{\kc}$ is $(2,A_2)$ (resp. $(1,A_4)$), there is actually no rational numbers pair $(u,v)$ such that the inequalities (\ref{(4.1)}) (resp. (\ref{(4.2)})) hold simultaneously. 
We will deal with these cases later (see Lemma \ref{sym(3)}). 
\end{rem}
\begin{lem}\label{sym(2)}
With the notation and the assumptions as above, if $d=1$ then the singular point $p$ is not of type $D_5^+$ nor $E_6^+$ on $S$. 
\end{lem}
\begin{proof}
Suppose that the singular point $p$ is of type $D_5^+$ or $E_6^+$ on $S$. 
By Proposition \ref{Bs(3)}, $\Bs (\wsL )$ consists of only one $k$-rational point, say $\wtp$. 
Note that $\wtp \in M_3 \cup M_4 \cup M_5$ but $\wtp \not\in M_1 \cup M_2$ (resp. $\wtp \in M_5 \cup M_6$ but $\wtp \not\in M_1 \cup M_2 \cup M_3 \cup M_4$) provided that $p$ is of type $D_5^+$ (resp. type $E_6^+$). 
Thus, we can write $\wsL \sim _{\bQ} a(-K_{\wS}) - bM$ for some $a,b \in \bQ _{>0}$ by Lemmas \ref{D-1} and \ref{E-1}, where $M$ is the effective $\bQ$-divisor and is given as in the Table \ref{DE+(1)} depending on one parameter $t$ and according to both the singularity type of $p$ on $S_{\kc}$ and the position of $\wtp$. 
Let $\gamma$ be the positive rational number defined by $\gamma := -(M)^2$. 
The value of $\gamma$ and its range are summarized in Table \ref{DE+(2)} depending on one parameter $t$ and according to both the singularity type of $p$ on $S_{\kc}$ and the position of $\wtp$ by easy computation. 
Let $E$ be the $(-1)$-curve on $\wS$ that as in Lemma \ref{ADE-(-1)}(2) or (3) according to the singularity type of $p$ on $S_{\kc}$. 
Noting that $0 \le (\wsL \cdot E) = a-2b$, we shall put $\alpha =2$. 
If $\gamma \le 3$, then we have $3 \alpha ^2 = 12 \ge 4 \gamma$, which contradicts Lemma \ref{Corti(3)}. 
Hence, we suppose $\gamma > 3$ in what follows. 
Then $p$ is of type $D_5$ on $S_{\kc}$ and lies on $M_5$ by Table \ref{DE+(2)}. 
In particular, we see $1 \le t \le 2$. 
We shall put $\beta := -(M \cdot M_5) = 2t-2$ and $(u,v):=(-t^2+3t-1,2t-3)$. 
Noting $u=-t^2+3t-1 >0$, we have: 
\begin{align*}
\alpha - u &= 2 - (-t^2+3t-1) = \left( t-\frac{3}{2}\right) ^2 +\frac{3}{4} > 0, \\
\alpha - u- v &= 2 - (-t^2+3t-1) - (2t-3) = (t-2)(t-3) \ge 0 , \\
\alpha u + \beta v - \gamma &= 2(-t^2+3t-1) + (2t-2)(2t-3) - (2t^2-4t+4) = 0
\end{align*}
and
\begin{align*}
4u^2+4uv+4v^2 -3\gamma &= 4(-t^2+3t-1)^2 + 4 (-t^2+3t-1)(2t-3) + 4 (2t-3)^2 - 3 (2t^2 -4t +4) \\
&=2(t-2)^2(2t^2-8t+5) \\
&\le 2(t-2)^2\{ 2t^2-8t+5 + (2t-1)\} \\
&= 4(t-2)^3(t-1) \\
&\le 0. 
\end{align*}
This implies that the inequalities (\ref{(4.2)}) hold simultaneously, which contradicts Lemma \ref{Corti(2)}(2). 
\end{proof}
\begin{table}[t]
\begin{tabular}{c}

\begin{minipage}[c]{1\hsize}
\begin{center}
\caption{Effective $\bQ$-divisor $M$ in Lemma \ref{sym(2)}}\label{DE+(1)}
\begin{tabular}{|c|c||l|c|} \hline
Singularity & Position of $\wtp$ & \multicolumn{1}{|c|}{Irreducible decomposition of $M$} & Range of $t$ \\ \hline \hline 

$D_5^+$ & $M_3 \cup M_4$ & $tM_1+tM_2+2tM_3+2M_4+M_5$ & $1 \le t \le \frac{3}{2}$ \\ \hline
$D_5^+$ & $M_5$ & $M_1+M_2+2M_3+2M_4+tM_5$ & $1 \le t \le 2$ \\ \hline \hline
$E_6^+$ & $M_5 \cup M_6$ & $tM_1+tM_2+2tM_3+2tM_4+3tM_5+2M_6$ & $1 \le t \le \frac{4}{3}$\\ \hline
\end{tabular}
\end{center}
\end{minipage}
\\

\begin{minipage}[c]{1\hsize}
\begin{center}
\caption{Value and range of $\gamma$ in Lemma \ref{sym(2)}}\label{DE+(2)}
\begin{tabular}{|c|c|c||c|c|} \hline
Singularity & Position of $\wtp$ & Range of $t$ & $\gamma$ & Range of $\gamma$ \\ \hline \hline

$D_5^+$ & $M_3 \cup M_4$& $1 \le t \le \frac{3}{2}$ & $4t^2-8t+6$ & $2 \le \gamma \le 3$ \\ \hline
$D_5^+$ & $M_5$ & $1 \le t \le 2 $ & $2t^2-4t+4$ & $2 \le \gamma \le 4$ \\ \hline \hline

$E_6^+$ & $M_5 \cup M_6$ & $1 \le t \le \frac{4}{3}$ & $6t^2-12t+8$ & $2 \le \gamma \le \frac{8}{3}$ \\ \hline
\end{tabular}
\end{center}
\end{minipage}

\end{tabular}
\end{table}
Finally, we treat the case that $p$ is of type $A_{6-2d}^+$ on $S$. 
If $p$ is of type $A_{6-2d}$ on $S_{\kc}$, then the type of $\wS$ is one of the following: 
\begin{itemize}
\item $d=2$ and $A_5+A_2$, $A_4+A_2$, $A_3+A_2+A_1$, $3A_2$, $A_3+A_2$, $2A_2+A_1$, $A_2+3A_1$, $2A_2$, $A_2+2A_1$, $A_2+A_1$ or $A_2$-type. 
\item $d=1$ and $2A_4$, $A_4+A_3$, $A_4+A_2+A_1$, $A_4+3A_1$, $A_4+A_2$, $A_4+2A_1$, $A_4+A_1$ or $A_4$-type. 
\end{itemize}
In particular, all singular points other than $p$ on $S_{\kc}$ are of type $A_n$ for some various possible values of $n$. 
Noting the above argument, we obtain the following lemma: 
\begin{lem}\label{sym(3)}
With the notation and the assumptions as above, then the singular point $p$ is not of type $A_{6-2d}^+$ on $S$ for $d =1,2$. 
\end{lem}
\begin{proof}
Suppose that the singular point $p$ is of type $A_{6-2d}^+$ on $S$. 
If $d=2$ and $\wS$ is of $A_2$-type, then $\wS$ is a weak del Pezzo surface of degree $2$ with $\rho _k(\wS )=2$. 
Hence, $\wS$ is $k$-minimal by Proposition \ref{minimal}(2) and contains the cylinder $\wU$. 
However, it is a contradiction to Proposition \ref{minimal}(3). 
In what follows, we shall treat other cases and consider the cases of $d=2$ and $d=1$ separately. 

{\it In the case $d=2$:}
Let $x_1,\dots ,x_{r}$ be all singular points on $S_{\kc}$ other than $p$, and let $M_{i,1},\dots ,M_{i,n(i)}$ be the irreducible components of the exceptional set on $\wS$ of the minimal resolution at $x_i$ for $i=1,\dots ,r$ such that the dual graph of $\sum _{i=1}^r\sum _{j=1}^{n(i)}M_{i,j}$ is as in (\ref{graphAA}). 
By using Proposition \ref{ADE-prop}, for $i=1,\dots ,r$, there exist two $(-1)$-curves $E_{i,1}$ and $E_{i,2}$ on $\wS _{\kc}$ such that $E_{i,1}+E_{i,2} \sim (-K_{\wS _{\kc}}) - (M_1+M_2) -\sum _{j=1}^{n(i)}M_{i,j}$. 
Then the dual graph of $M_1+M_2+E_{i,1}+E_{i,2}+\sum _{j=1}^{n(i)}M_{i,j}$ is as follows (cf. Example \ref{ex of prop(4-2)}), where ``$\circ$" and ``$\bullet$"mean a $(-2)$-curve and a $(-1)$-curve on $\wS _{\kc}$, respectively: 
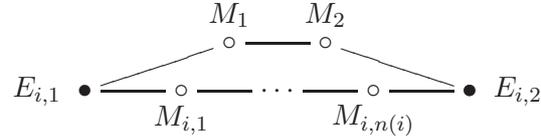
\begin{align*}
\xygraph{
\bullet ([]!{+(-.5,0)} {E_{i,1}}) (
        - []!{+(1.5,.5)} \circ ([]!{+(0,+.3)} {M_1}) -[r] \circ ([]!{+(0,+.3)} {M_2}) - []!{+(1.5,-.5)} \bullet ([]!{+(+.5,0)} {E_{i,2}}),
        - [r] \circ ([]!{+(0,-.3)} {M_{i,1}}) -[r] \cdots ([]!{+(0,+.3)} {}) -[r] \circ ([]!{+(0,-.3)} {M_{i,n(i)}})- [r] \circ ([]!{+(0,+.3)} {})
)}
\end{align*}
for $i=1,\dots, r$. 
Notice that $(E_{i,1}+E_{i,2} \cdot E_{i',1}+E_{i',2}) = -2\delta _{i,i'}$ since $(-K_{\wS} - M_1 - M_2 )^2=0$. 
Write $m(i) := \lceil \frac{n(i)}{2}\rceil$ for simplicity. 
Let $\tau :\wS \to V$ be the sequence of contractions of $(-1)$-curves and subsequently (smoothly) contractible curves in $\Supp \left( \sum _{i=1}^r\left (E_{i,1}+E_{i,2} + \sum _{j=1}^{n(i)}M_{i,j} \right) \right)$ such that 
the direct image of $E_{i,1}+E_{i,2}+\sum _{j=1}^{n(i)}M_{i,j}$ by $\tau$ is equal to $\tau (M_{i,m(i)})$ (resp. $\tau (M_{i,m(i)})+\tau (M_{i,m(i)+1})$) if $n(i)$ is odd (resp. even). 
In other words, all curves in the following dual graph are contracted by $\tau$ for $i=1,\dots ,r$: 
\begin{align*}
&\xygraph{
\bullet ([]!{+(0,.3)} {E_{i,1}}) -[r] \circ ([]!{+(0,.3)} {M_{i,1}}) -[r] {\cdots} -[r] \circ ([]!{+(0,.3)} {M_{i,m(i)-1}}) [rr] \circ ([]!{+(0,.3)} {M_{i,n(i)-m(i)+2}}) -[r] {\cdots} -[r] \circ ([]!{+(0,.3)} {M_{i,n(i)}}) -[r] \bullet ([]!{+(0,.3)} {E_{i,2}})
} &\text{if $m(i)>1$}; \\
&\xygraph{
\bullet ([]!{+(0,.3)} {E_{i,1}}) [rr]\bullet ([]!{+(0,.3)} {E_{i,2}})
} &\text{if $m(i)=1$}. 
\end{align*}
By construction, $\tau$ is defined over $k$ and $V$ is a smooth del Pezzo surface with $\rho _k(V)=2$ endowed with a structure of Mori conic bundle $\pi :V \to B$ over $k$ such that each $\tau _{\ast}(M_{i,m(i)})_{\kc}$ is included in a union of some closed fibers of $\pi _{\kc}$. 
Moreover, $\wtp$ is a $k$-rational point on $\wS$, so is its image via $\tau$. 
Thus, $B \simeq \bP ^1_k$ by Lemma \ref{Severi-Brauer}. 
In particular, we obtain $\Pic (V) _{\bQ} = \bQ [-K_V] \oplus \bQ [F]$, where $F$ is a general fiber of $\pi$. 
Let $\{ e_{i,j}\} _{1\le i \le r,\ 1 \le j \le m(i)-1}$ be the total transforms of all irreducible components on the exceptional set satisfying $(e_{i,j} \cdot M_{i,j}) < 0$ by $\tau$  for $i=1,\dots ,r$ and $j=1,\dots ,m(i)-1$, moreover, we set $e_{i,0} := E_{i,1}+E_{i,2}$. 
Note that each $e_{i,j}$ is defined over $k$, is uniquely determined and satisfies $(-K_{\wS} \cdot e_{i,j})=2$. Hence, we obtain: 
\begin{align*}
\Pic (\wS) _{\bQ} \subseteq \bQ [-K_{\wS}] \oplus \bQ [\widetilde{F}] \oplus \left( \bigoplus _{i=1}^{r} \bigoplus _{j=0}^{m(i)-1}\bQ [e_{i,j}] \right),
\end{align*}
where $\widetilde{F}$ is a total transform of $F$ by $\tau$. 
In particular, we can write: 
\begin{align*}
\wsL \sim _{\bQ} a(-K_{\wS})+b\widetilde{F}+\sum _{i=1}^{r} \sum _{j=0}^{m(i)-1}c_{i,j}e_{i,j}
\end{align*}
for some $a,b,c_{i,j} \in \bQ$. 
By construction, we obtain that $M_{i,j}+M_{i,m(i)-j+1} \sim e_{i,j}-e_{i,j-1}$ for $j=1,\dots ,m(i)-1$ and $M_{i,m(i)}$ (resp. $M_{i,m(i)}+M_{i,m(i)+1}$) is linearly equivalent to $\widetilde{F} -e_{i,m(i)-1}$ if $n(i)$ is odd (resp. even). 
Moreover, we notice $(\widetilde{F})^2 = (e_{i,j} \cdot \widetilde{F}) =0$ for any $i,j$. 
Hence, we have $c_{i,j}=0$ by virtue of $(\wsL \cdot M_{i,j}) =0$ for any $i,j$. 
On the other hand, since $E_{i,1}+E_{i,2} \sim e_{i,0}$, we have $a>0$ by virtue of $0 \le (\wsL \cdot e_{i,0})=2a$ and $0<(\wsL )^2=2a(a+2b)$. 
Moreover, we have $b>0$ by virtue of $0<(\wsL \cdot M_1+M_2) =b(\widetilde{F} \cdot M_1+M_2)$. 
Thus, we see $\wsL \sim _{\bQ} a(-K_{\wS}) + b\widetilde{F}$ as $a,b>0$, however, we obtain a contradiction by applying {\cite[Lemma 4.9]{Saw}}. 

{\it In the case $d=1$:} 
By Lemma \ref{ADE-(-1)}(1), there exists a $(-1)$-curve $E_0$ on $\wS$ such that $(E_0 \cdot M_i) = \delta _{1,i} + \delta _{4,i}$. 
Hence, we have the contraction $\tau _0:\wS \to W_2$ of $E_0$ defined over $k$ such that $W_2$ is a weak del Pezzo surface of degree $2$, moreover, this condition is as above case of $d=2$. 
Thus, by an argument similar to the above case with $d=2$, there exists a $0$-curve $\widetilde{F}$ on $\wS$ such that we can write: 
\begin{align*}
\wsL \sim _{\bQ} a(-K_{\wS})+b\widetilde{F}+c_0E_0
\end{align*}
for some $a,b,c_0 \in \bQ$. 
By the configuration of $\tau _0$, we see $(\widetilde{F} \cdot E_0)=0$ and $M_1+M_4 \sim \widetilde{F}-2E_0$. 
Hence, we have $c_0=0$ by virtue of $0=(\wsL \cdot M_1+M_4) = 2c_0$. 
Moreover, by an argument similar to the above case with $d=2$ we see $a,b>0$, which is a contradiction by applying {\cite[Lemma 4.9]{Saw}}. 
\end{proof}
As we already mentioned, the ``only if" part in Theorem \ref{main(1-3)}(4) follows from Lemmas \ref{sym(1)}, \ref{sym(2)} and \ref{sym(3)}. 

\subsection{Assumption for the ``if'' part of Theorem \ref{main(1-3)}(4)}\label{5-4}
In this subsection, in order to prove the ``if" part in Theorem \ref{main(1-3)}(4), we shall observe the assumption of this precisely. 
In other words, the purpose of this subsection is to show the following proposition: 
\begin{prop}\label{5-4(0)}
Let $S$ be a Du Val del Pezzo surface of the degree $d \le 2$ over $k$ with $\rho _k(S)=1$ not satisfying any condition (1), (2) nor (3) in Theorem \ref{main(1-3)}, and let $\sigma :\wS \to S$ be the minimal resolution over $k$. 
If $S_{\kc}$ has a singular point, which is $k$-rational, of type $A^-_n$, $D^-_n$ or $E^-_n$ on $S$, then the type of $\wS$ is one of the following:  
\begin{itemize}
\item $d=2$: $A_5+A_2$, $2A_3+A_1$, $2A_3$, $A_3+3A_1$, $3A_2$, $(A_5)'$, $(A_3+2A_1)''$, $A_2+3A_1$, $(A_3+A_1)'$, $A_3$ or $A_2$-type.  
\item $d=1$: $A_7+A_1$, $E_6+A_2$, $D_5+A_3$, $A_5+A_2+A_1$, $2A_4$, $(A_7)'$, $D_5+2A_1$, $A_5+A_2$, $E_6$, $(A_5+A_1)'$, $D_5$, $A_5$ or $A_4$-type. 
\end{itemize}
\end{prop}
In what follows, we will prove Proposition \ref{5-4(0)}. 
Let the notation and assumption be the same as in Proposition \ref{5-4(0)}. 
Then we can take a singular point $x_0$ on $S_{\kc}$, which is $k$-rational, of type $A_n^-$, $D_n^-$ or $E_n^-$. 
Let $r$ be the number of all singular points other than $x_0$ on $S_{\kc}$, which are $k$-rational, and let $x_1,\dots ,x_r$ be the singular points other than $x_0$ on $S_{\kc}$, which are $k$-rational. 
We shall consider two cases according to the degree $d$ of $S$ separately. 

At first, we shall treat the case $d=2$. 
Then we may assume that $x_0$ is of type $A_n^-$ on $S$ for some $2 \le n \le 5$ since $S$ does not satisfy any condition (1) nor (3) in Theorem \ref{main(1-3)}. 
Moreover, note that all singular points other than $x_0$ on $S_{\kc}$ are also necessarily of type $A_{n'}$ for some various possible values of $n'$. 
We obtain the following two lemmas: 
\begin{lem}\label{5-4(1)}
Let the notation and the assumptions be the same as above. 
If $r>0$, then $\wS$ is of $A_5+A_2$, $2A_3+A_1$, $2A_3$, $A_3+3A_1$ or $(A_3+A_1)'$-type. 
\end{lem}
\begin{proof}
Let $n(i)$ be the number such that the singular point $x_i$ is of type $A_{n(i)}$ for $i=1,\dots ,r$. 
Here, we may assume $n(1) \ge n(2) \ge \dots \ge n(r)$ by replacing the subscripts $i=1,\dots ,r$ as needed. 
Let $\{ M_{i,j}\} _{1 \le j \le n(i)}$ be all irreducible components of the exceptional set of the minimal resolution at $x_i$ for $i=0,1,\dots ,r$ with the configuration as in (\ref{graphAA}), where $n(0) := n$, and let $D$ be the divisor on $\wS _{\kc}$ defined by $D := (-K_{\wS _{\kc}}) - \sum _{i=0}^1 \sum _{j=1}^{n(i)} M_{i,j}$. 
Notice that $D$ is defined over $k$. 
Since the divisor $D$ is as in (a) in Table \ref{div D}, we see that $D$ satisfies the condition on divisors of either (A) or (B) by Proposition \ref{ADE-prop}(1). 

Assume that $D$ satisfies the condition (A). 
In other words, there exist two $(-1)$-curves $E_1$ and $E_2$ on $\wS _{\kc}$ such that $D \sim E_1+E_2$ (see Proposition \ref{ADE-prop}(2)). 
Hence, the dual graph of $\sum _{i=0}^1\sum _{j=1}^{n(i)}M_{i,j}+E_1+E_2$ is as follows, where ``$\circ$'' and ``$\bullet$'' mean a $(-2)$-curve and a $(-1)$-curve, respectively: 
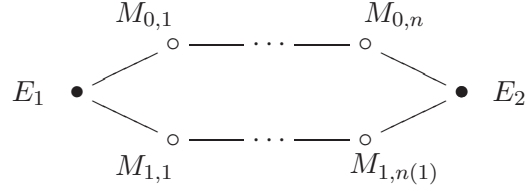
\begin{align*}
\xygraph{
\bullet ([]!{+(-.5,0)} {E_1}) (
        - []!{+(1,.5)} \circ ([]!{+(-.3,+.3)} {M_{0,1}}) -[r] \cdots -[r] \circ ([]!{+(+.3,+.3)} {M_{0,n}})- []!{+(1,-.5)} \bullet ([]!{+(+.5,0)} {E_2}),
        - []!{+(1,-.5)} \circ ([]!{+(-.3,-.3)} {M_{1,1}}) -[r] \cdots -[r] \circ ([]!{+(+.3,-.3)} {M_{1,n(1)}})- []!{+(1,.5)} \circ 
)}
\end{align*}
Since the singular point $x_0$ is of type $A_n^-$ on $S$ with $n \ge 2$, we see that $E_1$ and $E_2$ are defined over $k$, respectively. 
This implies that the two $\bQ$-divisors $E_1$ and $E_2$ are included in $\Pic (\wS )_{\bQ} = \bQ [-K_{\wS}] \oplus \left( \bigoplus _{i=0}^1 \bigoplus _{j=1}^{n(i)}\bQ [M_{i,j}] \right)$ since $\rho _k(S)=1$. 
Hence, the pair $(n,n(1))$ is $(5,2)$, $(2,5)$ or $(3,3)$ by Proposition \ref{ADE-prop}(6). 
If $(n,n(1))=(5,2)$ or $(2,5)$, then all singular points on $S_{\kc}$ are only $x_0$ and $x_1$ since there are at most seven $(-2)$-curves on $\wS _{\kc}$ by Lemma \ref{(-2)-curve}. 
Namely, $\wS$ is of $A_5+A_2$-type. 
If $(n,n(1))=(3,3)$, then there exists at most a singular point of type $A_1$ on $S_{\kc}$ other than $x_0$ and $x_1$ by a similar argument using Lemma \ref{(-2)-curve}. 
Namely, $\wS$ is then of $2A_3$ or $2A_3+A_1$-type. 

Assume that $D$ satisfies the condition (B). 
Then the pair $(n,n(1))$ is $(3,1)$, by Proposition \ref{ADE-prop}(7). 
In particular, we see $r=1$. 
Otherwise, supposing $r \ge 2$ and taking the divisor $(-K_{\wS _{\kc}}) - \sum _{i=1}^2 \sum _{j=1}^{n(i)} M_{i,j}$ on $\wS _{\kc}$, which is the divisor as in (a) in Table \ref{div D}, we have $n(2)=3$ by the argument similar to the above, however, it is a contradiction to $n(1) \ge n(2)$. 
Hence, if there exists a singular point on $S_{\kc}$ other than $x_0$ and $x_1$, then there exist exactly two singular points of type $A_1$ on $S_{\kc}$, which lie in the same $\Gal$-orbit. 
Indeed, there is no $A_3+mA_1$-type of $\wS$ for $m \ge 4$ by the classification of types of weak del Pezzo surfaces. 
Namely, $\wS$ is then of $A_3+A_1$ or $A_3+3A_1$-type. 
\end{proof}
\begin{lem}\label{5-4(2)}
Let the notation and the assumptions be the same as above. 
If $r=0$, then the following assertions hold: 
\begin{enumerate}
\item $x_0$ is not of type $A_4$ on $S_{\kc}$. Namely, $n=2,3$ or $5$. 
\item $\wS$ is not of $A_2+2A_1$-type. 
\item $\wS$ is of $(A_5)'$, $(A_3+2A_1)''$, $A_3$, $3A_2$, $A_2+3A_1$ or $A_2$-type. 
\end{enumerate}
\end{lem}
\begin{proof}
In (1), supposing that the singular point $x_0$ is of type $A_4$ on $S_{\kc}$, let $\{ M_j\} _{1 \le j \le 4}$ be all irreducible components of the exceptional set of the minimal resolution at $x_0$ with the configuration as in (\ref{graphAA}) and let $D$ be the divisor on $\wS _{\kc}$ defined by $D := (-K_{\wS _{\kc}})-(M_1+2M_2+2M_3+M_4)$, which is the divisor as in (b) in Table \ref{div D}. 
Notice that $D$ is defined over $k$. 
By Proposition \ref{ADE-prop}(1) and (8), we see that $D$ satisfies the condition (A). 
In particular, by Proposition \ref{ADE-prop}(2), there exist two $(-1)$-curves $E_2$ and $E_3$ on $\wS _{\kc}$ such that $D \sim E_2+E_3$. 
Hence, the dual graph of $\sum _{j=1}^4M_j+E_1+E_2$ is as follows, where ``$\circ$'' and ``$\bullet$'' mean a $(-2)$-curve and a $(-1)$-curve, respectively: 
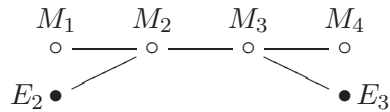
\begin{align*}
\xygraph{
     \circ ([]!{+(0,+.3)} {M_1}) 
 -[r] \circ ([]!{+(0,+.3)} {M_2}) 
( - []!{+(-1,-.5)} \bullet ([]!{+(-.3,0)} {E_2}), 
 -[r] \circ ([]!{+(0,+.3)} {M_3}) 
(  - []!{+(+1,-.5)} \bullet ([]!{+(+.3,0)} {E_3}), 
 -[r] \circ ([]!{+(0,+.3)} {M_4}) 
))}
\end{align*}
Since the singular point $x_0$ is of type $A_4^-$ on $S$ by assumption, $E_2$ and $E_3$ are defined over $k$, respectively. 
This implies that the two $\bQ$-divisors $E_1$ and $E_2$ are included in $\Pic (\wS )_{\bQ} = \bQ [-K_{\wS}] \oplus \left( \bigoplus _{j=1}^4\bQ [M_j] \right)$ since $\rho _k(S)=1$. 
However, it is a contradiction to Proposition \ref{ADE-prop}(6). 

In (2), supposing that $\wS$ is of $A_2+2A_1$-type, let $y_1$ and $y_2$ be two singular points of type $A_1$ on $S_{\kc}$, let $M_{0,1}$ and $M_{0,2}$ (resp. $M_{1,1}$, $M_{2,1}$) be all irreducible components of the exceptional set of the minimal resolution at $x_0$ (resp. $y_1$, $y_2$) with the configuration as in (\ref{graphAA}) and let $D_i$ be the divisor on $\wS _{\kc}$ defined by $D_i := (-K_{\wS _{\kc}})-(M_{0,1}+M_{0,2})-M_{i,1}$, which is the divisor as in (a) in Table \ref{div D} for $i=1,2$. 
By Proposition \ref{ADE-prop}(1) and (7), we see that $D_i$ satisfies the condition (A), for $i=1,2$. 
In particular, by Proposition \ref{ADE-prop}(2), there exist two $(-1)$-curves $E_{i,1}$ and $E_{i,2}$ on $\wS _{\kc}$ such that $D_i \sim E_{i,1}+E_{i,2}$. 
Hence, the dual graph of $M_{0,1}+M_{0,2}+M_{1,1}+M_{2,1}+ \sum _{i=1}^2\sum _{j=1}^2E_{i,j}$ is as follows, where ``$\circ$'' and ``$\bullet$'' mean a $(-2)$-curve and a $(-1)$-curve, respectively: 
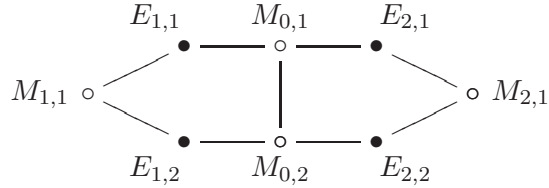
\begin{align*}
\xygraph{
\circ ([]!{+(-.5,0)} {M_{1,1}}) (
        - []!{+(1,.5)} \bullet ([]!{+(-.3,+.3)} {E_{1,1}}) -[r] \circ ([]!{+(0,+.3)} {M_{0,1}})( -[d] \circ ([]!{+(0,-.3)} {M_{0,2}}),
-[r] \bullet ([]!{+(+.3,+.3)} {E_{2,1}})- []!{+(1,-.5)} \circ ([]!{+(+.5,0)} {M_{2,1}})),
        - []!{+(1,-.5)} \bullet ([]!{+(-.3,-.3)} {E_{1,2}}) -[r] \circ ([]!{+(0,+.3)} {}) -[r] \bullet ([]!{+(+.3,-.3)} {E_{2,2}})- []!{+(1,.5)} \circ ([]!{+(0,+.3)} {})
)}
\end{align*}
Since the singular point $x_0$ is of type $A_2^-$ on $S$ by assumption, $M_{0,1}$ is defined over $k$. 
Hence, so is the union $E_{1,1} + E_{2,1}$. 
This implies that the $\bQ$-divisor $E_{1,1}+E_{2,1}$ is contained in $\Pic (\wS )_{\bQ} = \bQ [-K_{\wS}] \oplus \left( \bigoplus _{j=1}^2 \bQ [M_{0,j}] \right) \oplus  \left( \bigoplus _{i=1}^2 \bQ [M_{i,1}] \right)$ since $\rho _k(S)=1$. 
Hence, we have: 
\begin{align*}
E_{1,1}+E_{2,1} \sim _{\bQ} (-K_{\wS}) - \frac{1}{3}(2M_{0,1}+M_{0,2}) - \frac{1}{2}M_{1,1} - \frac{1}{2} M_{2,1} 
\end{align*}
by Lemma \ref{A-1} combined with the above graph, however, by the above formula, we then obtain $-2 = (E_{1,1}+E_{2,1})^2 = -\frac{5}{3}$, which is absurd. 

In (3), if $x_0$ is the only singular point of $S_{\kc}$, then we see that $\wS$ is of $(A_5)'$, $A_3$ or $A_2$-type by the assumptions and (1). 
In what follows, assume that there exists a singular point $y$ on $S_{\kc}$ other than $x_0$. 
Since $r=0$, there exists a singular point $y'$ other than $y$ on $S_{\kc}$ such that $y$ and $y'$ are included in the same $\Gal$-orbit. 
Moreover, since there are at most seven $(-2)$-curves on $\wS _{\kc}$ by Lemma \ref{(-2)-curve}, the singular point $y$ is of type $A_1$ or $A_2$ on $S_{\kc}$. 
If $y$ is of type $A_2$ on $S_{\kc}$, then all singular points on $S_{\kc}$ are only $x_0$, $y$ and $y'$, namely, $\wS$ is then of $3A_2$-type. 
In what follows, we can thus assume that any singular point on $\wS _{\kc}$ other than $x_0$ is of type $A_1$. 
Then $\wS$ is of $A_n+sA_1$-type for some integer $s$. 
In particular, we precisely see that $\wS$ is then of $(A_3+2A_1)''$ or $A_2+3A_1$-type by the classification of types of weak del Pezzo surfaces combined with (2). 
\end{proof}
Next, we shall treat the case of $d=1$. 
Notice that the singular point $x_0$ is of type $D_5^-$, $E_6^-$ or $A_n^-$ for some $2 \le n \le 7$, since $S$ does not satisfy any condition on singularities of (1) and (3) in Theorem \ref{main(1-3)}. 
If the singular point $x_0$ is of type $D_5^-$ or $E_6^-$ on $S$, then we obtain the following lemma by the argument similar to Lemma \ref{5-4(1)}: 
\begin{lem}\label{5-4(3)}
With the notation and the assumptions as above, assume further that $S_{\kc}$ has a singular point, which is $k$-rational, of type $D_5^-$ or $E_6^-$ on $S$, then the type of $\wS$ is one of the following according to the number of $r$: 
\begin{enumerate}
\item $r>0$: $D_5+A_3$ or $E_6+A_2$-type. 
\item $r=0$: $D_5+2A_1$, $D_5$ or $E_6$-type. 
\end{enumerate}
\end{lem}
\begin{proof}
By assumption of this lemma, we may assume that the singular point $x_0$ is of type $D_5^-$ or $E_6^-$ on $S$. 
We only treat the case where the singular point $x_0$ is of type $D_5^-$, the other cases are similar and left to the reader. 

In (1), let $\{ M_{i,j}\} _{1 \le j \le n(i)}$ be all irreducible components of the exceptional set of the minimal resolution at $x_i$ for $i=0,1$ with the configuration as in (\ref{graphDA}), where $n(0) := 5$, and let $D$ be the divisor on $\wS _{\kc}$ defined by $D := 2(-K_{\wS _{\kc}})-(2M_{0,1}+2M_{0,2}+3M_{0,3}+2M_{0,4}+M_{0,5})-\sum _{j=1}^{n(1)}M_{1,j}$, which is the divisor as in (f) in Table \ref{div D}. 
Notice that $D$ is defined over $k$. 
By the argument similar to Lemma \ref{5-4(1)}, we see that $n(1)=3$. 
In particular, all singular points on $S_{\kc}$ are only $x_0$ and $x_1$ since there are at most eight $(-2)$-curves on $\wS _{\kc}$ by Lemma \ref{(-2)-curve}. 
Namely, $\wS$ is then of $D_5+A_3$-type. 

In (2), if there exists a singular point other than $x_0$ on $S_{\kc}$, then there exist exactly two singular points of type $A_1$ on $S_{\kc}$, which lie in the same $\Gal$-orbit, by a similar argument using Lemma \ref{(-2)-curve}. 
Namely, $\wS$ is then of $D_5$ or $D_5+2A_1$-type. 
Indeed, there is no $D_5+3A_1$-type of $\wS$. (We also note that there is no $E_6+2A_1$-type of $\wS$. )
\end{proof}
In what follows, we shall treat the case that $S_{\kc}$ does not allow any singular point, which is $k$-rational, of type $D_5^-$ or $E_6^-$. 
Thus, the singular point $x_0$ is of type $A_n^-$ for some $2 \le n \le 7$. 
By the argument similar to Lemmas \ref{5-4(1)} and \ref{5-4(2)}, we obtain the following two lemmas: 
\begin{lem}\label{5-4(4)}
Let the notation and the assumptions be the same as above. 
If $r>0$, then the type of $\wS$ is one of the following according to the number of $r$: 
\begin{enumerate}
\item $r \ge 2$: $A_5+A_2+A_1$-type. 
\item $r=1$: $A_7+A_1$, $A_5+A_2$, $2A_4$ or $(A_5+A_1)'$-type. 
\end{enumerate}
\end{lem}
\begin{proof}
Let $\{ M_{i,j}\} _{1 \le j \le n(i)}$ be all irreducible components of the exceptional set of the minimal resolution at $x_i$ for $i=0,1,\dots ,r$ with the configuration as in (\ref{graphAA}), where $n(0) := n$. 

In (1), let $D$ be the divisor on $\wS _{\kc}$ defined by $D := 2(-K_{\wS _{\kc}})-\sum _{i=0}^2\sum _{j=1}^{n(i)}M_{i,j}$, which is the divisor as in (c) in Table \ref{div D}. 
Notice that $D$ is defined over $k$. 
By the argument similar to Lemma \ref{5-4(1)}, we see that $(n,n(1),n(2))=(5,2,1)$ or $(2,5,1)$. 
In particular, all singular points on $S_{\kc}$ are only $x_0$, $x_1$ and $x_2$ since there are at most eight $(-2)$-curves on $\wS _{\kc}$ by Lemma \ref{(-2)-curve}. 
Namely, $\wS$ is then of $A_5+A_2+A_1$-type. 

In (2), at first, we assume that $n \ge 4$. 
Let $D$ be the divisor on $\wS _{\kc}$ defined by $D := 2(-K_{\wS _{\kc}})+(M_{0,1}+M_{0,n}) -2\sum _{j=1}^n M_{0,j} -\sum _{j=1}^{n(1)}M_{1,j}$, which is the divisor as in (d) in Table \ref{div D}. 
Notice that $D$ is defined over $k$. 
By the argument similar to Lemma \ref{5-4(1)}, we see that $(n,n(1))=(7,1)$, $(5,2)$ or $(4,4)$ (resp. $(5,1)$) if $D$ satisfies the condition (A) (resp. (B)). 
In particular, all singular points on $S_{\kc}$ defined over $k$ are only $x_0$ and $x_1$ by a similar argument using Lemma \ref{(-2)-curve}. 
Namely, $\wS$ is then of $A_7+A_1$, $A_5+A_2$, $2A_4$ or $(A_5+A_1)'$-type. 
Here, note that there is no $A_5+3A_1$-type of $\wS$ by the classification of types of weak del Pezzo surfaces. 

On the other hand, if $n<4$, then we have $n(1) \ge 4$ since $S$ does not satisfy the condition on singularities of (3) in Theorem \ref{main(1-3)}. 
The same argument as above applies with the role of $i=0$ and $i=1$ exchanged. 
\end{proof}
\begin{lem}\label{5-4(5)}
Let the notation and the assumptions be the same as above. 
If $r=0$, then the following assertions hold: 
\begin{enumerate}
\item $x_0$ is not of type $A_2$, $A_3$ nor $A_6$ on $S_{\kc}$. Namely, $n=4,5$ or $7$. 
\item $\wS$ is not of $A_5+2A_1$ nor $A_4+2A_1$-type. 
\item $\wS$ is of $(A_7)'$, $A_5$ or $A_4$-type. 
\end{enumerate}
\end{lem}
\begin{proof}
In (1), since $r=0$, for any singular point $y$ other than $x_0$ on $S_{\kc}$, there exists a singular point $y'$ other than $y$ on $S_{\kc}$ such that $y$ and $y'$ are included in the same $\Gal$-orbit. 
Moreover, since there are at most eight $(-2)$-curves on $\wS _{\kc}$ by Lemma \ref{(-2)-curve}, the singular point $y$ is of type $A_1$, $A_2$ or $A_3$ on $S_{\kc}$. 
Hence, we see that $n \ge 4$ since $S$ does not satisfy the condition of (3) in Theorem \ref{main(1-3)}. 

Supposing that the singular point $x_0$ is of type $A_6$ on $S_{\kc}$, let $\{ M_j\} _{1 \le j \le 6}$ be all irreducible components of the exceptional set of the minimal resolution at $x_0$ with the configuration as in (\ref{graphAA}). 
Letting $D$ be the divisor on $\wS _{\kc}$ defined by $D := 2(-K_{\wS _{\kc}})-(M_1+2M_2+3M_3+3M_4+2M_5+M_6)$, which is the divisor as in (e) in Table \ref{div D} and is defined over $k$, we obtain a contradiction by the argument similar to Lemma \ref{5-4(2)}(1). 

In (2), otherwise, let $y_1$ and $y_2$ be two singular points of type $A_1$ on $S_{\kc}$, let $\{ M_{0,j}\} _{1 \le j \le n}$ (resp. $M_{2,1}$, $M_{3,1}$) be all irreducible components of the exceptional set of the minimal resolution at $x_0$ (resp. $y_1$, $y_2$) with the configuration as in (\ref{graphAA}). 
Letting $D$ be the divisor on $\wS _{\kc}$ defined by $D := 2(-K_{\wS _{\kc}})-\sum _{j=1}^nM_{0,j} -M_{1,1}-M_{2,1}$, which is the divisor as in (c) in Table \ref{div D} and is defined over $k$, we obtain a contradiction by the argument similar to Lemma \ref{5-4(2)}(1). 

In (3), by the classification of types of weak del Pezzo surfaces combined with the assumption $n \ge 4$, $\wS$ is of $A_n+sA_1$-type for some integer $s =0$ or $2$. 
Moreover, we precisely see that $\wS$ is then of $(A_7)''$, $A_5$ or $A_4$-type by (2) and a similar argument using Lemma \ref{(-2)-curve}. 
\end{proof}
Proposition \ref{5-4(0)} follows from Lemmas \ref{5-4(1)}, \ref{5-4(2)}, \ref{5-4(3)}, \ref{5-4(4)} and \ref{5-4(5)}. 

Conversely, for any type of weak del Pezzo surface in the list of Proposition \ref{5-4(0)}, there exists certainly a Du Val del Pezzo surface $S$ over $k$ with $\rho _k(S)=1$ admitting a singular point of type $A^-_n$, $D^-_n$ or $E^-_n$ such that its minimal resolution $\wS$ is of this type. 
Indeed, we can explicitly construct a birational morphism $\tau :\wS \to \bF _2$ over $k$ and the contraction $\sigma : \wS \to S$ of all $(-2)$-curves, so that $S$ is the Du Val del Pezzo surface of Picard rank one (see also \S \S \ref{5-5}, for detailed constructions of such morphisms $\tau$). 
Here, $\wS$ is a weak del Pezzo surface such that its Picard number is the number, which is summarized in ``$\rho _k(\wS )$" in Table \ref{list(3)} according to the type of $\wS$. 
Furthermore, the singularity types of all singular points on $S_{\kc}$, which are $k$-rational, are summarized in ``$k$-rat.\ sing.\," in Table \ref{list(3)} according to the type of $\wS$. 
As an example, in the case $d=2$ and of $3A_2$-type, $S_{\kc}$ has three singular points of type $A_2$. 
If $\rho _k(S)=1$, then one is $k$-rational and of type $A_2^-$ on $S$, however, the others lie in the same $\Gal$-orbit, namely $\rho _k(\wS ) = \rho _k(S)+4=5$. 
\begin{table}[htbp]
\begin{center}
\caption{Types of $\wS$ in the ``if" part of Theorem \ref{main(1-3)}(4)}\label{list(3)}
\begin{tabular}{|c|c|c||c|c|c|} \hline

\multirow{3}{*}{$d$} & Type & \multirow{3}{*}{Dual graph} & \multirow{3}{*}{$d$} & Type & \multirow{3}{*}{Dual graph} \\
 & ($k$-rat.\ sing.) & & & ($k$-rat.\ sing.) & \\ 
 & $\rho _k(\wS )$ & & & $\rho _k(\wS )$ & \\ \hline \hline

\multirow{3}{*}{$2$} & $A_5+A_2$ & 
\multirow{3}{*}{$$\xygraph{
\circ ([]!{+(0,-.3)} {M}) -[]!{+(.5,0)} \circ ([]!{+(0,-.3)} {C_3}) -[]!{+(.5,0)} \bullet ([]!{+(0,0)} {}) -[]!{+(.5,0)} \circ ([]!{+(0,0)} {})-[]!{+(.5,0)} \cdots ([]!{+(0,-.3)} {\underbrace{\qquad \qquad}_{5\text{-vertices}}})-[]!{+(.5,0)} \circ ([]!{+(0,0)} {})
}$$} & 

\multirow{3}{*}{$2$} & $2A_3+A_1$ & 
\multirow{3}{*}{$$\xygraph{
(-[]!{+(0,-.25)} \circ ([]!{+(0,-.3)} {F})(-[]!{+(-.5,0)} \circ ([]!{+(0,-.3)} {M}),-[]!{+(.5,0)} \circ ([]!{+(0,-.3)} {C_2})-[]!{+(.5,0)} \bullet ([]!{+(0,0)} {})-[]!{+(.5,0)} \circ ([]!{+(0,0)} {})-[]!{+(.5,0)} \circ ([]!{+(0,0)} {})-[]!{+(.5,0)} \circ ([]!{+(0,0)} {})),(-[]!{+(0,.25)} \bullet ([]!{+(0,0)} {}) -[]!{+(-.5,0)} \circ ([]!{+(0,0)} {}))
}$$}  \rule[0mm]{0mm}{4mm} \\ 

 & ($A^-_5$, $A^-_2$) & & & ($A^-_3$, $A^-_3$, $A^+_1$) &  \rule[0mm]{0mm}{4mm}  \\ 
 & $8$ & & & $8$ &  \rule[0mm]{0mm}{4mm}  \\ \hline

\multirow{3}{*}{$2$} & $2A_3$ & 
\multirow{3}{*}{$\xygraph{
\circ ([]!{+(0,-.3)} {M}) -[]!{+(.5,0)} \circ ([]!{+(0,-.3)} {F}) (-[]!{+(.5,0)} \circ ([]!{+(0,-.3)} {C_2})-[]!{+(.5,0)} \bullet ([]!{+(0,0)} {})-[]!{+(.5,0)} \circ ([]!{+(0,0)} {})-[]!{+(.5,0)} \circ ([]!{+(0,0)} {})-[]!{+(.5,0)} \circ ([]!{+(0,0)} {}),(-[]!{+(-.5,.25)} \bullet ([]!{+(0,0)} {}),-[]!{+(.5,.25)} \bullet ([]!{+(0,0)} {})))
}$} &

\multirow{3}{*}{$2$} & $A_3+3A_1$ & 
\multirow{3}{*}{$\xygraph{
(-[]!{+(0,-.125)} \circ ([]!{+(0,-.3)} {F}) (-[]!{+(-.5,0)} \circ ([]!{+(0,-.3)} {M}),-[]!{+(.5,0)} \circ ([]!{+(0,-.3)} {C_2}) (-[]!{+(.5,.25)} \bullet ([]!{+(0,0)} {})-[]!{+(.5,0)} \circ ([]!{+(0,0)} {}),-[]!{+(.5,-.25)} \bullet ([]!{+(0,0)} {})-[]!{+(.5,0)} \circ ([]!{+(0,0)} {})),-[]!{+(0,.5)} \bullet ([]!{+(0,0)} {}) -[]!{+(-.5,0)} \circ ([]!{+(0,0)} {}))
}$}  \rule[0mm]{0mm}{4mm} \\ 

 & ($A^-_3$, $A^-_3$) & & & ($A^-_3$, $A^+_1$) &  \rule[0mm]{0mm}{4mm} \\ 
 & $7$ & & & $6$ &  \rule[0mm]{0mm}{4mm} \\ \hline

\multirow{3}{*}{$2$} & $3A_2$ & 
\multirow{3}{*}{$\xygraph{
\circ ([]!{+(0,-.3)} {M}) -[]!{+(.5,0)} \circ ([]!{+(0,-.3)} {C_3}) (-[]!{+(.5,.25)} \bullet ([]!{+(0,0)} {})-[]!{+(.5,0)} \circ ([]!{+(0,0)} {})-[]!{+(.5,0)} \circ ([]!{+(0,0)} {}),-[]!{+(.5,-.25)} \bullet ([]!{+(0,0)} {})-[]!{+(.5,0)} \circ ([]!{+(0,0)} {})-[]!{+(.5,0)} \circ ([]!{+(0,0)} {}))
}$} &

\multirow{3}{*}{$2$} & $(A_5)'$ & 
\multirow{3}{*}{$\xygraph{
\circ ([]!{+(0,-.3)} {M_1}) -[]!{+(.5,0)} \circ ([]!{+(0,-.3)} {\Gamma _1}) -[]!{+(.5,0)} \circ ([]!{+(0,0)} {}) (-[]!{+(0,.5)} \bullet ([]!{+(0,0)} {}),-[]!{+(.5,0)} \circ ([]!{+(0,-.3)} {\Gamma_2}) -[]!{+(.5,0)} \circ ([]!{+(0,-.3)} {M_2}))
}$}  \rule[0mm]{0mm}{4mm} \\ 

 & ($A^-_2$) & & & ($A^-_5$) &  \rule[0mm]{0mm}{4mm} \\ 
 & $5$ & & & $6$ &  \rule[0mm]{0mm}{4mm} \\ \hline

\multirow{3}{*}{$2$} & $(A_3+2A_1)''$ & 
\multirow{3}{*}{$\xygraph{
\circ ([]!{+(0,-.3)} {M}) -[]!{+(.5,0)} \circ ([]!{+(0,-.3)} {F}) (-[]!{+(.5,0)} \circ ([]!{+(0,-.3)} {C_2}) (-[]!{+(.5,.25)} \bullet ([]!{+(0,0)} {}) -[]!{+(.5,0)} \circ ([]!{+(0,0)} {}),-[]!{+(.5,-.25)} \bullet ([]!{+(0,0)} {})-[]!{+(.5,0)} \circ ([]!{+(0,0)} {})),(-[]!{+(-.5,.25)} \bullet ([]!{+(0,0)} {}),-[]!{+(.5,.25)} \bullet ([]!{+(0,0)} {})))
}$} &

\multirow{3}{*}{$2$} & $A_2+3A_1$ & 
\multirow{3}{*}{$\xygraph{
\circ ([]!{+(0,-.3)} {C_3})
 ((-[]!{+(.5,.25)} \bullet ([]!{+(0,0)} {})-[]!{+(.5,0)} \circ ([]!{+(0,0)} {}), -[]!{+(.5,0)} \bullet ([]!{+(0,0)} {})-[]!{+(.5,0)} \circ ([]!{+(0,0)} {}), -[]!{+(.5,-.25)} \bullet ([]!{+(0,0)} {})-[]!{+(.5,0)} \circ ([]!{+(0,0)} {})), -[]!{+(-.5,0)} \circ ([]!{+(0,-.3)} {M}))
}$}  \rule[0mm]{0mm}{4mm} \\ 

 & ($A^-_3$) & & & ($A^-_2$) &  \rule[0mm]{0mm}{4mm} \\ 
 & $5$ & & & $4$ &  \rule[0mm]{0mm}{4mm} \\ \hline

\multirow{3}{*}{$2$} & $(A_3+A_1)'$ & 
\multirow{3}{*}{$\xygraph{
(-[]!{+(0,-.25)} \circ ([]!{+(0,-.3)} {\Gamma})(-[]!{+(-.5,0)} \circ ([]!{+(0,-.3)} {M_1}),-[]!{+(.5,0)} \circ ([]!{+(0,-.3)} {M_2})),(-[]!{+(0,.25)} \bullet ([]!{+(0,0)} {}) -[]!{+(-.5,0)} \circ ([]!{+(0,0)} {}))
}$} &

\multirow{3}{*}{$2$} & $A_3$ & 
\multirow{3}{*}{$\xygraph{
\circ ([]!{+(0,-.3)} {M_1}) -[]!{+(.5,0)} \circ ([]!{+(0,-.3)} {\Gamma}) (-[]!{+(.5,0)} \circ ([]!{+(0,-.3)} {M_2}),(-[]!{+(-.5,.25)} \bullet ([]!{+(0,0)} {}),-[]!{+(.5,.25)} \bullet ([]!{+(0,0)} {})))
}$}  \rule[0mm]{0mm}{4mm} \\ 

 & ($A^-_3$, $A^+_1$) & & & ($A^-_3$) &  \rule[0mm]{0mm}{4mm} \\ 
 & $5$ & & & $4$ &  \rule[0mm]{0mm}{4mm} \\ \hline

\multirow{3}{*}{$2$} & $A_2$ & 
\multirow{3}{*}{$\xygraph{
\circ ([]!{+(0,.3)} {M}) -[]!{+(.5,0)} \circ ([]!{+(0,.3)} {C_3}) (-[]!{+(-.5,-.25)} \bullet ([]!{+(0,0)} {})[]!{+(.5,0)} \cdots ([]!{+(0,-.3)} {\underbrace{\qquad \qquad}_{6\text{-vertices}}})[]!{+(.5,0)} \bullet ([]!{+(0,0)} {}),-[]!{+(.5,-.25)} \bullet ([]!{+(0,0)} {}))
}$} &

\multirow{3}{*}{$1$} & $A_7+A_1$ & 
\multirow{3}{*}{$\xygraph{
(-[]!{+(0,-.25)} \circ ([]!{+(0,0)} {}) (-[]!{+(.5,0)} \circ ([]!{+(0,0)} {})-[]!{+(.5,0)} \circ ([]!{+(0,0)} {})-[]!{+(.5,0)} \circ ([]!{+(0,0)} {}),-[]!{+(-.5,0)} \circ ([]!{+(0,0)} {})-[]!{+(-.5,0)} \circ ([]!{+(0,0)} {}) (-[]!{+(-.5,0)} \circ ([]!{+(0,0)} {}),-[]!{+(0,.5)} \bullet -[]!{+(-.5,0)} \circ)),-[]!{+(0,.25)} \bullet ([]!{+(.3,0)} {\widetilde{E}}))
}$}  \rule[0mm]{0mm}{4mm} \\ 

 & ($A^-_2$) & & & ($A^-_7$, $A^+_1$) &  \rule[0mm]{0mm}{4mm} \\ 
 & $3$ & & & $9$ &  \rule[0mm]{0mm}{4mm} \\ \hline

\multirow{3}{*}{$1$} & $E_6+A_2$ & 
\multirow{3}{*}{$\xygraph{
(-[]!{+(0,-.25)} \circ ([]!{+(0,-.2)} {\scriptsize \textcircled{\tiny 5}}),-[]!{+(0,.25)} \circ ([]!{+(0,.25)} {\scriptsize \textcircled{\tiny 4}}) (-[]!{+(.5,0)} \circ ([]!{+(0,.25)} {\scriptsize \textcircled{\tiny 3}})-[]!{+(.5,0)} \circ ([]!{+(0,.3)} {L_3}),-[]!{+(-.5,0)} \circ ([]!{+(0,.25)} {\scriptsize \textcircled{\tiny 2}})-[]!{+(-1,0)} \circ ([]!{+(0,.3)} {L_1})-[]!{+(0,-.5)} \bullet ([]!{+(0,-.2)} {\scriptsize \textcircled{\tiny 9}})-[]!{+(.5,0)} \circ ([]!{+(0,-.2)} {\scriptsize \textcircled{\tiny 8}})-[]!{+(.5,0)} \circ ([]!{+(0,-.2)} {\scriptsize \textcircled{\tiny 7}})))
}$} &

\multirow{3}{*}{$1$} & $D_5+A_3$ & 
\multirow{3}{*}{$\xygraph{
(-[]!{+(0,-.25)} \circ ([]!{+(0,0)} {})(-[]!{+(-.5,0)} \circ ([]!{+(0,0)} {}),-[]!{+(.5,0)} \circ ([]!{+(0,0)} {})-[]!{+(.5,0)} \bullet ([]!{+(0,0)} {})-[]!{+(.5,0)} \circ ([]!{+(0,0)} {})-[]!{+(.5,0)} \circ ([]!{+(0,0)} {})-[]!{+(.5,0)} \circ ([]!{+(0,0)} {})),(-[]!{+(0,.25)} \circ ([]!{+(0,0)} {}) (-[]!{+(-.5,0)} \circ ([]!{+(0,0)} {}),-[]!{+(.5,0)} \bullet ([]!{+(.3,0)} {\widetilde{E}})))
}$}  \rule[0mm]{0mm}{4mm} \\ 

 & ($E^-_6$, $A^-_2$) & & & ($D^-_5$, $A^-_3$) &  \rule[0mm]{0mm}{4mm} \\ 
 & $9$ & & & $9$ &  \rule[0mm]{0mm}{4mm} \\ \hline

\multirow{3}{*}{$1$} & $A_5+A_2+A_1$ & 
\multirow{3}{*}{$\xygraph{
(-[]!{+(0,.25)} \bullet ([]!{+(0,.25)} {\scriptsize \textcircled{\tiny 10}})-[]!{+(.5,0)} \circ ([]!{+(0,.25)} {\scriptsize \textcircled{\tiny 9}})-[]!{+(.5,0)} \circ ([]!{+(0,.25)} {\scriptsize \textcircled{\tiny 8}}),-[]!{+(0,-.25)} \circ ([]!{+(0,-.25)} {\scriptsize \textcircled{\tiny 7}}) (-[]!{+(.5,0)} \circ ([]!{+(0,-.25)} {\scriptsize \textcircled{\tiny 6}}),-[]!{+(-1,0)} \circ ([]!{+(0,-.25)} {\scriptsize \textcircled{\tiny 3}}) (-[]!{+(-.5,0)} \circ ([]!{+(0,-.25)} {\scriptsize \textcircled{\tiny 2}})-[]!{+(-.5,0)} \circ ([]!{+(0,-.25)} {\scriptsize \textcircled{\tiny 1}}),-[]!{+(0,.5)} \bullet ([]!{+(0,.25)} {\scriptsize \textcircled{\tiny 5}})-[]!{+(-.5,0)} \circ ([]!{+(0,.25)} {\scriptsize \textcircled{\tiny 4}}))))
}$} &

\multirow{3}{*}{$1$} & $2A_4$ & 
\multirow{3}{*}{$\xygraph{
(-[]!{+(0,.25)} \circ ([]!{+(0,.25)} {\scriptsize \textcircled{\tiny 8}}) (-[]!{+(-1,0)} \circ ([]!{+(0,.25)} {\scriptsize \textcircled{\tiny 1}}),-[]!{+(.5,0)} \circ ([]!{+(0,.25)} {\scriptsize \textcircled{\tiny 7}})-[]!{+(.5,0)} \circ ([]!{+(0,.25)} {\scriptsize \textcircled{\tiny 6}}),-[]!{+(0,-.5)} \bullet ([]!{+(0,-.25)} {\scriptsize \textcircled{\tiny 13}})-[]!{+(.5,0)} \circ ([]!{+(0,-.25)} {\scriptsize \textcircled{\tiny 12}})-[]!{+(.5,0)} \circ ([]!{+(0,-.25)} {\scriptsize \textcircled{\tiny 11}})-[]!{+(.5,0)} \circ ([]!{+(0,-.25)} {\scriptsize \textcircled{\tiny 10}})-[]!{+(.5,0)} \circ ([]!{+(0,-.25)} {\scriptsize \textcircled{\tiny 9}}))
}$}  \rule[0mm]{0mm}{4mm} \\ 

 & ($A^-_5$, $A^-_2$, $A^+_1$) & & & ($A^-_4$, $A^-_4$) &  \rule[0mm]{0mm}{4mm} \\
 & $9$ & & & $9$ & \rule[0mm]{0mm}{4mm} \\ \hline

\multirow{3}{*}{$1$} & $(A_7)'$ & 
\multirow{3}{*}{$\xygraph{
(-[]!{+(0,-.25)} \circ ([]!{+(0,0)} {}) (-[]!{+(.5,0)} \circ ([]!{+(0,0)} {})-[]!{+(.5,0)} \circ ([]!{+(0,0)} {})-[]!{+(.5,0)} \circ ([]!{+(0,0)} {}),-[]!{+(-.5,0)} \circ ([]!{+(0,0)} {})-[]!{+(-.5,0)} \circ ([]!{+(0,0)} {})-[]!{+(-.5,0)} \circ ([]!{+(0,0)} {})),-[]!{+(0,.25)} \bullet ([]!{+(.3,0)} {\widetilde{E}}))
}$} &

\multirow{3}{*}{$1$} & $D_5+2A_1$ & 
\multirow{3}{*}{$\xygraph{
(-[]!{+(0,-.125)} \circ ([]!{+(0,-.3)} {F}) (-[]!{+(-.5,0)} \circ ([]!{+(0,-.3)} {M}),-[]!{+(.5,0)} \circ ([]!{+(0,-.3)} {C_2}) (-[]!{+(.5,.25)} \bullet ([]!{+(0,0)} {})-[]!{+(.5,0)} \circ ([]!{+(0,0)} {}),-[]!{+(.5,-.25)} \bullet ([]!{+(0,0)} {})-[]!{+(.5,0)} \circ ([]!{+(0,0)} {})),-[]!{+(0,.5)} \circ ([]!{+(0,0)} {}) (-[]!{+(-.5,0)} \circ ([]!{+(0,0)} {}),-[]!{+(.5,0)} \bullet ([]!{+(.2,0)} {\widetilde{E}})))
}$}  \rule[0mm]{0mm}{4mm} \\ 

 & ($A^-_7$) & & & ($D^-_5$) &  \rule[0mm]{0mm}{4mm} \\ 
 & $8$ & & & $7$ &  \rule[0mm]{0mm}{4mm} \\ \hline

\multirow{3}{*}{$1$} & $A_5+A_2$ & 
\multirow{3}{*}{$\xygraph{
(-[]!{+(0,.25)} \bullet ([]!{+(0,.25)} {\scriptsize \textcircled{\tiny 10}})-[]!{+(.5,0)} \circ ([]!{+(0,.25)} {\scriptsize \textcircled{\tiny 9}})-[]!{+(.5,0)} \circ ([]!{+(0,.25)} {\scriptsize \textcircled{\tiny 8}}),-[]!{+(0,-.25)} \circ ([]!{+(0,-.25)} {\scriptsize \textcircled{\tiny 7}}) (-[]!{+(.5,0)} \circ ([]!{+(0,-.25)} {\scriptsize \textcircled{\tiny 6}}),-[]!{+(-1,0)} \circ ([]!{+(0,-.25)} {\scriptsize \textcircled{\tiny 3}}) (-[]!{+(-.5,0)} \circ ([]!{+(0,-.25)} {\scriptsize \textcircled{\tiny 2}})-[]!{+(-.5,0)} \circ ([]!{+(0,-.25)} {\scriptsize \textcircled{\tiny 1}}),(-[]!{+(-.5,.5)} \bullet ([]!{+(0,.25)} {\scriptsize \textcircled{\tiny 5}}),-[]!{+(.5,.5)} \bullet ([]!{+(0,.25)} {\scriptsize \textcircled{\tiny 4}})))))
}$} &

\multirow{3}{*}{$1$} & $E_6$ & 
\multirow{3}{*}{$\xygraph{
(-[]!{+(0,-.25)} \circ ([]!{+(0,0)} {}) (-[]!{+(.5,0)} \circ ([]!{+(0,0)} {})-[]!{+(.5,0)} \circ ([]!{+(0,0)} {}),-[]!{+(-.5,0)} \circ ([]!{+(0,0)} {})-[]!{+(-.5,0)} \circ ([]!{+(0,0)} {})),-[]!{+(0,.25)} \circ ([]!{+(0,0)} {})-[]!{+(.5,0)} \bullet ([]!{+(.3,0)} {\widetilde{E}}))
}$}  \rule[0mm]{0mm}{4mm} \\ 

 & ($A^-_5$, $A^-_2$) & & & ($E^-_6$) &  \rule[0mm]{0mm}{4mm} \\ 
 & $8$ & & & $7$ &  \rule[0mm]{0mm}{4mm} \\ \hline

\multirow{3}{*}{$1$} & $(A_5+A_1)'$ & 
\multirow{3}{*}{$\xygraph{
((-[]!{+(0,-.25)} \bullet ([]!{+(0,-.25)} {\scriptsize \textcircled{\tiny 5}})-[]!{+(-.5,0)} \circ ([]!{+(0,-.25)} {\scriptsize \textcircled{\tiny 4}}),(-[]!{+(0,.25)} \circ ([]!{+(0,.25)} {\scriptsize \textcircled{\tiny 3}}) (-[]!{+(1,0)} \circ ([]!{+(0,.25)} {\scriptsize \textcircled{\tiny 7}})-[]!{+(.5,0)} \circ ([]!{+(0,.25)} {\scriptsize \textcircled{\tiny 6}}),-[]!{+(-.5,0)} \circ ([]!{+(0,.25)} {\scriptsize \textcircled{\tiny 2}})-[]!{+(-.5,0)} \circ ([]!{+(0,.25)} {\scriptsize \textcircled{\tiny 1}})))
}$} &

\multirow{3}{*}{$1$} & $D_5$ & 
\multirow{3}{*}{$\xygraph{(
(-[]!{+(0,-.25)} \circ ([]!{+(0,0)} {}) (-[]!{+(.5,0)} \circ ([]!{+(0,0)} {}),-[]!{+(-.5,0)} \circ ([]!{+(0,0)} {})),-[]!{+(0,.25)} \circ ([]!{+(0,0)} {}) (-[]!{+(-.5,0)} \circ ([]!{+(0,0)} {}),-[]!{+(.5,0)} \bullet ([]!{+(.3,0)} {\widetilde{E}})))
}$}  \rule[0mm]{0mm}{4mm} \\ 

 & ($A^-_5$, $A^+_1$) & & & ($D^-_5$) & \rule[0mm]{0mm}{4mm} \\ 
 & $7$ & & & $6$ & \rule[0mm]{0mm}{4mm} \\ \hline

\multirow{3}{*}{$1$} & $A_5$ & 
\multirow{3}{*}{$\xygraph{
(-[]!{+(-.25,-.25)} \bullet ([]!{+(0,-.25)} {\scriptsize \textcircled{\tiny 4}}) -[]!{+(.5,.5)} \circ ([]!{+(0,.25)} {\scriptsize \textcircled{\tiny 3}})
(-[]!{+(.5,-.5)} \bullet ([]!{+(0,-.25)} {\scriptsize \textcircled{\tiny 5}}),(-[]!{+(1,0)} \circ ([]!{+(0,.25)} {\scriptsize \textcircled{\tiny 7}})-[]!{+(.5,0)} \circ ([]!{+(0,.25)} {\scriptsize \textcircled{\tiny 6}}),-[]!{+(-.5,0)} \circ ([]!{+(0,.25)} {\scriptsize \textcircled{\tiny 2}})-[]!{+(-.5,0)} \circ ([]!{+(0,.25)} {\scriptsize \textcircled{\tiny 1}})))
}$} &

\multirow{3}{*}{$1$} & $A_4$ & 
\multirow{3}{*}{$\xygraph{
\circ ([]!{+(0,-.25)} {\scriptsize \textcircled{\tiny 1}}) -[]!{+(1,0)} \circ ([]!{+(0,-.25)} {\scriptsize \textcircled{\tiny 8}}) -[]!{+(.5,0)} \circ ([]!{+(0,-.25)} {\scriptsize \textcircled{\tiny 7}}) -[]!{+(.5,0)} \circ ([]!{+(0,-.25)} {\scriptsize \textcircled{\tiny 6}})
}$}  \rule[0mm]{0mm}{4mm} \\ 

 & ($A^-_5$) & & & ($A^-_4$) & \rule[0mm]{0mm}{4mm} \\ 
 & $6$ & & & $5$ & \rule[0mm]{0mm}{4mm} \\ \hline
\end{tabular}
\end{center}
\end{table}

At the end of this subsection, we shall present the notation in Table \ref{list(3)}. 
The meanings of ``$k$-rat.\ sing.\," and ``$\rho _k(\wS )$" have already been presented. 
``Dual graph" in Table \ref{list(3)} means the dual graph corresponding to the union of all $(-2)$-curves and some $(-1)$-curves on $\wS$. 
Here, ``$\circ$" and ``$\bullet$"mean a $(-2)$-curve and a $(-1)$-curve, respectively. 
For all types of $\wS$ in the list in Table \ref{list(3)}, the union of the $(-1)$-curves on $\wS$ corresponding to all vertices $\bullet$ in Table \ref{list(3)} certainly exists and is further defined over $k$. 
The existence of these curves can be shown by using Propositions \ref{minimal}(1) and \ref{ADE-prop} with suitable choices of divisors on $\wS _{\kc}$. 
These dual graphs will be used for the construction of cylinders on the surfaces $\wS$ in \S \S \ref{5-5}. 

\subsection{Proof for the ``if'' part of Theorem \ref{main(1-3)}(4)}\label{5-5}
Let the notation and assumptions be the same as in Proposition \ref{5-4(0)}. 
Then the type of $\wS$ is one of those in Table \ref{list(3)}. 
In this subsection, we shall show the ``if" part of Theorem \ref{main(1-3)}(4). 
In other words, we will explicitly construct a cylinder on $S$ according to the type in the list in Table \ref{list(3)}. 
\begin{lem}\label{5-5(1)}
With the notation and assumptions as in Proposition \ref{5-4(0)}, assume further that one of the following conditions holds: 
\begin{itemize}
\item $d=2$ and $\wS$ is of one of those in the list of Table \ref{list(3)}; 
\item $d=1$ and $\wS$ is of $A_7+A_1$, $D_5+2A_1$, $(A_7)'$, $D_5+2A_1$, $E_6$ or $D_5$-type. 
\end{itemize}
Then $S$ contains a cylinder. 
\end{lem}
\begin{proof}
In the case of $d=2$, let $N$ be the union of all $(-2)$-curves on $\wS$. 
At first, we shall deal with the cases in which $\wS$ is of $(A_5)'$, $(A_3+A_1)'$ and $A_3$-type. 
For these cases, we can take a birational morphism $\tau :\wS \to W_4$, which is the compositions of the successive contractions of the $(-1)$-curves corresponding to the vertices $\bullet$ in the dual graph in Table \ref{list(3)} and that of the proper transform of the branch components such that all curves corresponding to vertices with no label in the dual graph in Table \ref{list(3)} are contracted by $\tau$, according to the type of $\wS$, where $W_4$ is a weak del Pezzo surface of degree $4$ and of $(2A_1)_<$-type over $k$. 
Note that, by construction, $\tau$ is defined over $k$. 
Moreover, the image of the reduced curves corresponding to all vertices of this dual graph via $\tau$ is the union of either $M_1+M_2+\Gamma$ or $M_1+M_2+\Gamma _1+\Gamma _2$, where $M_1$ and $M_2$ are $(-2)$-curves on $W_{4,\kc}$, $\Gamma$ is a $0$-curve on $W_{4,\kc}$, and $\Gamma _1$ and $\Gamma _2$ are $(-1)$-curves on $W_{4,\kc}$ meeting transversally at a point. 
Notice that these curves on $W_{4,\kc}$ are in one-to-one correspondence to these vertices with a label of this dual graph.  
Since two $(-2)$-curves on $W_{4,\kc}$ admit a $k$-rational point respectively, $W_4$ contains a cylinder, which contains $\tau _{\ast}(N)$ in its boundary (see also \S \S \ref{3-2}). 
Thus, $\wS$ contains a cylinder $\wU$, which contains $N$ in its boundary. 
Therefore, we see that $S$ contains a cylinder $\sigma (\wU ) \simeq \wU$. 

In what follows, we shall deal with the remaining cases. 
For all remaining cases, we can take a birational morphism $\tau :\wS \to \bF _2$, which is the compositions of the successive contractions of the $(-1)$-curves corresponding to the vertices $\bullet$ in the dual graph in Table \ref{list(3)} and that of the proper transform of the branch components such that all curves corresponding to vertices with no label in the dual graph in Table \ref{list(3)} are contracted by $\tau$, according to the type of $\wS$. 
Note that, by construction, $\tau$ is defined over $k$. 
Moreover, the image of the reduced curves corresponding to all vertices of this dual graph via $\tau$ is the union of either $M + F + C_2$ or $M + C_3$, where $M$ is the $(-2)$-curve on $\bF _2$, $F$ is a closed fiber of the $\bP ^1$-bundle $\bF _2 \to \bP ^1_k$ and $C_n$ is a rational curve on $\bF _2$ with $C_n \sim M+nF$ for $n=2,3$. 
Notice that these curves on $\bF _2$ are in one-to-one correspondence to these vertices with a label of this dual graph.  
For all cases, $\bF _2$ contains a cylinder, whose boundary includes the above union of curves, by Lemma \ref{cylinder F2}. 
Thus, we see that $S$ contains a cylinder by an argument similar to the above. 

In (2), for all cases, the dual graph in Table \ref{list(3)} corresponding to the type of $\wS$ contains a vertex with a label written $\widetilde{E}$. 
This vertex corresponds to a $(-1)$-curve on $\wS _{\kc}$, which is defined over $k$. 
Letting $\widetilde{E}$ be this $(-1)$-curve on $\wS$, we can take the contraction $\tau : \wS \to W_2$ of $\widetilde{E}$ over $k$, so that $W_2$ is a weak del Pezzo surface of degree $2$, whose type is one of those in the list in Table \ref{list(3)}, moreover, the point $\tau (\widetilde{E})$ lies on a curve, which corresponds to a vertex with no label in the dual graph in Table \ref{list(3)} according to the type of $W_2$. 
Thus, we see that $S$ contains a cylinder by using (1). 
\end{proof}
In order to deal with all remaining cases, we shall recall how to construct cylinders in del Pezzo surfaces with Du Val singularities found in {\cite[\S \S 4.2--4.3]{CPW16b}}. 
More precisely, we construct two birational morphisms  $g:\check{S} \to \wS _{\kc}$ and $h:\check{S} \to \bP ^2_{\kc}$ over $\kc$ (but not necessarily defined over $k$) in such a way that there exists a suitable cylinder $U$ in $\bP ^2_{\kc}$, which would be preserved via $g \circ h^{-1}: \bP ^2_{\kc} \dashrightarrow \wS _{\kc}$ and $(g \circ h^{-1})(U) \cap \Supp (N) = \emptyset$, where $N$ is the union of all $(-2)$-curves on $\wS _{\kc}$. 
In particular, $S_{\kc}$ contains the cylinder $(\sigma \circ g \circ h^{-1})(U)$. 
In the following lemmas (Lemmas \ref{5-5(3)}, \ref{5-5(4)} and \ref{5-5(5)}), in order to show that above argument is still working well over $k$, we shall prove that $g$ and $h$ are defined over $k$. 
In the proofs for Lemmas \ref{5-5(4)} and \ref{5-5(5)}, we look at the corresponding dual graphs in Table \ref{list(3)} and {\cite[Table 1]{CPW16b}}. 
We note that the numbering something like \textcircled{\scriptsize $i$} in Table \ref{list(3)} corresponds to that in {\cite[Table 1]{CPW16b}}. 
\begin{lem}\label{5-5(3)}
Let the notation and assumptions be the same as in Proposition \ref{5-4(0)}. 
If $d=1$ and $\wS$ is of $E_6+A_2$, $A_5+A_2+A_1$, $2A_4$ or $A_5+A_2$-type, then $S$ contains a cylinder. 
\end{lem}
\begin{proof}
For all cases, we see that any $(-2)$-curve on $\wS _{\kc}$ is defined over $k$ by the configuration of singular points on $S_{\kc}$ (see also Table \ref{list(3)}). 
In particular, any point meeting two $(-2)$-curves on $\wS _{\kc}$ is also defined over $k$. 
Then we can construct a birational morphism $g:\check{S} \to \wS _{\kc}$, whose $\check{S}$ is that as in {\cite[\S \S 4.2]{CPW16b}} according to the type of $\wS$, defined over $k$. 
Indeed, we shall consider a sequence of some blow-ups at some $k$-rational points starting at an intersection point of two $(-2)$-curves on $\wS _{\kc}$ (according to the type of $\wS$) and including infinitely near points such that we obtain the configuration of ``Construction" in {\cite[Table 1]{CPW16b}} according to the types of $\wS$. 
Moreover, we immediately have a birational morphism $h:\check{S} \to \bP ^2_{\kc}$, which plays same role as $h$  in {\cite[\S \S 4.2]{CPW16b}}. 
This $h$ is clearly defined over $k$. 
Therefore, we see that $S$ contains a cylinder. 
\end{proof}
\begin{rem}
In Lemma \ref{5-5(3)}, if $\wS$ is of $E_6+A_2$, $A_5+A_2+A_1$ or $2A_4$-type, then we could have also inferred the same result from the fact that $g:\check{S} \to \wS _{\kc}$ and $h:\check{S} \to \bP ^2_{\kc}$ are clearly defined over $k$, where $g$ and $h$ are those as in {\cite[\S \S 4.2]{CPW16b}}. 
Indeed, for these types, all $(-1)$-curves and $(-2)$-curves on $\wS _{\kc}$ are defined over $k$ since $\rho _k(\wS ) = \rho _{\kc}(\wS _{\kc}) = 9$. 
\end{rem}
\begin{lem}\label{5-5(4)}
Let the notation and assumptions be the same as in Proposition \ref{5-4(0)}. 
If $d=1$ and $\wS$ is of $(A_5+A_1)'$ or $A_5$-type, then $S$ contains a cylinder. 
\end{lem}
\begin{proof}
Let $M_i$ be the smooth rational curve on $\wS _{\kc}$ corresponding to the vertex with a label written \textcircled{\scriptsize $i$} in the weighted dual graph of Table \ref{list(3)}. 
There exists a $(-1)$-curve $\widetilde{E}$ on $\wS _{\kc}$, which is defined over $k$, such that $(\widetilde{E} \cdot M_i) = \delta _{1,i} + \delta _{6,i}$ by Lemma \ref{ADE-(-1)}(1). 
Hence, we obtain the birational morphism $\tau :\wS \to W_4$ over $k$ with the reduced exceptional divisor $M_4+M_5+\widetilde{E}$, so that $W_4$ is a weak del Pezzo surface of degree $4$ and of $(2A_1)_<$-type. 
Notice that $\tau _{\ast} (M_2)_{\kc}$ and $\tau _{\ast} (M_7)_{\kc}$ (resp. $\tau _{\ast} (M_1)_{\kc}$ and $\tau _{\ast} (M_6)_{\kc}$) are $(-2)$-curves (resp. $(-1)$-curves) on $W_{4,\kc}$. 
By Proposition \ref{minimal}(1), we know that $\tau _{\ast}(M_7)_{\kc}$ meets exactly four $(-1)$-curves such that one is $\tau _{\ast}(M_6)_{\kc}$. 
Let $E$ be the union of three $(-1)$-curves meeting $\tau _{\ast} (M_7)_{\kc}$ other than $\tau _{\ast} (M_6)_{\kc}$ on $W_{4,\kc}$. 
Noting that $E$ is defined over $k$, so is $\tau ^{-1}_{\ast}(E)$. 
Moreover, $\tau ^{-1}_{\ast}(E)_{\kc}$ consists of three $(-1)$-curve on $\wS _{\kc}$ corresponding to curves with a label written \textcircled{\scriptsize 8}, \textcircled{\scriptsize 9}, \textcircled{\scriptsize 10} in {\cite[Table 1]{CPW16b}}. 
Thus, we can construct two birational morphisms $g:\check{S} \to \wS _{\kc}$ and $h:\check{S} \to \bP ^2_{\kc}$, which play same role as in $g$ and $h$ in {\cite[\S \S 4.2]{CPW16b}}, defined over $k$ (see the following weighted dual graph): 
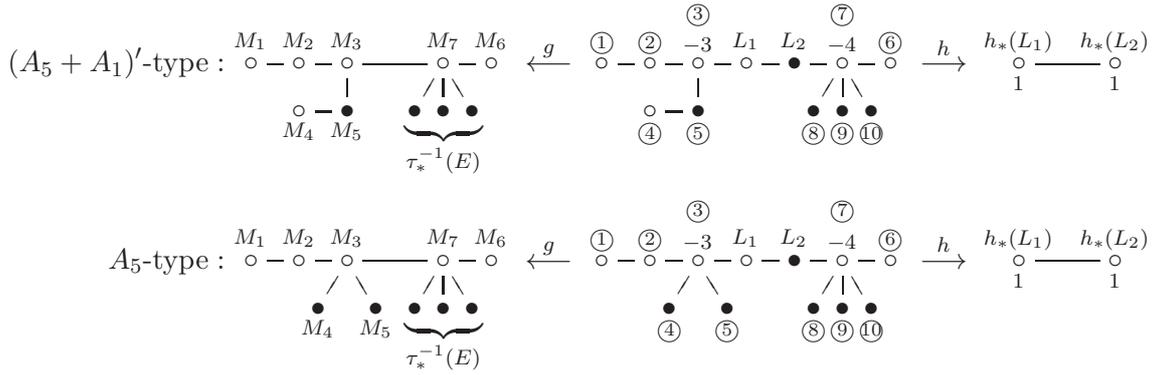
\begin{align*}
(A_5+A_1)'\text{-type}: 
&\xygraph{
\circ ([]!{+(0,.2)} {^{M_1}}) - []!{+(.5,0)} \circ ([]!{+(0,.2)} {^{M_2}}) - []!{+(.5,0)} \circ ([]!{+(0,.2)} {^{M_3}}) ((- []!{+(0,-.5)} \bullet ([]!{+(0,-.25)} {^{M_5}}) - []!{+(-.5,0)} \circ ([]!{+(0,-.25)} {^{M_4}}) ), - [r] \circ ([]!{+(0,.2)} {^{M_7}})
((- []!{+(-.3,-.5)} \bullet , - []!{+(.3,-.5)} \bullet ), - []!{+(0,-.5)} \bullet ([]!{+(0,-.4)} {\underbrace{\qquad \ \ }_{\tau ^{-1}_{\ast}(E)}})) - []!{+(.5,0)} \circ ([]!{+(0,.2)} {^{M_6}})
))}
\overset{g}{\longleftarrow}
\xygraph{
\circ ([]!{+(0,.2)} {\scriptsize \textcircled{\tiny 1}}) - []!{+(.5,0)} \circ ([]!{+(0,.2)} {\scriptsize \textcircled{\tiny 2}}) - []!{+(.5,0)} \circ ([]!{+(0,.15)} {^{-3}}) ([]!{+(0,.5)} {\scriptsize \textcircled{\tiny 3}}) ((- []!{+(0,-.5)} \bullet ([]!{+(0,-.25)} {\scriptsize \textcircled{\tiny 5}}) - []!{+(-.5,0)} \circ ([]!{+(0,-.25)} {\scriptsize \textcircled{\tiny 4}}) ), - []!{+(.5,0)} \circ ([]!{+(0,.2)} {^{L_1}}) - []!{+(.5,0)} \bullet ([]!{+(0,.2)} {^{L_2}}) - []!{+(.5,0)} \circ ([]!{+(0,.15)} {^{-4}}) ([]!{+(0,.5)} {\scriptsize \textcircled{\tiny 7}})
((- []!{+(-.3,-.5)} \bullet ([]!{+(0,-.25)} {\scriptsize \textcircled{\tiny 8}}), - []!{+(.3,-.5)} \bullet ([]!{+(0,-.25)} {\scriptsize \textcircled{\tiny 10}})), - []!{+(0,-.5)} \bullet ([]!{+(0,-.25)} {\scriptsize \textcircled{\tiny 9}})) - []!{+(.5,0)} \circ ([]!{+(0,.2)} {\scriptsize \textcircled{\tiny 6}})
))}
\overset{h}{\longrightarrow}
\xygraph{\circ ([]!{+(0,-.2)} {_{1}}) ([]!{+(0,.2)} {^{h_{\ast}(L_1)}}) -[r] \circ ([]!{+(0,-.2)} {_{1}}) ([]!{+(0,.2)} {^{h_{\ast}(L_2)}})}\\
A_5\text{-type}: 
&\xygraph{
\circ ([]!{+(0,.2)} {^{M_1}}) - []!{+(.5,0)} \circ ([]!{+(0,.2)} {^{M_2}}) - []!{+(.5,0)} \circ ([]!{+(0,.2)} {^{M_3}}) ((- []!{+(-.3,-.5)} \bullet ([]!{+(0,-.25)} {^{M_4}}) , - []!{+(.3,-.5)} \bullet ([]!{+(0,-.25)} {^{M_5}}) ), - [r] \circ ([]!{+(0,.2)} {^{M_7}})
((- []!{+(-.3,-.5)} \bullet , - []!{+(.3,-.5)} \bullet ), - []!{+(0,-.5)} \bullet ([]!{+(0,-.4)} {\underbrace{\qquad \ \ }_{\tau ^{-1}_{\ast}(E)}})) - []!{+(.5,0)} \circ ([]!{+(0,.2)} {^{M_6}})
))}
\overset{g}{\longleftarrow}
\xygraph{
\circ ([]!{+(0,.2)} {\scriptsize \textcircled{\tiny 1}}) - []!{+(.5,0)} \circ ([]!{+(0,.2)} {\scriptsize \textcircled{\tiny 2}}) - []!{+(.5,0)} \circ ([]!{+(0,.15)} {^{-3}}) ([]!{+(0,.5)} {\scriptsize \textcircled{\tiny 3}}) ((- []!{+(-.3,-.5)} \bullet ([]!{+(0,-.25)} {\scriptsize \textcircled{\tiny 4}}) , - []!{+(.3,-.5)} \bullet ([]!{+(0,-.25)} {\scriptsize \textcircled{\tiny 5}}) ), - []!{+(.5,0)} \circ ([]!{+(0,.2)} {^{L_1}}) - []!{+(.5,0)} \bullet ([]!{+(0,.2)} {^{L_2}}) - []!{+(.5,0)} \circ ([]!{+(0,.15)} {^{-4}}) ([]!{+(0,.5)} {\scriptsize \textcircled{\tiny 7}})
((- []!{+(-.3,-.5)} \bullet ([]!{+(0,-.25)} {\scriptsize \textcircled{\tiny 8}}), - []!{+(.3,-.5)} \bullet ([]!{+(0,-.25)} {\scriptsize \textcircled{\tiny 10}})), - []!{+(0,-.5)} \bullet ([]!{+(0,-.25)} {\scriptsize \textcircled{\tiny 9}})) - []!{+(.5,0)} \circ ([]!{+(0,.2)} {\scriptsize \textcircled{\tiny 6}})
))}
\overset{h}{\longrightarrow}
\xygraph{\circ ([]!{+(0,-.2)} {_{1}}) ([]!{+(0,.2)} {^{h_{\ast}(L_1)}}) -[r] \circ ([]!{+(0,-.2)} {_{1}}) ([]!{+(0,.2)} {^{h_{\ast}(L_2)}})}
\end{align*}
Here, in the above graph, the numbering something like \textcircled{\scriptsize $i$} corresponds to that in {\cite[Table 1]{CPW16b}} and vertices ``$\circ$'' and ``$\bullet$'', whose weights are omitted, mean a $(-2)$-curve and a $(-1)$-curve, respectively. 
Therefore, we see that $S$ contains a cylinder. 
\end{proof}
\begin{lem}\label{5-5(5)}
Let the notation and assumptions be the same as in Proposition \ref{5-4(0)}. 
If $d=1$ and $\wS$ is of $A_4$-type, then $S$ contains a cylinder. 
\end{lem}
\begin{proof}
Let $M_i$ be the $(-2)$-curve on $\wS$ corresponding to the vertex with a label written \textcircled{\scriptsize $i$} in the dual graph of Table \ref{list(3)}. 
There exists a $(-1)$-curve $\widetilde{E}$ on $\wS _{\kc}$, which is defined over $k$, such that $(\widetilde{E} \cdot M_i) = \delta _{1,i} + \delta _{6,i}$ by Lemma \ref{ADE-(-1)}(1). 
Hence, we have the contraction $\tau _1:\wS \to W_2$ of $\widetilde{E}$ over $k$, so that $W_2$ is a weak del Pezzo surface of degree $2$ and of $A_2$-type. 
Notice that $\tau _{1,\ast}(M_7)_{\kc}$ and $\tau _{1,\ast}(M_8)_{\kc}$ (resp. $\tau _{1,\ast}(M_1)_{\kc}$ and $\tau _{1,\ast}(M_6)_{\kc}$) are $(-2)$-curves (resp. $(-1)$-curves) on $W_{2,\kc}$. 
By Proposition \ref{minimal}(1), we know that $\tau _{1,\ast}(M_8)_{\kc}$ meets exactly six $(-1)$-curves such that one is the $\tau _{1,\ast}(M_1)_{\kc}$. 
Let $E$ be the union of five $(-1)$-curves meeting $\tau _{1,\ast}(M_8)_{\kc}$ other than $\tau _{1,\ast} (M_1)_{\kc}$ on $W_{2,\kc}$. 
Noting that $E$ is defined over $k$, so is $\tau ^{-1}_{1,\ast}(E)$. 
Moreover, $\tau ^{-1}_{1,\ast}(E)_{\kc}$ consists of five $(-1)$-curves on $\wS _{\kc}$ corresponding to curves with a label written \textcircled{\scriptsize 9}--\textcircled{\scriptsize 13} in {\cite[Table 1]{CPW16b}}. 
On the other hand, we have the contraction $\tau _2:W_2 \to \bF _2$ of $\tau _{1,\ast}(M_1)+E$ over $k$. 
Set $M := \tau _{\ast}(M_7)$, $F_0 := \tau _{\ast}(M_6)$ and $C_3 := \tau _{\ast}(M_8)$, where $\tau := \tau _2 \circ \tau _1: \wS \to W_2 \to \bF _2$. 
Then we see $\Pic (\bF _2) = \bZ [M] \oplus \bZ [F_0]$ and $C_3 \sim M + 3F_0$ (cf. Lemma \ref{5-5(1)}(1)). 
Since $(F_0 \cdot C_3) = 1$, $F_0$ and $C_3$ meet transversely at a point, say $p$, which is $k$-rational. 
Moreover, we see that there exists a unique curve $C_2$ on $\bF _2$ such that $C_2 \sim M+2F$ and $i(C_2,C_3;p)=3$, where $i(C_2,C_3;p)$ is the local intersection multiplicity at $p$ of $C_2$ and $C_3$. 
Notice that $C_2$ is defined over $k$. 
Moreover, $\tau ^{-1}_{\ast}(C_2)$, which is also defined over $k$, corresponds to the curve with a label written \textcircled{\scriptsize 5} in {\cite[Table 1]{CPW16b}}. 
Thus, we can construct two birational morphisms $g:\check{S} \to \wS _{\kc}$ and $h:\check{S} \to \bP ^2_{\kc}$, which play same role as in $g$ and $h$ in {\cite[\S \S 4.2]{CPW16b}}, defined over $k$ (see the following weighted dual graph): 
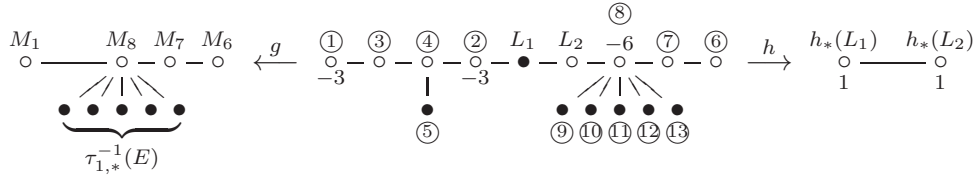
\begin{align*}
\xygraph{
\circ ([]!{+(0,.2)} {^{M_1}}) -[r] \circ ([]!{+(0,.2)} {^{M_8}})
((((- []!{+(-.3,-.5)} \bullet , - []!{+(.3,-.5)} \bullet),(- []!{+(-.6,-.5)} \bullet, - []!{+(.6,-.5)} \bullet)), - []!{+(0,-.5)} \bullet ([]!{+(0,-.4)} {\underbrace{\qquad \qquad}_{\tau ^{-1}_{1,\ast}(E)}})), - []!{+(.5,0)}
\circ ([]!{+(0,.2)} {^{M_7}}) - []!{+(.5,0)} \circ ([]!{+(0,.2)} {^{M_6}}))}
\overset{g}{\longleftarrow}
\xygraph{
\circ ([]!{+(0,.2)} {\scriptsize \textcircled{\tiny 1}}) ([]!{+(0,-.2)} {^{-3}}) - []!{+(.5,0)} \circ ([]!{+(0,.2)} {\scriptsize \textcircled{\tiny 3}}) - []!{+(.5,0)} \circ ([]!{+(0,.2)} {\scriptsize \textcircled{\tiny 4}}) ( - []!{+(0,-.5)} \bullet ([]!{+(0,-.25)} {\scriptsize \textcircled{\tiny 5}}), - []!{+(.5,0)} \circ ([]!{+(0,.2)} {\scriptsize \textcircled{\tiny 2}}) ([]!{+(0,-.2)} {^{-3}}) - []!{+(.5,0)} \bullet ([]!{+(0,.2)} {^{L_1}}) - []!{+(.5,0)} \circ ([]!{+(0,.2)} {^{L_2}}) - []!{+(.5,0)} \circ ([]!{+(0,.15)} {^{-6}}) ([]!{+(0,.5)} {\scriptsize \textcircled{\tiny 8}})
((((- []!{+(-.3,-.5)} \bullet ([]!{+(0,-.25)} {\scriptsize \textcircled{\tiny 10}}), - []!{+(.3,-.5)} \bullet ([]!{+(0,-.25)} {\scriptsize \textcircled{\tiny 12}})),(- []!{+(-.6,-.5)} \bullet ([]!{+(0,-.25)} {\scriptsize \textcircled{\tiny 9}}), - []!{+(.6,-.5)} \bullet ([]!{+(0,-.25)} {\scriptsize \textcircled{\tiny 13}}))), - []!{+(0,-.5)} \bullet ([]!{+(0,-.25)} {\scriptsize \textcircled{\tiny 11}})), - []!{+(.5,0)}
\circ ([]!{+(0,.2)} {\scriptsize \textcircled{\tiny 7}}) - []!{+(.5,0)} \circ ([]!{+(0,.2)} {\scriptsize \textcircled{\tiny 6}})))}
\overset{h}{\longrightarrow}
\xygraph{\circ ([]!{+(0,-.2)} {_{1}}) ([]!{+(0,.2)} {^{h_{\ast}(L_1)}}) -[r] \circ ([]!{+(0,-.2)} {_{1}}) ([]!{+(0,.2)} {^{h_{\ast}(L_2)}})}
\end{align*}
Here, in the above graph, the numbering something like \textcircled{\scriptsize $i$} corresponds to that in {\cite[Table 1]{CPW16b}} and vertices ``$\circ$'' and ``$\bullet$'', whose weights are omitted, mean a $(-2)$-curve and a $(-1)$-curve, respectively. 
Therefore, we see that $S$ contains a cylinder. 
\end{proof}

The ``if" part of Theorem \ref{main(1-3)}(4) follows from Proposition \ref{5-4(0)} and Lemmas \ref{5-5(1)}, \ref{5-5(3)}, \ref{5-5(4)} and \ref{5-5(5)}. 

\section{Examples}\label{6}

In this section, we shall present some examples of Du Val del Pezzo surfaces of Picard rank one and canonical del Pezzo fibrations. 

At first, we treat some examples of Du Val del Pezzo surfaces of Picard rank one over $k$, moreover, we shall discuss whether these surfaces contain or not a cylinder. 
\begin{eg}\label{eg1}
Put $\zeta := \frac{-1+\sqrt{-3}}{2}$ and let $S$ be the cubic surface over $\bQ$ defined by: 
\begin{align*}
S := \left( 12z^2w -2x^3-y^3-4w^3 + 6xyw = 0\right) \subseteq \bP ^3_{\bQ} = \Proj (\bQ [x,y,z,w]). 
\end{align*}
Then $S _{\overline{\bQ}}$ has exactly three singular points $[\sqrt[3]{2}\zeta ^i\!:\!\sqrt[3]{4} \zeta ^{2i}\!:\!0\!:\!1] \in \bP ^3_{\overline{\bQ}}$ of type $A_1$ for $i=0,1,2$ (see also Remark \ref{rem-eg1}). 
Let $\sigma : \wS \to S$ be the minimal resolution over $\bQ$. 
Then there exists the blow-down $\tau : \wS \to S_6$ over $\bQ$ such that $S_6$ is a smooth del Pezzo surface of degree $6$. 
Hence, $S_{6,\overline{\bQ}}$ has six $(-1)$-curves, say $\{ E_i \} _{1\le i \le 6}$. 
Moreover, the proper transform of these $(-1)$-curve by $\tau \circ \sigma ^{-1}$ are defined by the following equations: 
\begin{align*}
\sqrt[3]{2}\zeta ^ix=y,\ x = \pm \frac{\sqrt[3]{2}}{3}\zeta ^i(\zeta -1)z + \sqrt[3]{2}\zeta ^iw
\end{align*}
for $i=0,1,2$. 
Since all $(-1)$-curves on $S _{6,\overline{\bQ}}$ lie in the same ${\rm Gal}(\overline{\bQ}/\bQ)$-orbit, $S_6$ is $\bQ$-minimal, in particular, we obtain $\rho _{\bQ}(S_6) = 1$. 
By construction of $\sigma$ and $\tau$, we also obtain $\rho _{\bQ}(S) = 1$. 
Thus, $S$ does not contain a cylinder by Theorem \ref{main(1-2)}. 
Indeed, $S_{\overline{\bQ}}$ does not allow any singular point which is $\bQ$-rational (see also Tables \ref{list(1-1)} and \ref{list(1-2)}). 
On the other hand, we know that $S_{\overline{\bQ}}$ contains a cylinder by {\cite[Theorem 1.5]{CPW16b}}. 
This implies that any cylinder on $S_{\overline{\bQ}}$ is not defined over $\bQ$. 
\end{eg}
\begin{rem}\label{rem-eg1}
Let $S$ and $\zeta$ be those as in Example \ref{eg1} and let $A$ be the square matrix of order $4$ defined by: 
\begin{align*}
A := 
\left[ \begin{array}{cccc} 
\sqrt[3]{2} & \sqrt[3]{2} \zeta & \sqrt[3]{2} \zeta ^2 & 0 \\
\sqrt[3]{4} & \sqrt[3]{4} \zeta  ^2 & \sqrt[3]{4} \zeta & 0 \\
0 & 0 & 0 & 3 \\
1 & 1 & 1 & 0 \\
\end{array} \right] 
\in GL (4; \overline{\bQ}). 
\end{align*}
Then we obtain the projective transformation $\varphi _A: \bP ^3 _{\overline{\bQ}} \overset{\sim}{\to} \bP ^3 _{\overline{\bQ}}$ associated to $A$ and we see: 
\begin{align*}
\varphi _A^{-1}(S _{\overline{\bQ}}) = \left( w^2(x+y+z)+xyz=0\right) \subseteq \bP ^3_{\overline{\bQ}} = \Proj (\overline{\bQ} [x,y,z,w]). 
\end{align*}
It is easily to see that $\varphi _A^{-1}(S _{\overline{\bQ}})$ has exactly three singular points $[1\!:\!0\!:\!0\!:\!0], [0\!:\!1\!:\!0\!:\!0], [0\!:\!0\!:\!1\!:\!0] \in \bP ^3_{\overline{\bQ}}$, which are of type $A_1$. 
\end{rem}
\begin{eg}\label{eg2}
Let $S$ be the complete intersection of two quadrics over $\bR$ in $\bP ^4_{\bR}$ as follows: 
\begin{align*}
S := \left( x^2+y^2+wv = zw + wv + vz = 0\right) \subseteq \bP ^4_{\bR} = \Proj (\bR [x,y,z,w,v]). 
\end{align*}
Then $S$ is a Du Val del Pezzo surface of degree $4$ such that $S _{\bC}$ has exactly three singular points $p_{\pm} := [1\!:\! \pm \sqrt{-1} \!:\!0\!:\!0\!:\!0]$ and $p := [0\!:\!0\!:\!1\!:\!0\!:\!0]$ in $\bP ^4_{\bC}$, which are of type $A_1$. 
Since $p_+$ and $p_-$ lie in the same ${\rm Gal}(\bC /\bR )$-orbit, we see $\rho _{\bR}(S)=1$ (see also Table \ref{list(1-1)}). 
Hence, $S$ contains a cylinder if and only if $p$ is of type $A_1^+$ on $S$, by Theorem \ref{main(1-2)}. 
However, $p \in S$ is actually of type $A_1^{++}$, that is, $S$ does not contain any cylinder. 
Indeed, the exceptional set by the minimal resolution at $p$ does not have any $\bR$-rational point since it  can be written locally as follows: 
\begin{align*}
(u^2+v^2 + 1 = 0) \subseteq \bA ^2_{\bR} = \Spec (\bR [u,v])
\end{align*}
for some two parameters $u$ and $v$. 
Meanwhile, this example can not be constructed if the base field of $S$ is a $C_1$-field (see Example \ref{eg4}). 
\end{eg}
\begin{eg}\label{eg3}
Let $S$ be the del Pezzo surface over $k$ of degree $2$ defined by: 
\begin{align*}
S:=\left(\lambda w^2+x^2y^2+xz^3 = 0 \right) \subseteq \bP (1,1,1,2) = \Proj (k[x,y,z,w]), 
\end{align*}
where $\lambda \in k \backslash \{0 \}$. 
Then $S_{\kc}$ has exactly two singular points $p_1 := [1\!:\!0\!:\!0\!:\!0],\ p_2 := [0\!:\!1\!:\!0\!:\!0] \in \bP (1,1,1,2)$, which are $k$-rational and of type $A_2$ and $(A_5)'$, respectively. 
Namely, $\rho _k(S)=1$ (see also Table \ref{list(3)}). 
Let $\sigma : \wS \to S$ be the minimal resolution over $k$. 
By Example \ref{ex of prop(4-2)}, we see that $\wS _{\kc}$ contains reduced curves, whose union is defined over $k$, corresponding to the following dual graph, where ``$\circ$" and ``$\bullet$" mean a $(-2)$-curve and a $(-1)$-curve on $\wS _{\kc}$, respectively: 
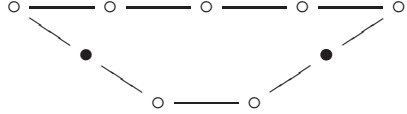
\begin{align*}
\xygraph{
\bullet ([]!{+(-.5,0)} {}) (
        - []!{+(-.75,.5)} \circ ([]!{+(-.3,+.3)} {}) -[r] \circ ([]!{+(-.3,+.3)} {}) -[r] \circ ([]!{+(-.3,+.3)} {}) -[r] \circ ([]!{+(0,+.3)} {}) -[r] \circ ([]!{+(+.3,+.3)} {})- []!{+(-.75,-.5)} \bullet ([]!{+(+.5,0)} {})) 
        - []!{+(.75,-.5)} \circ ([]!{+(-.3,-.3)} {}) -[r] \circ ([]!{+(0,+.3)} {}) - []!{+(.75,.5)} \circ ([]!{+(0,+.3)} {})
)}
\end{align*}
By Theorem \ref{main(1-3)}(4) combined with the above dual graph, $S$ contains a cylinder if and only if $p_1$ is of type $A_2^-$ on $S$. 
By easy computation, we see that the exceptional curve by the minimal resolution at $p_1 \in S$ can be written locally as follows: 
\begin{align*}
M := (\lambda u^2+v^2 = 0) \subseteq \bA ^2_k = \Spec (k[u,v]). 
\end{align*}
for some two parameters $u$ and $v$. 
Note that $p_1 \in S$ is of type $A_2^-$ if and only if $M$ is reducible. 
Therefore, $S$ contains a cylinder if and only if the element $\sqrt{-\lambda} \in \kc$ with $(\sqrt{-\lambda})^2=-\lambda$ is included in $k$. 
\end{eg}
In what follows, we treat three examples of canonical del Pezzo fibrations defined over $\bC$. 
\begin{eg}\label{eg4}
Let $f:X \to Y$ be a canonical del Pezzo fibration of degree $3$ or $4$ over a curve $Y$ and let $X_{\eta}$ be the generic fiber of $f$. 
For simplicity, we put $S := X_{\eta}$ and $k:= \bC (Y)$. 
Assuming that $S_{\kc}$ has a singular point $x$ of type $A_1$, which is $k$-rational, and let $\sigma : \widetilde{S} \to S$ be the minimal resolution at $x$. 
Since $x$ is defined over $k$, so is the exceptional curve $E := \sigma ^{-1}(x)$. 
Note that $E_{\kc}$ is a $(-2)$-curve. 
Now, we see that $E$ has a $k$-rational point since $k=\bC (Y)$ is a $C_1$-field by the Tsen's theorem. 
In other words, the singular point $x$ is always of type $A_1^+$ on $S$ (compare Example \ref{eg2}). 
Therefore, by Theorem \ref{main(1-2)} combined with the above observation, we obtain that $f$ admits a vertical cylinder if and only if $X_{\eta ,\overline{\bC (Y)}}$ allows a singular point defined over $\bC (Y)$. 
\end{eg}
\begin{eg}\label{eg5}
Note that the classification of Du Val del Pezzo surfaces of Picard rank one over $\bC$ is well-known, in particular, the degree of a Du Val del Pezzo surface of Picard rank one with a singular point over $\bC$ is $1,\dots ,6$ or $8$ (see, e.g., {\cite{MZ88}}). 
Let $S$ be a Du Val del Pezzo surface of Picard rank one with degree $d \in \{ 1,\dots ,6,8\}$ over $\bC$ such that ${\rm Sing}(S) \not= \emptyset$, let $Y$ be an algebraic variety over $\bC$ and let $X$ be the direct product $S \times Y$. 
Then the second projection $f:X \to Y$ is a canonical del Pezzo fibration of degree $d$. 
Let $X_{\eta}$ be the generic fiber of $f$. 
For simplicity, put $k := \bC (Y)$. 
Then all $(-1)$-curves and $(-2)$-curves on $X_{\eta ,\kc}$ are defined over $k$. 
Therefore, $f$ does not admit any vertical cylinder if and only if $d=1$ and $X_{\eta ,\kc}$ allows only singular points of types $A_1$, $A_2$, $A_3$ or $D_4$ by Theorems \ref{main(1-1)}, \ref{main(1-2)} and \ref{main(1-3)}. 
This condition is actually equivalent to the condition that $S$ does not contain a cylinder (see {\cite[Theorem 1.6]{Bel17}}). 
\end{eg}
\begin{eg}\label{eg6}
Let $\sO$ be a discrete valuation ring of the rational function field $\bC (t)$ such that the maximal ideal of $\sO$ is generated by $t$, and let $X$ be the $3$-fold variety defined by: 
\begin{align*}
X := ( t^nw^2+x^2y^2+xz^3 = 0) \subseteq \bP _{\sO}(1,1,1,2) = \Proj (\sO [x,y,z,w]), 
\end{align*}
where $n \in \bZ$. 
Then we obtain the structure morphism $f: X \to \Spec (\sO )$. 
Letting $\eta$ be the generic point on $\Spec (\sO )$, the generic fiber $X_{\eta}$ of $f$ can be written as follows: 
\begin{align*}
X_{\eta} = ( t^nw^2+x^2y^2+xz^3 = 0) \subseteq \bP _{\bC (t)}(1,1,1,2) = \Proj (\bC (t)[x,y,z,w]). 
\end{align*}
By Example \ref{eg3}, $X_{\eta}$ is a Du Val del Pezzo surface over $\bC (Y)$ with $\rho _{\bC (t)}(X_{\eta})=1$ and of degree $2$, moreover, $X_{\eta}$ contains a cylinder if and only if $\sqrt{-t^n} \in \bC (t)$. 
Hence, $f$ is a canonical del Pezzo fibration of degree $2$, furthermore, $f$ admits a vertical cylinder if and only if $n$ is even. 
\end{eg}



\begin{thebibliography}{99}
\bibitem{Bel17} G. Belousov, {\em  Cylinders in del Pezzo surfaces with du Val singularities}, Bulletin of the Korean Mathematical Society, {\bf 54} (2017), 1655--1667. 
\bibitem{BCHM10} C. Birkar, P. Cascini, C. Hacon and J. McKernan, {\em Existence of minimal models for varieties of log general type}, Journal of the American Mathematical Society, {\bf 23} (2010), 405--468. 
\bibitem{BW79} J. W. Bruce and C. T. C. Wall, {\em On the classification of cubic surfaces}, Journal of the London Mathematical Society. Second Series, {\bf 19} (1979), 245--256. 
\bibitem{CPPZ21} I. Cheltsov, J. Park, Y. Prokhorov and M. Zaidenberg, {\em Cylinders in Fano varieties}, The European Mathematical Society Surveys in Mathematical Sciences, {\bf 8} (2021), 39--105. 
\bibitem{CPW16a} I. Cheltsov, J. Park and J. Won, {\em Affine cones over smooth cubic surfaces}, Journal of European Mathematical Society, {\bf 18} (2016), 1537--1564. 
\bibitem{CPW16b} I. Cheltsov, J. Park and J. Won, {\em  Cylinders in singular del Pezzo surfaces}, Compositio Mathematica, {\bf 152} (2016), 1198--1224. 
\bibitem{CT88} D. F. Coray and M. A. Tsfasman, {\em Arithmetic on singular Del Pezzo surfaces}, Proceedings of the London
Mathematical Society. Third Series, {\bf 57} (1988), 25--87. 
\bibitem{Cor00} A. Corti, {\em Singularities of linear systems and 3-fold birational geometry}, In:\ {\em Explicit Birational Geometry of 3-folds}, London Mathematical Society Lecture Notes Series, Vol. 281, Cambridge University Press, Cambridge, 2000, 259--312. 
\bibitem{Dol12} I. V. Dolgachev, {\em Classical Algebraic Geometry: A Modern View}, Cambridge University Press, Cambridge, 2012. 
\bibitem{DK18} A. Dubouloz and T. Kishimoto, {\em Cylinders in del Pezzo fibrations}, Israel Journal of Mathematics, {\bf 225} (2018), 797--815. 
\bibitem{DK19} A. Dubouloz and T. Kishimoto, {\em Deformations of $\mathbb{A}^1$-cylindrical varieties}, Mathematische Annalen, {\bf 373} (2019), 1135--1149. 
\bibitem{Dur79} A. Durfee, {\em Fifteen characterizations of rational double points and simple critical points}, L'Enseignement Math\'{e}matique, {\bf 25} (1979), 131--163. 
\bibitem{Fur93a} M. Furushima, {\em The complete classification of compactifications of $\mathbb{C}^3$ which are projective manifolds with the second Betti number one}, Mathematische Annalen, {\bf 297} (1993), 627--662. 
\bibitem{Fur93b} M. Furushima, {\em A new example of a compactification of $\mathbb{C}^3$}, Mathematische Zeitschrift, {\bf 212} (1993), 395--399. 
\bibitem{KPZ11} T. Kishimoto, Y. Prokhorov and M. Zaidenberg, {\em Group actions on affine cones}, In: {\em Affine Algebraic Geometry}, CRM Proceedings and Lecture Notes, Vol. 54, American Mathematical Society, Providence, RI, 2011, 123--163. 
\bibitem{KPZ13} T. Kishimoto, Y. Prokhorov and M. Zaidenberg, {\em $\mathbb{G}_a$-actions on affine cones}, Transformation Groups, {\bf 18} (2013), 1137--1153.  
\bibitem{KPZ14a} T. Kishimoto, Y. Prokhorov and M. Zaidenberg, {\em Affine cones over Fano threefolds and additive group actions}, Osaka Journal of Mathematics, {\bf 51} (2014), 1093--1112. 
\bibitem{KPZ14b} T. Kishimoto, Y. Prokhorov and M. Zaidenberg, {\em Unipotent group actions on del Pezzo cones}, Algebraic Geometry, {\bf 1} (2014), 46--56. 
\bibitem{Koj02} H. Kojima, {\em Algebraic compactifications of some affine surfaces}, Algebra Colloquium, {\bf 9} (2002), 417--425. 
\bibitem{Kol99} J. Koll\'ar, {\em Real algebraic threefolds III. Conic bundles}, Journal of Mathematical Sciences, {\bf 94} (1999), 996--1020.  
\bibitem{Miy94} M. Miyanishi, {\em Algebraic Geometry}, Translations of Mathematical Monographs, Vol. 136, American Mathematical Society, Providence, RI, 1994. 
\bibitem{MZ88} M. Miyanishi and D.-Q. Zhang, {\em Gorenstein log del Pezzo surfaces of rank one}, Journal of Algebra, {\bf 118} (1988), 63--84. 
\bibitem{Mum61} D. Mumford, {\em The topology of normal singularities of an algebraic surface and a criterion for simplicity}, Publications Math\'{e}matiques. Institut de Hautes \'{E}tudes Scientifiques, {\bf 9} (1961), 5--22. 
\bibitem{Poo17} B. Poonen, {\em Rational Points on Varieties}, Graduate Studies in Mathematics, Vol. 186, American Mathematical Society, Providence, RI, 2017. 
\bibitem{PZ16} Y. Prokhorov and M. Zaidenberg, {\em Examples of cylindrical Fano fourfolds}, European Journal of Mathematics, {\bf 2} (2016), 262--282. 
\bibitem{PZ17} Y. Prokhorov and M. Zaidenberg, {\em New examples of cylindrical Fano fourfolds}, In: {\em Algebraic Varieties and Automorphism Groups}, Advanced Studied in Pure Mathematics, Vol. 75, Mathematical Society of Japan, Tokyo, 2017, 443--463. 
\bibitem{PZ18} Y. Prokhorov and M. Zaidenberg, {\em Fano-Mukai fourfolds of genus 10 as compactifications of $\mathbb{C}^4$}, European Journal of Mathematics, {\bf 4} (2018), 1197--1263. 
\bibitem{Saw} M. Sawahara, {\em Cylinders in weak del Pezzo fibrations}, Transformation Groups (to appear), \href{https://arxiv.org/abs/1912.09016}{arXiv:1912.09016}. 
\bibitem{Suz79} M. Suzuki, {\em Compactifications of $\mathbb{C} \times \mathbb{C}^{\ast}$ and $(\mathbb{C}^{\ast})^2$}, The Tohoku Mathematical Journal. Second Series, {\bf 31} (1979), 453--468. 
\bibitem{Ura83} T. Urabe, {\em On singularities on degenerate del Pezzo surfaces of degree 1, 2}, In: {\em Singularities, Part 2 (Arcata, California, 1981)}, Proceedings of Symposia in Pure Mathematics, Vol. 40, American Mathematical Society, Providence, RI, 1983, 587–591. 
\end{thebibliography}
\end{document}